\documentclass[10pt]{article}

\usepackage[letterpaper,margin=1in]{geometry}

\usepackage[T1]{fontenc}
\usepackage{lmodern} % font option #2
\usepackage{microtype}
\usepackage{mathtools,amssymb,amsthm,mathrsfs}
\usepackage[noadjust]{cite}
\usepackage{graphicx}
\usepackage{comment}
\usepackage[shortlabels]{enumitem}
\usepackage{caption}
\usepackage{xcolor}
\usepackage{tikz}
\usetikzlibrary{decorations.markings}
\usepackage[bookmarks=true,unicode,hidelinks]{hyperref}

\numberwithin{equation}{section}

\newtheorem{thm}{Theorem}[section]
\newtheorem{prop}[thm]{Proposition}
\newtheorem{lem}[thm]{Lemma}
\newtheorem{cor}[thm]{Corollary}

\theoremstyle{definition}
\newtheorem{definition}[thm]{Definition}

\theoremstyle{remark}

\newcommand{\beq}{\begin{equation}}
\newcommand{\eeq}{\end{equation}}
\newcommand{\beqq}{\begin{equation*}}
\newcommand{\eeqq}{\end{equation*}}

\DeclareMathOperator{\re}{Re}

\newcommand{\be}{\beta}
\newcommand{\ga}{\gamma}
\newcommand{\om}{\omega}
\newcommand{\tn}[1]{\textnormal{#1}}

\newcommand{\mr}{\mathrm}
\newcommand{\mc}{\mathcal}
\newcommand{\mb}{\mathbf}
\newcommand{\bs}{\boldsymbol} %\newcommand{\bs}{\mathbf}

\newcommand{\N}{{\mathbb N}}
\newcommand{\R}{{\mathbb R}}
\newcommand{\C}{{\mathbb C}}

\newcommand{\LPP}{\mathcal{L}}
\newcommand{\lv}{\ell}
\newcommand{\prob}{\mathbb{P}}

\DeclareMathOperator{\TW}{TW}
\def \ii {\mathrm{i}}
\newcommand{\dd}{\mathrm{d}}

\newcommand{\ab}{\mathsf{\sigma}}	%{\mathsf{r}}

\newcommand{\cd}{\mathsf{c}}	
\newcommand{\bn}{\mb n}
\newcommand{\bu}{\mb u}
\newcommand{\bv}{\mb v}
\newcommand{\bsigma}{{\bs \sigma}}
\newcommand{\btau}{{\bs \tau}}

\newcommand{\listset}{\mc S}
\newcommand{\listn}{\listset_{\bn}}

\newcommand{\mcH}{\mr H}
\DeclareMathOperator{\type}{\mathbf{type}} %{\textbf{type}}

\newcommand{\GG}{\mr G}
\newcommand{\KK}{\mr K}
\newcommand{\FF}{{\mathsf{F}}}
\newcommand{\bM}{\mathbf{M}}
\newcommand{\bN}{\mathbf{N}}
\newcommand{\bT}{\mathbf{T}}
\newcommand{\bz}{\mathbf{z}}
\newcommand{\bsxi}{\bs{\xi}}
\newcommand{\bseta}{\bs{\eta}}

\newcommand{\slope}{\mr m}

\newcommand{\bfr}{\mathbf{r}}
\newcommand{\bfs}{\mathbf{s}}

\newcommand{\bxo}{\bs{\xi}^1}
\newcommand{\beo}{\bs{\eta}^1}
\newcommand{\bxt}{\bs{\xi}^2}
\newcommand{\bet}{\bs{\eta}^2}
\newcommand{\bxot}{\bs{\xi}^{12}}
\newcommand{\beot}{\bs{\eta}^{12}}

\newcommand{\dc}{\mathsf{A}}

\newcommand{\ff}{{\mathsf{f}}}

\newcommand{\con}{\Sigma}
\newcommand{\HH}{\mr H}
\newcommand{\EE}{\mr E}

\newcommand{\cp}{z_c}
\newcommand{\J}{\mathsf{J}}

\newcommand{\mv}{\mathsf{h}} %{h}

\newcommand{\rr}{\mathrm{r}}

\newcommand{\sdev}{\mathsf{g}}
\newcommand{\pp}{\mathsf{z}_{\text{c}}} %{\mathrm{z}_{\text{c}}}

\newcommand{\rL}{\mathrm{L}}
\newcommand{\rR}{\mathrm{R}}

\DeclareSymbolFont{eulergreek}{U}{eur}{m}{n}
\DeclareMathSymbol{\nPi}{\mathord}{eulergreek}{"05}
\newcommand{\K}{\mathsf{K}} % {\mathrm{K}}

\newcommand{\QQ}{\mathrm{Q}} % Upper-tail quantity; BCT uses sans serif for equality conditioning
\newcommand{\DD}{\mathrm{D}}

\newcommand{\ac}{\mathsf{a}}
\newcommand{\bc}{\mathsf{b}}
\newcommand{\discr}{\mathsf{D}}

\newcommand{\ww}{\mathrm{w}}

\newcommand{\cont}{\Gamma}
\newcommand{\Thetao}{\Theta_1}
\newcommand{\Thetat}{\Theta_2}
\newcommand{\Thetar}{\Theta_3}
\newcommand{\ffz}{\mathtt{f}_L^c}%{\mathbf{f}_{L}} %% \ff with a, b, \lv

\newcommand{\tz}{\mathtt{w}} %% critical points of \Theta_i
\newcommand{\tzo}{\tz_1}
\newcommand{\tzt}{\tz_2}
\newcommand{\tzr}{\tz_3}
\newcommand{\tzra}{\mathtt{r}} %% one of two z_3 critical points 
\newcommand{\tzrb}{\mathtt{s}}  %% one of two z_3 critical points 
\newcommand{\tp}{\mathtt{p}} %% p temporary 
\newcommand{\tpo}{\tp_1}
\newcommand{\tpt}{\tp_2}
\newcommand{\tpr}{\tp_3}
\newcommand{\tQ}{\mathtt{Q}} %% p temporary 

\newcommand{\tTo}{\mathtt{T}_{L}}

\newcommand{\tP}{\QQ_{1,L}}
\newcommand{\tDel}{\mathtt{\Delta}_1}
\newcommand{\tDelt}{\mathtt{\Delta}_2}
\newcommand{\tDelr}{\mathtt{\Delta}_3}

\newcommand{\rz}{\mathrm{z}} %% critical points of the reordered G-phases 
\newcommand{\rzo}{\rz_1}
\newcommand{\rzt}{\rz_2}
\newcommand{\rzot}{\rz_{12}}

\newcommand{\Gone}{\mr G_1}
\newcommand{\Gtwo}{\mr G_2}
\newcommand{\Gonetwo}{\mr G_{12}}

\newcommand{\rS}{\mathrm{S}} 

\newcommand{\hide}[1]{}

\newcommand{\mcG}{\mr G}

\title{One-point fluctuations for exponential last passage percolation under upper-tail conditioning}

\author{Jinho Baik\footnote{Department of Mathematics, University of Michigan,
Ann Arbor, MI, 48109, USA, \texttt{baik@umich.edu}} 
\and Tejaswi Tripathi\footnote{Department of Mathematics, University of Kansas,
Lawrence, KS, 66046, USA, \texttt{tejaswit@ku.edu}}}

\date{\today}

\begin{document}

\maketitle

\begin{abstract}
We study exponential directed last passage percolation conditioned on the last passage time to a specified macroscopic point being atypically large. We determine the one-point fluctuations throughout two of the three spatial regions arising under this conditioning, as well as on the boundaries between these regions, extending the work of Baik-Cordaro-Tripathi. Depending on the location of the observation point, the limiting fluctuations are governed by the GUE Tracy-Widom distribution, a Gaussian distribution, or the one-spike BBP distribution on the boundaries between these regions. Through the correspondence with tandem queues, our results also describe how an atypically late departure at one customer-station pair affects departure epochs elsewhere in the network.
\end{abstract}

%\tableofcontents

%%%%%%%%%%%%%%%%%%%%%%%%
%%%%%%%%%%%%%%%%%%%%%%%%
\section{Introduction and main results} \label{sec:intro}

%%%%%%%%%%%%%%%%%%%%%%
\subsection{Background}

Under typical scaling, models in the Kardar-Parisi-Zhang (KPZ) universality class are expected to exhibit universal behavior described by the KPZ fixed point. Atypical events, particularly those in which the height is much larger or smaller than its typical value, are also of considerable interest. Large deviation principles for the KPZ fixed point and related models have been studied extensively; see, for example, \cite{DDV24} and the references therein. One-point upper- and lower-tail large deviation principles were established for the longest increasing subsequence in \cite{LoganShepp77,Sep98,DZ99} and for exponential directed last passage percolation (LPP) in \cite{Johansson00}.

At the logarithmic scale, multi-point, process-level, and metric-level upper-tail large deviation principles have been studied for several models in the KPZ universality class. Hydrodynamic large deviations for the totally asymmetric simple exclusion process were studied by Jensen and Varadhan and, more recently, by Quastel-Tsai \cite{Jensen,Varadhan04,QT25}. Multi-point upper-tail large deviations and the associated conditional limit shapes were obtained for the KPZ equation in \cite{GLLT23,LinTsai25}. At the metric level, an upper-tail large deviation principle was established for the directed landscape in \cite{DDV24}, with related marginal variational problems and conditional limit shapes studied in \cite{DasTsai24}. Upper-tail large deviation principles for last passage percolation models were obtained in \cite{Agarwal25}.

At a finer scale, sharp upper-tail asymptotics have been used to study conditional fluctuations. For the KPZ equation, such estimates and the associated conditional limit shapes were obtained in \cite{GH22}. Conditional fluctuation limits were established for the KPZ fixed point %at times before, after, and near the conditioned high point 
in \cite{LW24,NZ24,LiuZhang25}, for the periodic KPZ fixed point in \cite{BL24}, and for exponential LPP in \cite{Baik-Cordaro-Tripathi25}. Related conditional upper-tail results for geodesics and polymer paths include the diffusive transversal scale for exponential LPP in \cite{BasuGanguly19}, Gaussian fluctuations and rigidity for the directed-landscape geodesic in \cite{Liu22c}, the Brownian-bridge limit in \cite{GHZ23}, and the finer endpoint-scale fluctuations studied in \cite{LiuMaTripathi25}.

In this paper, we study exponential directed last passage percolation conditioned on the last passage time at one site being unusually large. Equivalently, we obtain sharp two-point upper-tail asymptotics relative to the one-point conditioning probability. Our work extends the analysis of exponential LPP in \cite{Baik-Cordaro-Tripathi25} to a broader class of observation points, including points on the transition boundaries between regions with different fluctuation behaviors.

%%%%%%%%%%%%%%%%%%%%%%
\subsection{Exponential LPP}

Directed last passage percolation is an energy-maximization model in a random environment that may be viewed as the zero-temperature limit of a directed polymer model. We consider the point-to-point model with independent exponential weights. Exponential LPP is equivalent to several fundamental models, including the corner growth model with wedge initial condition, the continuous-time totally asymmetric simple exclusion process with step initial condition, and tandem queues.

The model is defined as follows.\footnote{In this paper, $\N$ denotes the set of positive integers. Also, $\R_+=(0,\infty)$.}
For $\mb p=(p_1,p_2)$ and $\mb q=(q_1,q_2)$ in $\N^2$ with $p_1\le q_1$ and $p_2\le q_2$, an up/right path from $\mb p$ to $\mb q$ is a sequence $\pi=(\mb v_i)_{i=1}^r$, where $r=q_1+q_2-p_1-p_2+1$, such that $\mb v_1=\mb p$, $\mb v_r=\mb q$, and $\mb v_{i+1}-\mb v_i\in\{(1,0),(0,1)\}$ for all $i$.

\begin{definition}[Exponential LPP]\label{def:explpp}
Let $\{\om_{\mb v}:\mb v\in\N^2\}$ be a collection of i.i.d.\ exponential random variables with mean $1$. The last passage time from $\mb p$ to $\mb q$ is
\beqq
    \LPP_{\mb p}(\mb q)=\max_{\pi:\mb p\to\mb q}E(\pi),
    \qquad
    E(\pi)=\sum_{\mb v\in\pi}\om_{\mb v},
\eeqq
where the maximum is over all up/right paths from $\mb p$ to $\mb q$. When $\mb p=(1,1)$, we write $\LPP_{(1,1)}(\mb q)=\LPP(\mb q)$. The random field $\LPP=\{\LPP(\mb q):\mb q\in\N^2\}$ is called exponential directed last passage percolation, or simply exponential LPP. For $(\alpha,\beta)\in\R_+^2$, we also set\footnote{The notation $\lceil\alpha\rceil$ denotes the smallest integer greater than or equal to $\alpha$.}
\beq\label{eq:LPPforreal}
    \LPP(\alpha,\beta):=\LPP(\lceil\alpha\rceil,\lceil\beta\rceil).
\eeq
\end{definition}

Under typical scaling, Rost \cite{Ros81} established the law of large numbers, while Johansson \cite{Johansson00} obtained the one-point fluctuation limit. For $x,y>0$, 
\beq\label{eq:onepoint}
    \frac{\LPP(xN,yN)}{N}\xrightarrow{a.s.}\bar{\LPP}(x,y):=(\sqrt{x}+\sqrt{y})^2,\qquad \frac{\LPP(xN,yN)-\bar{\LPP}(x,y)N}{(xy)^{-1/6}(\sqrt{x}+\sqrt{y})^{4/3}N^{1/3}}\xrightarrow{d}\TW_2,
\eeq
as $N\to\infty$, where $\TW_2$ has the GUE Tracy-Widom distribution.

Johansson \cite{Johansson00} also established the one-point upper large deviation principle. For $\ac,\bc>0$ and $\lv>\bar{\LPP}(\ac,\bc)$,
\beq\label{eq:conditioning_tail_asymptotic}
\lim_{N\to\infty}\frac1N\log \prob\left(\LPP(\ac N,\bc N)>\lv N\right)=-\sqrt{\discr}-\ac\log\left(\frac{\lv+\ac-\bc-\sqrt{\discr}}{\lv+\ac-\bc+\sqrt{\discr}}\right)-\bc\log\left(\frac{\lv-\ac+\bc-\sqrt{\discr}}{\lv-\ac+\bc+\sqrt{\discr}}\right),
\eeq
where
\beq\label{eq:Ddefn}
\discr:=\lv^2-2(\ac+\bc)\lv+(\ac-\bc)^2.
\eeq

We now fix $\ac,\bc>0$ and $\lv>\bar{\LPP}(\ac,\bc)$ and consider the asymptotic behavior of $\LPP(\ac xN,\bc yN)$ for $(x,y)\in\R_+^2$, conditioned on the event $\LPP(\ac N,\bc N)>\lv N$. The conditional limit shape depends on the location of the observation point $(x,y)$. To describe this dependence, we introduce the following regions; see Figure~\ref{fig:U123}.

\begin{definition}\label{def:regions}
Set
\beq\label{eq:slopedef}
\slope:=\frac{\lv-\ac-\bc+\sqrt{\discr}}{\lv-\ac-\bc-\sqrt{\discr}},
\eeq
where $\discr$ is given in \eqref{eq:Ddefn}. Define the regions 
\beqq
\begin{aligned}
U_1&:=\left\{(x,y)\in(1,\infty)^2:(x-1)/\slope<y-1<\slope(x-1)\right\},\\
U_2&:=\left\{(x,y)\in\R_+^2:x/\slope<y<\slope x\right\}\setminus\left(\overline{U}_1\cup\{(t,t):t\in(0,1)\}\right),\\
U_3&:=\R_+^2\setminus\left(\overline{U}_1\cup\overline{U}_2\right).
\end{aligned}
\eeqq
For $i=2,3$, also set
\beqq
U_i^{<}:=U_i\cap\{(x,y)\in\R_+^2:y<x\},\qquad U_i^{>}:=U_i\cap\{(x,y)\in\R_+^2:y>x\}.
\eeqq
\end{definition}

\begin{figure}[ht]
\centering
\begin{tikzpicture}[scale=1.6]

\def\xmax{2}
\def\ymax{2}
\def\invm{0.41}
\def\xtop{0.82}
\def\yright{1.41}

% U_2^<
\fill[lightgray] (0,0) -- (\xmax,\invm*\xmax) -- (\xmax,\yright) -- (1,1) -- cycle;

% U_2^>
\fill[lightgray] (0,0) -- (1,1) -- (\yright,\ymax) -- (\xtop,\ymax) -- cycle;

% U_1
\fill[gray] (1,1) -- (\xmax,\yright) -- (\xmax,\ymax) -- (\yright,\ymax) -- cycle;

% boundary lines
\draw[thin] (0,0) -- (\xmax,\invm*\xmax);
\draw[thin] (0,0) -- (\xtop,\ymax);
\draw[thin] (0,0) -- (1,1);
\draw[thin] (1,1) -- (\xmax,\yright);
\draw[thin] (1,1) -- (\yright,\ymax);

% axes
\draw[->, thin] (0,0) -- (2.15,0) node[right] {$x$};
\draw[->, thin] (0,0) -- (0,2.15) node[left] {$y$};

% point (1,1)
%\filldraw (1,1) circle (1pt);
\node[below] at (1,0) {$1$};
\node[left] at (0,1) {$1$};

% labels
\node at (1.62,1.62) {$U_1$};
\node at (1.33,0.74) {$U_2^{<}$};
\node at (0.78,1.33) {$U_2^{>}$};
\node at (1.55,0.22) {$U_3^{<}$};
\node at (0.27,1.55) {$U_3^{>}$};

\end{tikzpicture}
\caption{Regions from Definition~\ref{def:regions}.}
\label{fig:U123}
\end{figure}
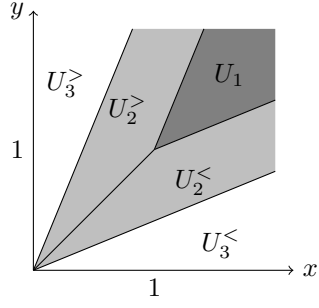

For $0\le x'\le x$ and $0\le y'\le y$, set
$\bar{\LPP}_{(x',y')}(x,y):=(\sqrt{x-x'}+\sqrt{y-y'})^2$. 
The metric-level upper-tail large deviation principle for LPP models \cite{Agarwal25} leads to the following constrained variational problem for the conditional limit shape:
\beqq
    \max_{0\le t\le \min\{1,x,y\}}\left\{t\lv+\bar{\LPP}_{(\ac t,\bc t)}(\ac x,\bc y)\right\}.
\eeqq
See \cite[Section 2.1]{Baik-Cordaro-Tripathi25} for the corresponding heuristic argument. Solving this variational problem gives the conditional limit shape: for every $(x,y)\in\R_+^2$ and every $\epsilon>0$,
\beq\label{eq:CLLN}
    \lim_{N\to\infty}
    \prob\left(
        \left|
            \frac{\LPP(\ac xN,\bc yN)}{N}-\mv(x,y)
        \right|>\epsilon
        \,\bigg|\,
        \LPP(\ac N,\bc N)>\lv N
    \right)
    =0, 
\eeq
where
\beq\label{eq:LLNconjv}
    \mv(x,y):=
    \begin{dcases}
        \lv+\bar{\LPP}_{(\ac,\bc)}(\ac x,\bc y),
        &\text{for }(x,y)\in\overline{U}_1,\\
        \frac12\left[(\lv+\ac-\bc)x+(\lv-\ac+\bc)y-|x-y|\sqrt{\discr}\right],
        &\text{for }(x,y)\in\overline{U}_2,\\
        \bar{\LPP}(\ac x,\bc y),
        &\text{for }(x,y)\in\overline{U}_3.
    \end{dcases}
\eeq
The three formulas agree on their common boundaries and hence $\mv$ is well defined. It is continuous on $\R_+^2$, and is continuously differentiable on $\R_+^2\setminus\{(t,t):0\le t\le1\}$.

%%%%%%%%%%%%%%%%%%%%%%
\subsection{Main results} \label{sec:results}

Conditional fluctuations of exponential LPP were studied in \cite{Baik-Cordaro-Tripathi25} at points in the open set
\beqq
    \left\{(x,y)\in(0,1)^2:x/\slope<y<\slope x\right\}.
\eeqq
That paper obtained limits of conditional multi-point distributions on and near the diagonal, as well as of conditional two-point distributions in the lower and upper triangular subregions.\footnote{The results in \cite{Baik-Cordaro-Tripathi25} were stated under the equality conditioning $\LPP(\ac N,\bc N)=\lv N$. The same analysis extends, with only minor modifications, to the upper-tail conditioning $\LPP(\ac N,\bc N)>\lv N$. We therefore also cite \cite{Baik-Cordaro-Tripathi25} for the corresponding results under the upper-tail conditioning. See Section~\ref{sec:equalcond} for a discussion of the relation between the two conditionings.} It also made heuristic predictions for all $(x,y)\in U_1\cup U_2\cup U_3$.

In the present paper, we prove the one-point predictions of \cite{Baik-Cordaro-Tripathi25} for points in $U_1$ and $U_2$ and obtain the one-point limits on the $U_1/U_2$ and $U_2/U_3$ boundaries, excluding $(1,1)$ and $(0,0)$.
 
Define the positive constants\footnote{In the notation of \cite{Baik-Cordaro-Tripathi25}, $\ab_\pm=\sqrt{2} \ab\cd_\pm$.} 
\beq\label{eq:abdf}
	\ab_\pm:=
	\frac{\sqrt{(\ac+\bc)\lv-(\ac-\bc)^2}\,\discr^{1/4}}{\sqrt{2\ac\bc}} 
	\left(
		1\pm
		\frac{(\ac-\bc)\sqrt{\discr}}{(\ac+\bc)\lv-(\ac-\bc)^2}
	\right)^{1/2}, 
\eeq
and the function
\beq \label{eq:CLTconstant}
	\sdev(x,y):=\begin{dcases}
	\bigl(\sqrt{\ac(x-1)}+\sqrt{\bc(y-1)}\bigr)^{4/3}/(\ac\bc(x-1)(y-1))^{1/6} &\text{for }(x,y)\in\overline{U}_1\cap (1,\infty)^2,\\
	\ab_+ \left( \frac{\slope y-x}{\slope -1}\right)^{1/2} \left( 1- \frac{\slope y-x}{\slope -1}\right)^{1/2} &\text{for }(x,y)\in U_2^{<},\\
	\ab_-\left( \frac{\slope x-y}{\slope-1}\right)^{1/2} \left(1- \frac{\slope x-y}{\slope-1}\right)^{1/2} &\text{for }(x,y)\in U_2^{>},\\
	\bigl(\sqrt{\ac x}+\sqrt{\bc y}\bigr)^{4/3}/(\ac\bc xy)^{1/6} &\text{for }(x,y)\in\overline{U}_3 \cap \R_+^2. 
	\end{dcases}
\eeq
The following is the first main result. Let $\mv(x,y)$ be the function defined in \eqref{eq:LLNconjv}. 

\begin{thm}[Conditional one-point fluctuations] \label{result:CCLT}
For every $(x,y)\in U_1$ and $\rr\in\R$,
\beq\label{eq:GUElimitU1}
	\lim_{N\to\infty}
	\prob\left[
		\frac{\LPP(\ac xN,\bc yN)-\mv(x,y)N}{\sdev(x,y)N^{1/3}}
		\le\rr
		\,\bigg|\,
		\LPP(\ac N,\bc N)>\lv N
	\right]
	=F_{\tn{GUE}}(\rr).
\eeq
For every $(x,y)\in U_2$ and $\rr\in\R$,
\beq\label{eq:GaussianlimitU2}
	\lim_{N\to\infty}
	\prob\left[
		\frac{\LPP(\ac xN,\bc yN)-\mv(x,y)N}{\sdev(x,y)N^{1/2}}
		\le\rr
		\,\bigg|\,
		\LPP(\ac N,\bc N)>\lv N
	\right]
	=\Phi(\rr).
\eeq
\end{thm}

Here, $F_{\tn{GUE}}$ denotes the GUE Tracy-Widom distribution function, and $\Phi$ denotes the cumulative distribution function of the standard normal distribution.

%\medskip

We next consider the boundaries between $U_1$ and $U_2$ and between $U_2$ and $U_3$. To describe the limiting distributions, we recall the one-spike BBP distribution introduced in \cite{BBP05}.

\begin{definition}[One-spike BBP distribution] \label{def:BBPdist}
For $\ww,\rr\in\R$, define the one-spike BBP distribution function by
\beq
\label{eq:BBPFredholmexpnas}
	F_{\tn{BBP},\ww}(\rr)=1+\sum_{n=1}^\infty \frac{(-1)^n}{(n!)^2(2\pi\ii)^{2n}}\int_{(\Sigma_{\rL})^n}\dd\bu\int_{(\Sigma_{\rR})^n}\dd\bv\,\K_n(\bu\mid\bv)^2\prod_{i=1}^n\frac{e^{-u_i^3/3+\rr u_i}(v_i-\ww)}{e^{-v_i^3/3+\rr v_i}(u_i-\ww)}.
\eeq
Here,
$\K_n(\bu\mid\bv)=\det\bigl[\frac{1}{u_i-v_j}\bigr]_{i,j=1}^n$
is the Cauchy determinant. The contours $\Sigma_{\rL}$ and $\Sigma_{\rR}$ are the usual left and right Airy contours, running from $\infty e^{-2\ii\pi/3}$ to $\infty e^{2\ii\pi/3}$ and from $\infty e^{-\ii\pi/3}$ to $\infty e^{\ii\pi/3}$, respectively.\footnote{Throughout the paper, all infinite contours are oriented from bottom to top, and all simple closed contours are oriented counterclockwise.} 
They are disjoint, $\Sigma_{\rR}$ lies to the right of $\Sigma_{\rL}$, and $\Sigma_{\rL}$ lies to the right of $\ww$.
\end{definition}

The one-spike BBP distribution interpolates between the GUE Tracy-Widom and Gaussian distributions: for every $\rr\in\R$,
\beqq
	F_{\tn{BBP},\ww}(\rr)
	\longrightarrow F_{\tn{GUE}}(\rr)
	\quad\text{as }\ww\to-\infty,
	\qquad
	F_{\tn{BBP},\ww}\bigl(\ww^2+\sqrt{2\ww}\,\rr\bigr)
	\longrightarrow\Phi(\rr)
	\quad\text{as }\ww\to\infty.
\eeqq

\begin{thm}[Conditional one-point fluctuations on the boundaries] \label{result:crossdist}
Set
\beq\label{eq:littlec}
	\mathbf v=(v_1,v_2):=(\sqrt{\slope \ac},-\sqrt{\bc}), \qquad 
	\mr c:=2(\slope \ac\bc)^{1/6}(\sqrt{\slope \ac}+\sqrt{\bc})^{-1/3}.
\eeq
\begin{enumerate}
\item[\textnormal{(a)}]
Fix $(x,y)\in(1,\infty)^2$ on the lower boundary between $U_1$ and $U_2$, so that $\frac{y-1}{x-1}=\frac{1}{\slope}$, 
and fix $\ww\in\R$. Set $\delta_N:=\mr c\ww(y-1)^{2/3}N^{2/3}$. Then, for every $\rr\in\R$,
\beqq
	\lim_{N\to\infty}
	\prob\left[
		\frac{\LPP\left(\ac xN+v_1\delta_N,\bc yN+v_2\delta_N\right)-\mv(x,y)N}
		{\sdev(x,y)N^{1/3}}
		\le\rr
		\,\bigg|\,
		\LPP(\ac N,\bc N)>\lv N
	\right]
	=F_{\tn{BBP},\ww}(\rr+\ww^2).
\eeqq

\item[\textnormal{(b)}]
Fix $(x,y)\in(0,\infty)^2$ on the lower boundary between $U_2$ and $U_3$, so that $\frac{y}{x}=\frac{1}{\slope}$, 
and fix $\ww\in\R$. Set $\delta_N:=\mr c\ww y^{2/3}N^{2/3}$. Then, for every $\rr\in\R$,
\beqq
	\lim_{N\to\infty}
	\prob\left[
		\frac{\LPP\left(\ac xN-v_1\delta_N,\bc yN-v_2\delta_N\right)-\mv(x,y)N}
		{\sdev(x,y)N^{1/3}}
		\le\rr
		\,\bigg|\,
		\LPP(\ac N,\bc N)>\lv N
	\right]
	=F_{\tn{BBP},\ww}(\rr+\ww^2).
\eeqq

\item[\textnormal{(c)}]
The corresponding statements on the upper boundaries follow by symmetry under the simultaneous exchanges $\ac\leftrightarrow \bc$ and $x\leftrightarrow y$, with $\mathbf v$ replaced by $(-\sqrt{\ac},\sqrt{\slope \bc})$. 
\end{enumerate}
\end{thm}

The limit \eqref{eq:GaussianlimitU2} for $(x,y)\in U_2\cap(0,1)^2$ was previously proved in \cite{Baik-Cordaro-Tripathi25}. 
On the diagonal $x=y\in(0,1)$, the fluctuations are of order $N^{1/2}$, and the corresponding multi-point distribution limit was obtained in the same paper. 
For points in $U_3$, that paper conjectured that the limit in \eqref{eq:GUElimitU1} holds with the corresponding choices of $\mv(x,y)$ and $\sdev(x,y)$.
The one-point fluctuations in $U_3$ are not analyzed in the present paper for technical reasons and will be considered elsewhere. 
Multi-point fluctuations near the conditioning point $(1,1)$ can also be computed, and their analysis will appear in a separate paper.

All theorems above remain valid with conditioning on $\LPP(\ac N,\bc N)=\lv N$, suitably interpreted, in place of conditioning on $\LPP(\ac N,\bc N)>\lv N$. See Section~\ref{sec:equalcond}.

%%%%%%%%%%%%%%%%%%%%%%
\subsection{Tandem queues}

The conditioning has a natural interpretation in a tandem queueing network. Consider infinitely many single-server stations arranged in series. Customers are served in first-come, first-served order and, after completing service at station $i$, immediately join the queue at station $i+1$. Service times are independent rate-one exponential random variables. Initially, station $1$ has an infinite backlog, while all downstream stations are empty.

Let $\mathcal E(i,j)$ denote the departure epoch of customer $j$ from station $i$, and let $\om_{i,j}$ be the corresponding service time. Then
\beqq
    \mathcal E(i,j)
    =\max\{\mathcal E(i-1,j),\mathcal E(i,j-1)\}+\om_{i,j}.
\eeqq
This is precisely the exponential LPP recursion, and under the natural coupling,
$\mathcal E(i,j)=\LPP(i,j)$. See, for example, \cite{Bar01}.

The conditioning event $\LPP(\ac N,\bc N)>\lv N$ therefore has the following queueing interpretation. Suppose that the departure of a particular customer from a particular station is atypically late. How does conditioning on this event affect the departure epochs of other customers at other stations, both before and after the atypically late departure? The conditional law of large numbers \eqref{eq:CLLN} determines which customer-station pairs undergo a macroscopic change in their departure epochs under this conditioning, while the results of \cite{Baik-Cordaro-Tripathi25} and the present paper describe the corresponding fluctuations.

%%%%%%%%%%%%%%%%%%%%%%
\subsection{Method of proof and outline of the paper}

Our approach builds on \cite{Baik-Cordaro-Tripathi25} and begins with the explicit multi-point distribution formula for exponential LPP obtained in \cite{Liu22a}. The formula involves a contour integral of a Fredholm determinant, and we work with its series expansion. Unlike many Fredholm determinant formulas arising in random matrix theory and in equal-time distributions of KPZ models, the resulting series is not directly amenable to upper large-deviation asymptotics because its integration contours are nested. 
In particular, the contour deformations through the critical points required for the steepest-descent analysis may cross numerous poles of the integrand.

On the diagonal segment $0<x=y<1$, the relevant contour deformations cross no poles. This case was studied for the KPZ fixed point in \cite{LW24,LiuZhang25} and for exponential LPP in \cite{Baik-Cordaro-Tripathi25}. Away from the diagonal, however, the ordering of the relevant critical points depends on the location of $(x,y)$, so the required contour deformations may cross many poles and produce residue contributions that are difficult to track. In \cite{Baik-Cordaro-Tripathi25}, we developed a systematic bookkeeping procedure for these residue integrals and applied it to a restricted subset of $(0,1)^2$. Related residue calculations appear in \cite{NZ24} for the conditional KPZ fixed point in a regime corresponding, in the setting of our model, roughly to $x,y>1$ and $x-y=O(N^{-1/3})$.

In the present paper, we extend this method to a broader class of observation points while restricting attention to one-point conditional distributions. 
The central task is to identify the leading terms in the double Fredholm series indexed by $(n_1,n_2)\in\N^2$. In the regimes treated in \cite{Baik-Cordaro-Tripathi25}, one or two terms typically suffice. In several regimes considered here, however, infinitely many terms contribute at leading order. The most delicate cases are the boundaries $y=x/\slope$ and $y=\slope x$. Depending on the location along these boundaries, the leading terms are indexed either by $(n,n)$ and $(n+1,n)$, or by $(n,1)$, with $n\in\N$. 
See Table~\ref{tab:leadingterms} in Section~\ref{sec:strategy}.

\medskip

This paper is organized as follows. Section~\ref{sec:exact} reviews the explicit multi-point distribution formula for exponential LPP, and Section~\ref{sec:integrals} specializes it to the two-point setting, introduces the contour-integral notation, and develops the contour-rearrangement identities used to track residue contributions. Section~\ref{sec:cp} analyzes the relevant phase functions and critical-point geometry, while Section~\ref{sec:asymgeneral} establishes the Gaussian- and Airy-type steepest-descent estimates and the uniform bounds needed later. Section~\ref{sec:conditioning} derives a sharp asymptotic formula for the conditioning probability, and Section~\ref{sec:strategy} summarizes the common asymptotic strategy and identifies the leading terms in each regime.

The main theorems are proved in Sections~\ref{sec:case1}-\ref{sec:U2U3}, which treat, respectively, the region $U_1$, the boundary between $U_1$ and $U_2$, the region $U_2$, and the boundary between $U_2$ and $U_3$. Finally, Section~\ref{sec:equalcond} discusses the corresponding results under the equality conditioning $\LPP(\ac N,\bc N)=\lv N$ and their relation to the upper-tail conditioning $\LPP(\ac N,\bc N)>\lv N$.

%%%%%%%%%%%%%%%%%%%%%%%%%%%%%%%%%%%%%%%%%%%%%%%%%
\subsubsection*{Acknowledgments}

The work of Baik was supported in part by NSF grant DMS-2246790.

%%%%%%%%%%%%%%%%%%%%%%%%%%%%%%%%
%%%%%%%%%%%%%%%%%%%%%%%%%%%%%%%%
\section{Multi-point distribution formula for exponential LPP} \label{sec:exact}

We review the explicit multi-point distribution formula for exponential LPP. Although we state the formula in the general multi-point setting, only the two-point case will be used in the subsequent sections.

Fix an integer $m\ge1$, and let $\bM=(M_1,\ldots,M_m)\in\N^m$, $\bN=(N_1,\ldots,N_m)\in\N^m$, and $\bT=(T_1,\ldots,T_m)\in\R_+^m$. Assume that $0<T_1\le\cdots\le T_m$ and that $(N_1,T_1),\ldots,(N_m,T_m)$ are all distinct. For each $1\le i\le m$, define
\beqq
	A_i^+:=\{\LPP(M_i,N_i)>T_i\},
	\qquad
	A_i^-:=\{\LPP(M_i,N_i)\le T_i\}.
\eeqq
The multi-point distribution formula for TASEP obtained in~\cite{Liu22a}, together with the standard correspondence between TASEP and exponential LPP, gives explicit formulas for probabilities of the form 
\beqq
	A_1^{\epsilon_1}\cap\cdots\cap A_{m-1}^{\epsilon_{m-1}}\cap A_m^-,
	\qquad
	\epsilon_1,\ldots,\epsilon_{m-1}\in\{+,-\}.
\eeqq
For our purposes, we need a formula for the all-upper-tail event
\beqq
	A_1^+\cap\cdots\cap A_m^+.
\eeqq
A corresponding all-upper-tail formula for the KPZ fixed point was obtained in~\cite[Proposition~3.1]{LiuZhang25}. The same induction argument adapts directly to exponential LPP, and we state the resulting formula below.

For $\bfr=(r_1,\ldots,r_n)$ and $\bfs=(s_1,\ldots,s_n)$ in $\C^n$, define the Cauchy determinant
\beqq
	\K_n(\bfr\mid\bfs):=\det\left[\frac{1}{r_i-s_j}\right]_{i,j=1}^n.
\eeqq
%When the sizes of the vectors are clear, we suppress the subscript $n$ and write $\K(\bfr\mid\bfs)$. 
For $\bn=(n_1,\ldots,n_m)\in\N^m$, define the functions\footnote{Throughout the paper, the empty product is $1$.}
\beq\label{eq:Pi_n}
	\nPi_{\bn}(\bsxi,\bseta):=\K_{n_1}(\bseta^1\mid\bsxi^1)\left[\prod_{i=1}^{m-1}\K_{n_i+n_{i+1}}(\bsxi^i,\bseta^{i+1}\mid\bseta^i,\bsxi^{i+1})\right]\K_{n_m}(\bsxi^m\mid\bseta^m), 
\eeq
and
\beq\label{eq:defFF}
	\FF^{(\bn)}_{\bM,\bN,\bT}(\bsxi,\bseta):=\prod_{i=1}^m\prod_{k_i=1}^{n_i}\frac{\ff_i(\xi_{k_i}^i)}{\ff_i(\eta_{k_i}^i)},
	\qquad
	\ff_i(z):=\frac{z^{N_i-N_{i-1}}e^{(T_i-T_{i-1})z}}{(z+1)^{M_i-M_{i-1}}},
\eeq
for $\bsxi=(\bsxi^1,\ldots,\bsxi^m)$ and $\bseta=(\bseta^1,\ldots,\bseta^m)$, with $\bsxi^i=(\xi_1^i,\ldots,\xi_{n_i}^i)\in \C^{n_i}$ and $\bseta^i=(\eta_1^i,\ldots,\eta_{n_i}^i)\in \C^{n_i}$. We set $M_0=N_0=T_0=0$.

Let
\beqq
	C^{\tn{in}}_{m,\tn{left}},\ldots,C^{\tn{in}}_{2,\tn{left}},C_{1,\tn{left}},C^{\tn{out}}_{2,\tn{left}},\ldots,C^{\tn{out}}_{m,\tn{left}}
\eeqq
be $2m-1$ small circles around $-1$, nested from inside to outside. Similarly, let
\beqq
	C^{\tn{in}}_{m,\tn{right}},\ldots,C^{\tn{in}}_{2,\tn{right}},C_{1,\tn{right}},C^{\tn{out}}_{2,\tn{right}},\ldots,C^{\tn{out}}_{m,\tn{right}}
\eeqq
be $2m-1$ small circles around $0$, also nested from inside to outside and disjoint from the left circles.\footnote{As noted in a footnote in Section~\ref{sec:results}, all closed contours in this paper are oriented counterclockwise, while infinite contours are oriented from bottom to top.} See Figure~\ref{fig:expLPP} for the case $m=3$.

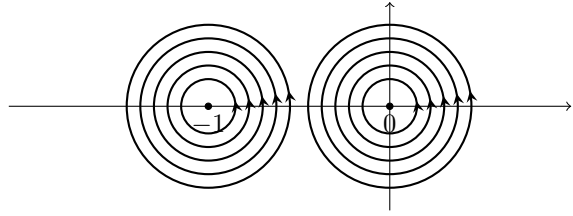
\begin{figure}[h]
\centering
\begin{tikzpicture}[scale=1.2,
    ccw/.style={thick, postaction={decorate},
                decoration={markings,
                            mark=at position 0.03 with {\arrow{stealth}}}}]
	% Draw x-axis and y-axis
	\draw[->] (-4.2,0) -- (2.0,0);
	\draw[->] (0,-1.15) -- (0,1.15);
	% Left center at -1
	\node[below] at (-2,0) {$-1$};
	\filldraw[black] (-2,0) circle (1pt);
	\foreach \r in {0.30,0.45,0.60,0.75,0.90}{
		\draw[ccw] (-2,0) circle [radius=\r];
	}
	% Right center at 0
	\node[below] at (0,0) {$0$};
	\filldraw[black] (0,0) circle (1pt);
	\foreach \r in {0.30,0.45,0.60,0.75,0.90}{
		\draw[ccw] (0,0) circle [radius=\r];
	}
\end{tikzpicture}
\caption{Contours for $m=3$. The five circles on the left, from inside to outside, are $C^{\tn{in}}_{3,\tn{left}}$, $C^{\tn{in}}_{2,\tn{left}}$, $C_{1,\tn{left}}$, $C^{\tn{out}}_{2,\tn{left}}$, and $C^{\tn{out}}_{3,\tn{left}}$. The five circles on the right are, from inside to outside, the corresponding contours with $\tn{left}$ replaced by $\tn{right}$.}
\label{fig:expLPP}
\end{figure}

Define the polynomial $\DD^{(\bn)}_{\bM,\bN,\bT}(\bz)$ in $\bz=(z_1,\ldots,z_{m-1})$ by
\beq\label{eq:D_hat_n}
\begin{split}
	\DD^{(\bn)}_{\bM,\bN,\bT}(\bz)
	&:=\frac{1}{(2\pi\ii)^{2|\bn|}}\prod_{i=2}^m\prod_{k_i=1}^{n_i}
	\left[\int_{C_{i,\tn{left}}^{\tn{in}}}\dd\xi_{k_i}^i+z_{i-1}\int_{C_{i,\tn{left}}^{\tn{out}}}\dd\xi_{k_i}^i\right]
	\left[\int_{C_{i,\tn{right}}^{\tn{in}}}\dd\eta_{k_i}^i+z_{i-1}\int_{C_{i,\tn{right}}^{\tn{out}}}\dd\eta_{k_i}^i\right] \\
	&\quad\times\prod_{k_1=1}^{n_1}\int_{C_{1,\tn{left}}}\dd\xi_{k_1}^1\int_{C_{1,\tn{right}}}\dd\eta_{k_1}^1\,
	\nPi_{\bn}(\bsxi,\bseta)\FF^{(\bn)}_{\bM,\bN,\bT}(\bsxi,\bseta).
\end{split}
\eeq
Here $\bn=(n_1,\ldots,n_m)\in\N^m$ and $|\bn|:=n_1+\cdots+n_m$. For each $1\le i\le m-1$, the degree of $\DD^{(\bn)}_{\bM,\bN,\bT}$ in $z_i$ is at most $2n_{i+1}$. Its coefficients are linear combinations of $2|\bn|$-fold contour integrals. When $m=1$, $\DD^{(n)}_{M,N,T}$ is a constant.

\begin{prop} \label{prop:tail}
Fix an integer $m\ge1$. Let $\bM=(M_1,\ldots,M_m)\in\N^m$, $\bN=(N_1,\ldots,N_m)\in\N^m$, and $\bT=(T_1,\ldots,T_m)\in\R_+^m$. Assume that $0<T_1\le\cdots\le T_m$ and that $(N_1,T_1),\ldots,(N_m,T_m)$ are all distinct. Then
\beq
	\prob\left(\LPP(M_1,N_1)>T_1,\ldots,\LPP(M_m,N_m)>T_m\right)=\QQ_m(\bM,\bN,\bT),
\eeq
where
\beq \label{def:cQQ}
	\QQ_m(\bM,\bN,\bT):=\sum_{\bn\in\N^m}\frac{1}{(\bn!)^2}\QQ^{(\bn)}_m(\bM,\bN,\bT),
\eeq
with $\bn!:=n_1!\cdots n_m!$ for $\bn=(n_1,\ldots,n_m)$, and
\beq \label{def:cQQn}
	\QQ^{(\bn)}_m(\bM,\bN,\bT):=\frac{(-1)^{|\bn|+m}}{(2\pi\ii)^{m-1}}\oint_{>1}\cdots\oint_{>1}\DD^{(\bn)}_{\bM,\bN,\bT}(\bz)\prod_{i=1}^{m-1}\frac{(z_i+1)^{n_i-n_{i+1}-1}}{z_i^{n_{i+1}+1}}\dd z_i.
\eeq
The function $\DD^{(\bn)}_{\bM,\bN,\bT}(\bz)$ is defined in \eqref{eq:D_hat_n}, and the contours are circles centered at the origin with radii greater than $1$.
\end{prop}

\begin{proof}
We adapt the induction argument of~\cite[Proposition~3.1]{LiuZhang25}. For $1\le j\le m$, set
\beqq
	B_j:=\bigcap_{i=1}^j\{\LPP(M_i,N_i)>T_i\}.
\eeqq
Let $\N_0:=\{0\}\cup\N$ and extend the definition of $\DD^{(\bn)}_{\bM,\bN,\bT}$ to $\bn\in\N_0^m$ by interpreting every empty product and every zero-dimensional Cauchy determinant as $1$. Using $\LPP(M,N)>T$ if and only if $x_N(T)<M-N$, where $x_i(t)$ denotes the location of the $i$th particle at time $t$ in TASEP,~\cite[Proposition~2.3]{Liu22a} with $I=\{1,\ldots,m-1\}$ implies
\beq\label{eq:mixed_tail_formula}
\begin{split}
	&\prob\left(B_{m-1}\cap\{\LPP(M_m,N_m)\le T_m\}\right)\\
	&\quad=\frac{1}{(2\pi\ii)^{m-1}}\sum_{\bn\in\N_0^m}\frac{(-1)^{|\bn|+m-1}}{(\bn!)^2}\oint_{>1}\cdots\oint_{>1}\DD^{(\bn)}_{\bM,\bN,\bT}(\bz)\prod_{i=1}^{m-1}\frac{(z_i+1)^{n_i-n_{i+1}-1}}{z_i^{n_{i+1}+1}}\dd z_i.
\end{split}
\eeq

We use two elementary observations. First, every term with $n_i=0$ for some $i\le m-1$ vanishes. Indeed, as a function of $z_i$, $\DD^{(\bn)}_{\bM,\bN,\bT}(\bz)$ is a polynomial of degree at most $2n_{i+1}$. 
Thus, if $n_i=0$, then, with the other variables fixed,
\beqq
	\DD^{(\bn)}_{\bM,\bN,\bT}(\bz)\frac{(z_i+1)^{-n_{i+1}-1}}{z_i^{n_{i+1}+1}}=O(z_i^{-2})
	\qquad\text{as }z_i\to\infty.
\eeqq
Hence, the $z_i$-integral vanishes by Cauchy's theorem. Thus, the sum in \eqref{eq:mixed_tail_formula} may be restricted to $n_1,\ldots,n_{m-1}\ge1$ and $n_m\ge0$.

Second, consider the terms with $n_m=0$. Write $\widehat\bn=(n_1,\ldots,n_{m-1})$ and $\widehat\bM=(M_1,\ldots,M_{m-1})$, and define $\widehat\bN$ and $\widehat\bT$ similarly. Since $\bsxi^m=\bseta^m=\varnothing$ and
\beqq
	\frac{1}{2\pi\ii}\oint_{>1}\frac{(z_{m-1}+1)^{n_{m-1}-1}}{z_{m-1}}\dd z_{m-1}=1,
\eeqq
we find that 
\beqq
	\DD^{(\widehat\bn,0)}_{\bM,\bN,\bT}(z_1,\ldots,z_{m-1})=\DD^{(\widehat\bn)}_{\widehat\bM,\widehat\bN,\widehat\bT}(z_1,\ldots,z_{m-2}),
\eeqq
which is independent of $z_{m-1}$. 
Thus, the contribution of the terms with $n_m=0$ in \eqref{eq:mixed_tail_formula} is precisely $\QQ_{m-1}(\widehat\bM,\widehat\bN,\widehat\bT)$. The terms with $n_m\ge1$ contribute $-\QQ_m(\bM,\bN,\bT)$, and hence
\beq\label{eq:mixed_tail_split}
	\prob\left(B_{m-1}\cap\{\LPP(M_m,N_m)\le T_m\}\right)=\QQ_{m-1}(\widehat\bM,\widehat\bN,\widehat\bT)-\QQ_m(\bM,\bN,\bT).
\eeq

For $m=1$, we use the usual one-point distribution formula:
\beqq
	\prob\left(\LPP(M_1,N_1)\le T_1\right)=1+\sum_{n_1\ge1}\frac{(-1)^{n_1}}{(n_1!)^2}\DD^{(n_1)}_{M_1,N_1,T_1}=1-\QQ_1(M_1,N_1,T_1).
\eeqq
Thus, $\prob(B_1)=\QQ_1$, and the result follows from an induction. 
\end{proof}

%%%%%%%%%%%%%%%%%%%%%%
%%%%%%%%%%%%%%%%%%%%%%
\section{Two-point formulas}
\label{sec:integrals}

We specialize Proposition~\ref{prop:tail} to the two-point case and rewrite the resulting integrals in forms suitable for steepest-descent analysis. We also derive contour-rearrangement identities for the residue contributions.
The analogous three-point formulas were studied in \cite[Section~7]{Baik-Cordaro-Tripathi25}.

%%%%%%%%%%%%%%%%%%%%%%%%%%
\subsection{Integral notation}
\label{subsec:integralnotation}

\begin{definition}\label{def:Bnset}
Define $\mc A_2:=\{1,2,12\}$. For $\bn=(n_1,n_2)\in\N^2$, let $\listn$ denote the set of lists
\beqq
    \bsigma=\sigma_1\sigma_2\cdots\sigma_k,
    \qquad
    \sigma_j\in\mc A_2,
\eeqq
such that, for each $i=1,2$, the total number of occurrences of the symbol $i$ among $\sigma_1,\ldots,\sigma_k$, counting an entry $12$ as one occurrence of both $1$ and $2$, is $n_i$. We write $|\bsigma|=k$.
The type of a list $\bsigma\in\listn$ is
\beqq
    \type(\bsigma)=(a_{12},a_1,a_2)\in\N_0^3,
\eeqq
where, for each $*\in\mc A_2$, $a_*$ denotes the number of entries of $\bsigma$ equal to $*$.
\end{definition}

We usually write a list $\bsigma$ in the form
\beqq
    \bsigma=\alpha_1^{m_1}\alpha_2^{m_2}\alpha_3^{m_3}\cdots,
\eeqq
where $\alpha_i,\alpha_{i+1}\in\mc A_2$ are distinct for every $i$, and $\alpha^m$ denotes a block of $m$ consecutive copies of $\alpha$. To avoid ambiguity, we write the symbol $12$ as $(12)$ when necessary and omit the superscript when $m_i=1$. For example,
\beqq
    11(12)12112
    =1^2(12)^1 1^1 2^1 1^2 2^1
    =1^2(12)1 2 1^2 2,
\eeqq
has type $(1,5,2)$ and belongs to $\mathcal{S}_{(6,3)}$.

Note that if $\type(\bsigma)=(a_{12},a_1,a_2)$, then
\beqq
a_{12}+a_1=n_1,
\qquad
a_{12}+a_2=n_2.
\eeqq

\begin{definition}\label{def:Pisigmaxi}
For $\bn\in\N^2$ and $\bsigma,\btau\in\listn$ with $\type(\bsigma)=(a_{12},a_1,a_2)$ and $\type(\btau)=(b_{12},b_1,b_2)$, 
define the product of Cauchy determinants
\beq\label{eq:Pisigmaxi}
\nPi^{\bsigma}_{\btau}(\bsxi,\bseta):=\K_{n_1}(\beot,\beo\mid\bxot,\bxo)\K_{a_1+b_2}(\bxo,\bet\mid\beo,\bxt)\K_{n_2}(\bxot,\bxt\mid\beot,\bet)
\eeq
and the function\footnote{The notation $\FF^{\bsigma\mid\btau}(\bsxi,\bseta)$ was used in \cite{Baik-Cordaro-Tripathi25}.}
\beq\label{eq:Fsigmatau}
\FF^{\bsigma}_{\btau}(\bsxi,\bseta):=\frac{\prod_{*\in\mc A_2}\prod_{i=1}^{a_*}\ff_*(\xi_i^*)}{\prod_{*\in\mc A_2}\prod_{i=1}^{b_*}\ff_*(\eta_i^*)}.
\eeq
Here $\bsxi=(\bsxi^{12},\bsxi^1,\bsxi^2)$ and $\bseta=(\bseta^{12},\bseta^1,\bseta^2)$, where
$\bsxi^*=(\xi_1^*,\ldots,\xi_{a_*}^*)\in\C^{a_*}$ and
$\bseta^*=(\eta_1^*,\ldots,\eta_{b_*}^*)\in\C^{b_*}$ for each $*\in\mc A_2$.
The functions $\ff_1$ and $\ff_2$ are given in \eqref{eq:defFF}, and we set $\ff_{12}:=\ff_1\ff_2$. Explicitly,
\beq \label{eq:ff1212def}
\ff_1(z)= \frac{z^{N_1}e^{T_1z}}{(z+1)^{M_1}},
\quad
\ff_2(z)= \frac{z^{N_2-N_{1}}e^{(T_2-T_{1})z}}{(z+1)^{M_2-M_{1}}},
\quad
\ff_{12}(z)= \frac{z^{N_2}e^{T_2z}}{(z+1)^{M_2}}.
\eeq
\end{definition}

The middle Cauchy determinant in \eqref{eq:Pisigmaxi} is square because $a_1+b_2=b_1+a_2$, which follows from
$a_{12}+a_1=b_{12}+b_1=n_1$ and $a_{12}+a_2=b_{12}+b_2=n_2$.
The functions $\nPi^{\bsigma}_{\btau}$ and $\FF^{\bsigma}_{\btau}$ depend only on $\type(\bsigma)$ and $\type(\btau)$ and are symmetric separately in the variables $\xi_1^*,\ldots,\xi_{a_*}^*$ and $\eta_1^*,\ldots,\eta_{b_*}^*$ for each $*\in\mc A_2$.

Let $\bsigma=\sigma_1\cdots\sigma_k\in\listn$. For $1\le r\le k$, let $\nu_r$ be the number of occurrences of $\sigma_r$ among $\sigma_1,\ldots,\sigma_r$, and define
\beqq
	\bsxi^{\bsigma}:=(\xi_{\nu_1}^{\sigma_1},\ldots,\xi_{\nu_k}^{\sigma_k}).
\eeqq
Define $\bseta^{\btau}$ analogously. Thus, $\bsxi^{\bsigma}$ and $\bseta^{\btau}$ list the integration variables in the order prescribed by $\bsigma$ and $\btau$.
For example, 
\beqq
	\bsxi^{12(12)1}=(\xi_1^1,\ \xi_1^2,\ \xi_1^{12},\ \xi_2^1), \qquad 
	\bseta^{1(12)^2}= (\eta_1^1,\ \eta_1^{12},\ \eta_2^{12}).
\eeqq

\begin{definition}\label{def:Jintegral}
For $\bsigma,\btau\in\listn$, define
\beq\label{eq:ingdef}
	\J^{\bsigma}_{\btau}:=\frac{1}{(2\pi\ii)^{|\bsigma|+|\btau|}}\int\dd\bsxi^{\bsigma}\int\dd\bseta^{\btau}\,\nPi^{\bsigma}_{\btau}(\bsxi,\bseta)
\FF^{\bsigma}_{\btau}(\bsxi,\bseta).
\eeq
The coordinates of $\bsxi^{\bsigma}$ are integrated over small circles centered at $-1$, nested from inside to outside in the displayed order, and the coordinates of $\bseta^{\btau}$ are integrated over small circles centered at $0$, also nested from inside to outside in the displayed order. All left and right circles are mutually disjoint.
\end{definition}

Although the contours in Definition~\ref{def:Jintegral} are initially taken to be distinct, the integrand has no pole when two variables with the same superscript coincide. Thus, for each $*\in\{1,2,12\}$, all consecutive contours carrying the variables $\xi_i^*$ may be collapsed to the same contour, and likewise for the variables $\eta_i^*$.
For example, 
\beqq
\begin{split}
	\J^{12(12)1}_{1(12)^2}
	&=\frac{1}{(2\pi\ii)^7}
	\int\dd\xi^1_1\int\dd\xi^2_1\int\dd\xi^{12}_1\int\dd\xi^1_2
	\int\dd\eta^1_1\int\dd\eta^{12}_1\int\dd\eta^{12}_2 \\
	&\quad\times\K_3(\eta^{12}_1,\eta^{12}_2,\eta^1_1\mid\xi^{12}_1,\xi^1_1,\xi^1_2)
	\K_2(\xi^1_1,\xi^1_2\mid\eta^1_1,\xi^2_1)
	\K_2(\xi^{12}_1,\xi^2_1\mid\eta^{12}_1,\eta^{12}_2)
	\frac{\ff_1(\xi^1_1)\ff_1(\xi^1_2)\ff_2(\xi^2_1)\ff_{12}(\xi^{12}_1)}
	{\ff_1(\eta^1_1)\ff_{12}(\eta^{12}_1)\ff_{12}(\eta^{12}_2)},
\end{split}
\eeqq
where the contours satisfy $0<|\xi^1_1+1|<|\xi^2_1+1|<|\xi^{12}_1+1|<|\xi^1_2+1|$ and $0<|\eta^1_1|<|\eta^{12}_1|=|\eta^{12}_2|$.

%%%%%%%%%%%%%%%%%%%%%%%%%%
\subsection{Expansion of the two-point probability}
\label{subsec:exp2ptprob}

When $m=2$, we express $\QQ_2^{(n_1,n_2)}(\bM,\bN,\bT)$ in \eqref{def:cQQ} in terms of the integrals above.

\begin{lem}\label{lem:simpleQ}
If $n_1\ge n_2+1$, then
\beq\label{eq:case3_simpleQ_greater}
	\QQ_2^{(n_1,n_2)}=(-1)^{n_1+n_2}
	\sum_{\substack{0\le i,j\le n_2\\2n_2-n_1+1\le i+j\le n_2}}
	\binom{n_2}{i}\binom{n_2}{j}\binom{n_1-n_2-1}{n_2-i-j}
	\J^{2^{n_2-i}1^{n_1}2^i}_{2^{n_2-j}1^{n_1}2^j}.
\eeq
If $n_1\le n_2$, then
\beq\label{eq:case3_simpleQ_less}
	\QQ_2^{(n_1,n_2)}=(-1)^{n_2-1}
	\sum_{\substack{0\le i,j\le n_2\\i+j\ge2n_2-n_1+1}}
	(-1)^{i+j}\binom{n_2}{i}\binom{n_2}{j}
	\binom{i+j-n_2-1}{n_2-n_1}
	\J^{2^{n_2-i}1^{n_1}2^i}_{2^{n_2-j}1^{n_1}2^j}.
\eeq
\end{lem}

\begin{proof}
Expanding \eqref{eq:D_hat_n} for $m=2$ as a polynomial in $z$ and using the symmetry in each set of variables $\xi_1^i,\ldots,\xi_{n_i}^i$ and $\eta_1^i,\ldots,\eta_{n_i}^i$ for $i=1,2$, we find that 
\beqq
	\DD^{(n_1,n_2)}_{\bM,\bN,\bT}(z)
	=\sum_{i,j=0}^{n_2}
	\binom{n_2}{i}\binom{n_2}{j}
	\J^{2^{n_2-i}1^{n_1}2^i}_{2^{n_2-j}1^{n_1}2^j} \, z^{i+j}.
\eeqq

Substituting this into \eqref{def:cQQn}, we obtain
\beqq
	\QQ_2^{(n_1,n_2)}
	=(-1)^{n_1+n_2}
	\sum_{i,j=0}^{n_2}
	\binom{n_2}{i}\binom{n_2}{j}
	\J^{2^{n_2-i}1^{n_1}2^i}_{2^{n_2-j}1^{n_1}2^j} \, C_{n_1,n_2}(i+j),
	\quad
	C_{n_1,n_2}(k)
	:=\frac{1}{2\pi\ii}\oint_{>1}
	\frac{(1+z)^{n_1-n_2-1}}{z^{n_2+1-k}}\,\dd z.
\eeqq
If $n_1\ge n_2+1$, then
\beqq
	C_{n_1,n_2}(k)
	=\binom{n_1-n_2-1}{n_2-k}
	\mathbf 1_{2n_2-n_1+1\le k\le n_2},
\eeqq
while, if $n_1\le n_2$, then
\beqq
	C_{n_1,n_2}(k)
	=(-1)^{k-n_1-1}
	\binom{k-n_2-1}{n_2-n_1}
	\mathbf 1_{k\ge2n_2-n_1+1}.
\eeqq
The result follows.
\end{proof}

\begin{cor}\label{result:Qn1n2bd}
For every $n_1,n_2\ge1$,
\beqq
	\left|\QQ_2^{(n_1,n_2)}\right|
	\le %(n_2+1)^2 8^{n_1+n_2}
	2^{5(n_1+n_2)}
	\max_{\substack{0\le i,j\le n_2\\i+j\ge2n_2-n_1+1}}
	\left|
	\J^{2^{n_2-i}1^{n_1}2^i}_{2^{n_2-j}1^{n_1}2^j}
	\right|.
\eeqq
\end{cor}

\begin{proof}
This follows from Lemma~\ref{lem:simpleQ} and the bounds $\binom{m}{k}\le2^m$ and $n_2+1\le 2^{n_2}$. 
\end{proof}

%%%%%%%%%%%%%%%%%%%%%%%%%
\subsection{Contour-rearrangement identities}
\label{subsec:contourrearrangeiden}

The $\J$-integrals appearing in $\QQ_2^{(n_1,n_2)}$ in Lemma~\ref{lem:simpleQ} inherit the nesting of the original contours. This nesting is not always compatible with the ordering of the critical points needed for steepest-descent analysis. To deform the contours through the relevant critical points, we therefore need to rearrange their nesting. In doing so, contours may cross poles of the integrand and produce residue contributions. The following lemma gives the corresponding contour-rearrangement identities. An analogous result for $\mc A_3=\{1,2,3,12,23,123\}$ was proved in \cite[Corollary 7.23]{Baik-Cordaro-Tripathi25}.

\begin{lem}\label{result:originalintexprss}
Define
\beqq
	\dc_k^{m,n}:=k!\binom{m}{k}\binom{n}{k}.
\eeqq
For a list containing the indicated consecutive blocks, define
\beqq
	\mathrm{R}_k^{12}(\cdots1^m2^n\cdots):=\cdots2^{n-k}(12)^k1^{m-k}\cdots,
	\qquad
	\mathrm{R}_k^{21}(\cdots2^n1^m\cdots):=\cdots1^{m-k}(12)^k2^{n-k}\cdots.
\eeqq

\begin{enumerate}[(a)]
\item
If $\bsigma'$ is obtained from $\bsigma$ by interchanging consecutive blocks $\alpha^m\beta^n$ with $\beta^n\alpha^m$, where $\alpha,\beta\in\mc A_2$ and $\{\alpha,\beta\}\ne\{1,2\}$, then $\J^{\bsigma}_{\btau}=\J^{\bsigma'}_{\btau}$. The analogous statement also holds when two such blocks occur in $\btau$.

\item
Suppose $\type(\btau)=(b_{12},b_1,b_2)$. Then,
\begin{eqnarray}
	&&\J^{\bsigma}_{\btau}
	=\sum_{k=0}^{m\wedge n}(-1)^{k(b_1+1)}\dc_k^{m,n}\J^{\mathrm{R}_k^{12}(\bsigma)}_{\btau},
	\qquad \text{if $\bsigma=\cdots1^m2^n\cdots$,}
	\label{eq:Jup12}\\
	&&\J^{\bsigma}_{\btau}
	=\sum_{k=0}^{m\wedge n}(-1)^{kb_1}\dc_k^{m,n}\J^{\mathrm{R}_k^{21}(\bsigma)}_{\btau},
	\qquad \text{if $\bsigma=\cdots2^n1^m\cdots$.}
	\label{eq:Jup21}
\end{eqnarray}

\item
Suppose $\type(\bsigma)=(a_{12},a_1,a_2)$. Then,
\begin{eqnarray}
	&&\J^{\bsigma}_{\btau}
	=\sum_{k=0}^{m\wedge n}(-1)^{ka_1}\dc_k^{m,n}\J^{\bsigma}_{\mathrm{R}_k^{12}(\btau)},
	\qquad \text{if $\btau=\cdots1^m2^n\cdots$,}
	\label{eq:Jdown12}\\
	&&\J^{\bsigma}_{\btau}
	=\sum_{k=0}^{m\wedge n}(-1)^{k(a_1+1)}\dc_k^{m,n}\J^{\bsigma}_{\mathrm{R}_k^{21}(\btau)},
	\qquad \text{if $\btau=\cdots2^n1^m\cdots$.}
	\label{eq:Jdown21}
\end{eqnarray}
\end{enumerate}
\end{lem}

\begin{proof}
We move the relevant contours one at a time and apply Cauchy's residue theorem. The coefficient $\dc_k^{m,n}$ counts the possible choices and pairings of the $k$ residues.
\end{proof}

Note that since $\binom{m}{k}\le2^m$ and $\binom{n}{k}\le2^n$,
\beq\label{eq:dcest}
	\dc_k^{m,n}\le2^{m+n}k!.
\eeq
We use this estimate repeatedly in later sections. 

Applying the above identities to the $\J$-integrals appearing in Lemma~\ref{lem:simpleQ} gives the following bounds.

\begin{cor} \label{result:Jintrearboundgen}
Let $n_1,n_2\in\N$ and $0\le i,j\le n_2$. Define\footnote{The superscript and subscript indicate whether the $1$- and $2$-blocks in the left and right lists, respectively, are ordered as $12$ or $21$.} 
\beqq
\begin{aligned}
	\J^{12}_{12}(a,b,n_1,n_2)&:=\J^{1^{n_1-a}(12)^a2^{n_2-a}}_{1^{n_1-b}(12)^b2^{n_2-b}},
	&\qquad
	\J^{12}_{21}(a,b,n_1,n_2)&:=\J^{1^{n_1-a}(12)^a2^{n_2-a}}_{2^{n_2-b}(12)^b1^{n_1-b}},\\
	\J^{21}_{12}(a,b,n_1,n_2)&:=\J^{2^{n_2-a}(12)^a1^{n_1-a}}_{1^{n_1-b}(12)^b2^{n_2-b}},
	&\qquad 
	\J^{21}_{21}(a,b,n_1,n_2)&:=\J^{2^{n_2-a}(12)^a1^{n_1-a}}_{2^{n_2-b}(12)^b1^{n_1-b}}.
\end{aligned}
\eeqq
Then
\beqq
	\left| \J^{2^{n_2-i}1^{n_1}2^i}_{2^{n_2-j}1^{n_1}2^j} \right| \le 2^{2(n_1+n_2)} S^{\alpha}_{\beta}
	\qquad
	\text{for each } \alpha,\beta \in\{12,21\},
\eeqq
where
\beqq
\begin{alignedat}{2}
	S^{12}_{12}
	&:=
	\sum_{a=0}^{n_1\wedge(n_2-i)}
	\sum_{b=0}^{n_1\wedge(n_2-j)}
	a!b!\,|\J^{12}_{12}(a,b,n_1,n_2)|,
	\qquad&
	S^{12}_{21}
	&:=
	\sum_{a=0}^{n_1\wedge(n_2-i)}
	\sum_{b=0}^{n_1\wedge j}
	a!b!\,|\J^{12}_{21}(a,b,n_1,n_2)|,\\
	S^{21}_{12}
	&:=
	\sum_{a=0}^{n_1\wedge i}
	\sum_{b=0}^{n_1\wedge(n_2-j)}
	a!b!\,|\J^{21}_{12}(a,b,n_1,n_2)|,
	\qquad&
	S^{21}_{21}
	&:=
	\sum_{a=0}^{n_1\wedge i}
	\sum_{b=0}^{n_1\wedge j}
	a!b!\,|\J^{21}_{21}(a,b,n_1,n_2)|.
\end{alignedat}
\eeqq
\end{cor}

\begin{proof}
Apply either \eqref{eq:Jup12} or \eqref{eq:Jup21} to the superscript and either \eqref{eq:Jdown12} or \eqref{eq:Jdown21} to the subscript, and use \eqref{eq:dcest}.
\end{proof}

We record the following bound that we use in several places: 
For $n_1,n_2\in\N$ and $0\le a,b\le n_1\wedge n_2$, 
\beq\label{eq:triplfactorial}
	a!\,b!(n_1+n_2-a-b)!
	\le2^{n_1+n_2}n_1!n_2!.
\eeq
Indeed, $a!b!(n_1+n_2-a-b)!\le(n_1+n_2)!$, while $\binom{n_1+n_2}{n_1}\le2^{n_1+n_2}$, which gives \eqref{eq:triplfactorial}.

%%%%%%%%%
\begin{comment}
\begin{lem}
Let $C>0$. Then
\begin{equation}
\label{eq:generaln1n2absumest}
\sum_{n_1,n_2\in\N}
\frac{C^{n_1+n_2}}
{(n_1!n_2!)^{3/2}}
\sum_{0\leq a,b\leq n_1\wedge n_2}
a!b!\sqrt{(n_1+n_2-a-b)!}
<\infty.
\end{equation}
\end{lem}

\begin{proof}
Since $k^k\leq e^k k!$ for every $k\in\N$, the sum is bounded by
\begin{equation*}
\sum_{n_1,n_2\in\N}
\frac{(eC)^{n_1+n_2}}{(n_1!n_2!)^{3/2}}
\sum_{0\leq a,b\leq n_1\wedge n_2}
a!b!\sqrt{(n_1+n_2-a-b)!}.
\end{equation*}
Moreover,
\begin{equation*}
a!b!(n_1+n_2-a-b)!\leq(n_1+n_2)!
\leq 2^{n_1+n_2}n_1!n_2!.
\end{equation*}
Since the number of admissible pairs $(a,b)$ is at most $2^{n_1+n_2}$, the sum is bounded by
\begin{equation*}
\sum_{n_1,n_2\in\N}
\frac{(4eC)^{n_1+n_2}}{\sqrt{n_1!n_2!}},
\end{equation*}
which converges. Hence, the result follows.
\end{proof}
\end{comment}
%%%%%%%%

%%%%%%%%%%%%%%%%%%%%%%%%%%%%%%%
%%%%%%%%%%%%%%%%%%%%%%%%%%%%%%%
\section{Critical points}
\label{sec:cp}

The contour-rearrangement identities of Section~\ref{sec:integrals} allow us to place the contours in an order suitable for steepest-descent analysis. To determine the required contour ordering, we analyze the phase functions associated with $\ff_1$, $\ff_2$, and $\ff_{12}$ in \eqref{eq:ff1212def}. In this section, we compute their critical points and determine their ordering. For the $U_2$ regime, it suffices by symmetry to analyze the case $y<x$; the case $y>x$ follows by simultaneously exchanging $\ac\leftrightarrow\bc$ and $x\leftrightarrow y$.

%%%%%%%%%%%%%%%%%%%%%%%%%%%%%%%
\subsection{Phase functions}
\label{sec:phasefunctions}

Let $\ac,\bc,\lv>0$ satisfy $\lv>(\sqrt \ac+\sqrt \bc)^2$. For $(x,y)\in\R_+^2$, let $\mv=\mv(x,y)$ be given by \eqref{eq:LLNconjv}. In particular,
\begin{eqnarray}
	&&\mv=\lv+\bigl(\sqrt{\ac(x-1)}+\sqrt{\bc(y-1)}\bigr)^2,
	\qquad (x,y)\in\overline{U}_1,
	\label{eq:mvu112}\\
	&&\mv=\frac12\bigl[(\lv+\ac-\bc-\sqrt{\discr})x+(\lv-\ac+\bc+\sqrt{\discr})y\bigr],
	\qquad (x,y)\in\overline{U}_2^<.
	\label{eq:mvu212}
\end{eqnarray}
The two expressions agree on the boundary between $U_1$ and $U_2^<$.

For $(x,y)\in \overline{U}_1 \cup \overline{U}_2^<$, define the functions\footnote{Throughout the paper, we take the branch cut of $\log(z+1)$ to be $(-\infty,-1]$ and the branch cut of $\log z$ to be $[0,\infty)$.}
\beq
\label{eq:allphasefunctions}
\begin{split}
	\Thetao(z)&:=-\ac\log(1+z)+\bc\log z+\lv z,\\
	\Thetat(z)&:=-\ac x\log(1+z)+\bc y\log z+\mv z,\\
	\Thetar(z)&:=-\ac(x-1)\log(1+z)+\bc(y-1)\log z+(\mv-\lv)z.
%	\tGf(z):=-(1-x)a\log(1+z)+(1-y)b\log z+(\lv-\mv)z=-\tGr(z).
\end{split}
\eeq
Define the quadratic polynomials
\beq
\begin{split}
	\tpo(z)&:=\lv z^2+(\lv-\ac+\bc)z+\bc,\\
	\tpt(z)&:=\mv z^2+(\mv-\ac x+\bc y)z+\bc y,\\
	\tpr(z)&:=(\mv-\lv)z^2+\bigl(\mv-\lv-\ac(x-1)+\bc(y-1)\bigr)z+\bc(y-1),
\end{split}
\eeq
so that
\beq\label{eq:tGdertp}
	\Theta_i'(z)=\frac{\tp_i(z)}{z(z+1)},
	\qquad i=1,2,3.
\eeq
Note the identities
\beq
\label{eq:tGide}
	\Thetao(z) + \Thetar(z)=\Thetat(z), \qquad \tpo(z)+ \tpr(z)=\tpt(z). 
\eeq

Set
\beq\label{eq:ppthree}
	\pp:=-\frac{\lv-\ac+\bc-\sqrt \discr}{2\lv}
	=-\frac{2\bc}{\lv-\ac+\bc+\sqrt \discr}
	=-\frac{\lv-\ac-\bc-\sqrt \discr}{\lv+\ac-\bc-\sqrt \discr},
\eeq
where the identities follow from $\discr=\lv^2-2(\ac+\bc)\lv+(\ac-\bc)^2$.

%%%%%%%%%%%%%%%%%%%%%%%%%%%%%%%
%\subsection{The conditioning phase}
%\subsection{Critical points of $\tGo$}

\begin{lem}\label{result:cpconditioning}
The roots of $\tpo$ are
\beq\label{eq:cpconditioning}
	(\tzo^-,\tzo^+)
	=\left(-\frac{\lv-\ac+\bc+\sqrt \discr}{2\lv},\pp\right)
	= \left( -\frac{\sqrt{\slope \bc}}{\sqrt \ac+\sqrt{\slope \bc}}, -\frac{\sqrt \bc}{\sqrt{\slope \ac}+\sqrt \bc} \right), 
\eeq
and they satisfy
\beqq
	-1<\tzo^-<\tzo^+<0.
\eeqq
\end{lem}

\begin{proof} 
It follows from the quadratic formula, and noting from the formula \eqref{eq:slopedef} that 
\beq \label{eq:sqslope}
	\sqrt{\slope}
	=\frac{\lv-\ac-\bc+\sqrt \discr}{2\sqrt{\ac\bc}}
	=\frac{2\sqrt{\ac\bc}}{\lv-\ac-\bc-\sqrt \discr}.
\eeq
\end{proof}

%%%%%%%%%%%%%%%%%%%%%%%%%%%%%%%
\subsection{Critical points in $\overline{U}_1 \cap (1,\infty)^2$}

\begin{lem}\label{result:cpU1}
Assume $(x,y)\in\overline{U}_1\cap (1,\infty)^2$. The roots of $\tpt$ are
\beq\label{eq:cpU1_Q}
	(\tzt^-,\tzt^+)
	=\left(
	-\frac{\mv-\ac x+\bc y+\sqrt{\tQ}}{2\mv},
	-\frac{\mv-\ac x+\bc y-\sqrt{\tQ}}{2\mv}
	\right),
	\qquad 
	\tQ:=\mv^2-2(\ac x+\bc y)\mv+(\ac x-\bc y)^2. 
\eeq
We have $\tQ>0$. The quadratic polynomial $\tpr$ has the double root 
\beq\label{eq:U1cpformula}
	\tzr:= \tzr^-=\tzr^+
	=-\frac{\sqrt{\bc(y-1)}}{\sqrt{\ac(x-1)}+\sqrt{\bc(y-1)}}.
\eeq
For $(x,y)\in U_1$,
\beq\label{eq:cpU1_order}
	-1<\tzo^-<\tzt^-<\tzr<\tzt^+<\tzo^+=\pp<0.
\eeq
For $(x,y)$ on the boundary between $U_1$ and $U_2^<$,
\beq\label{eq:cpU1U2_order}
	-1<\tzo^-<\tzt^-<\tzr=\tzt^+=\tzo^+=\pp<0.
\eeq
%At $(x,y)=(1,1)$, we have $\tpt=\tpo$ and $\tpr=0$.
\end{lem}

\begin{proof}
The quadratic formula gives \eqref{eq:cpU1_Q}, where   
\beqq
	\tQ
	=\bigl(\mv-(\sqrt{\ac x}+\sqrt{\bc y})^2\bigr)
	\bigl(\mv-(\sqrt{\ac x}-\sqrt{\bc y})^2\bigr).
\eeqq
To prove $\tQ>0$,  it is enough to show that $A:=\mv-(\sqrt{\ac x}+\sqrt{\bc y})^2>0$. Set $X:=\ac(x-1)$ and $Y:=\bc(y-1)$. By the formula \eqref{eq:mvu112} for $\mv$, 
\beqq
	A=\lv-\ac-\bc+2\sqrt{XY}-2\sqrt{(\ac+X)(\bc+Y)}.
\eeqq
Since $1/\slope \le(y-1)/(x-1)\le\slope$ in $\overline{U}_1$, using \eqref{eq:sqslope}, it follows that 
\beqq
	\ac Y-(\lv-\ac-\bc)\sqrt{XY}+\bc X\le0.
\eeqq
Thus, 
\beqq
	(\lv-\ac-\bc+2\sqrt{XY})^2-4(\ac+X)(\bc+Y)
	=\discr+4\bigl((\lv-\ac-\bc)\sqrt{XY}-\bc X-\ac Y\bigr)
	\ge \discr>0.
\eeqq
Since $\lv-\ac-\bc+2\sqrt{XY}>0$, we obtain $A>0$, and hence $\tQ>0$.

Since $\mv-\lv=(\sqrt X+\sqrt Y)^2$ by the formula \eqref{eq:mvu112} of $\mv$, 
\beqq
	\tpr(z)=\bigl((\sqrt X+\sqrt Y)z+\sqrt Y\bigr)^2.
\eeqq
Hence, \eqref{eq:U1cpformula} holds. 

We now consider the ordering of the roots. 
Suppose first that $(x,y)\in U_1$. 
%Set $q:=\sqrt{Y/X}$. 
Since
\beq \label{eq:qineq}
	\frac{\sqrt \bc}{\sqrt{\slope \ac}}<\sqrt{\frac{Y}{X}}<\frac{\sqrt{\slope \bc}}{\sqrt \ac},
\eeq
and $q\mapsto-q/(1+q)$ is strictly decreasing on $(0,\infty)$, \eqref{eq:cpconditioning} gives
\beqq
	\tzo^-<\tzr<\tzo^+.
\eeqq
We have
\beqq
	\tpt=\tpo+\tpr,
	\qquad
	\tpo(z)=\lv(z-\tzo^-)(z-\tzo^+),
	\qquad
	\tpr(z)=(\mv-\lv)(z-\tzr)^2.
\eeqq
Thus, $\tpt(\tzo^\pm)=\tpr(\tzo^\pm)>0$ and $\tpt(\tzr)=\tpo(\tzr)<0$. Hence, $\tpt$ has one root in $(\tzo^-,\tzr)$ and the other in $(\tzr,\tzo^+)$, proving \eqref{eq:cpU1_order}.

Suppose next that $(x,y)$ lies on the boundary between $U_1$ and $U_2^<$. Then, \eqref{eq:qineq} becomes 
$\frac{\sqrt{\bc}}{\sqrt{\slope \ac}}= \sqrt{\frac{Y}{X}} < \frac{\sqrt{\slope \bc}}{\sqrt{\ac}}$ and hence, $\tzo^-<\tzr=\tzo^+=\pp$. Moreover,
\beqq
	\tpt(\pp)=0,
	\qquad
	\tpt'(\pp)=\tpo'(\pp)=\lv(\pp-\tzo^-)>0.
\eeqq
Thus, $\pp$ is the larger root of $\tpt$, and the other root lies strictly between $\tzo^-$ and $\pp$. This proves \eqref{eq:cpU1U2_order}.
\end{proof}

%%%%%%%%%%%%%%%%%%%%%%%%%%%%%%%
\subsection{Critical points in $\overline{U}_2^<\cap\{ (x,y): y<x\}$}

\begin{lem}\label{result:cpU2less}
Let $(x,y)\in\overline{U}_2^<$ with $y< x$. The roots of $\tpt$ are
\beq\label{eq:cpU2less_tzt}
	(\tzt^-,\tzt^+)
	=\left(\frac{\bc y}{\mv\pp},\pp\right), 
\eeq
and $\tzt^-=\tzt^+$ if and only if $(x,y)$ lies on the boundary between $U_2^<$ and $U_3^<$.
When $\mv\ne\lv$, the roots of $\tpr$ are
\beqq
	(\tzra,\tzrb)
	=\left(\pp,\frac{\bc(y-1)}{(\mv-\lv)\pp}\right).
\eeqq
The roots are ordered as follows:
\begin{itemize}
\item If $\mv>\lv$, then
\beqq
	-1<\tzo^-<\tzt^-\le\tzt^+=\tzo^+=\tzra=\pp\le\tzrb,
\eeqq
where
\beqq
	\text{$\tzrb\in [\pp, 0)$ for $y>1$;}
	\qquad
	\text{$\tzrb=0$ for $y=1$;}
	\qquad
	\text{$\tzrb>0$ for $y<1$,}
\eeqq
and $\tzrb=\pp$ if and only if $(x,y)$ lies on the boundary between $U_1$ and $U_2^<$.

\item 
If $\mv<\lv$, then
\beqq
	\tzrb<\tzo^-<\tzt^-\le\tzt^+=\tzo^+=\tzra=\pp<0,
\eeqq
where
\beqq
	\text{$\tzrb<-1$ for $x>1$;}
	\qquad
	\text{$\tzrb=-1$ for $x=1$;}
	\qquad
	\text{$\tzrb\in (-1, \tzo^-)$ for $x<1$.}
\eeqq
\end{itemize}
\end{lem}

\begin{proof}
Inserting the formula \eqref{eq:mvu212} of $\mv$, 
\beqq
	\tpt(z)
	=\frac12z\bigl[(\lv+\ac-\bc-\sqrt \discr)z+\lv-\ac-\bc-\sqrt \discr\bigr]x
	+\frac12(z+1)\bigl[(\lv-\ac+\bc+\sqrt \discr)z+2\bc\bigr]y.
\eeqq
Using the second and third formulas of \eqref{eq:ppthree}, we find that each coefficient of $x$ and $y$ vanishes at $z=\pp$, and hence, $\pp$ is a root of $\tpt$. 
Using the first formula of \eqref{eq:ppthree}, we have
\beqq
	\tpt'(\pp)=\frac{\lv-\ac-\bc-\sqrt \discr}{2}(\slope y-x).
\eeqq
Since $y \ge x/\slope$ in $\overline{U}_2^<$, $\tpt'(\pp) \ge 0$, and we find that $\pp$ is the larger root of $\tpt$. 
The other root is found from the fact that the product of the two roots of $\tpt$ is $\bc y/\mv$. 
Hence, \eqref{eq:cpU2less_tzt} follows. Moreover, $\tzt^-=\tzt^+$ if and only if $y=x/\slope$. 
Furthermore, since $\tpt(-1)=\ac x>0$ and $\pp>-1$, we also have $\tzt^->-1$.

From a direct computation, 
\beq \label{eq:tprtoze}
	\tpt(\tzo^-)
	=\frac{\sqrt \discr(\lv-\ac+\bc+\sqrt \discr)(\lv+\ac-\bc-\sqrt \discr)}{4\lv^2}(x-y)
	=-(x-y)\sqrt \discr\,\tzo^-(\tzo^-+1). 
\eeq
Thus, $\tpt(\tzo^-)>0$ since $y<x$ and $\tzo^-\in (-1,0)$. 
As $\tzo^-<\pp=\tzt^+$, it follows that
\beqq
	-1<\tzo^-<\tzt^-\le\tzt^+=\tzo^+=\pp,
\eeqq
where $\tzt^-=\tzt^+$ if and only if $(x,y)$ lies on the boundary between $U_2^<$ and $U_3^<$.

Assume $\mv\ne\lv$. Since $\tpr=\tpt-\tpo$ and $\pp$ is a zero of both $\tpo$ and $\tpt$, it is also a zero of $\tpr$. Set $\tzra:=\pp$. 
Since the product of the roots of $\tpr$ is $\bc(y-1)/(\mv-\lv)$, the other root is 
\beq\label{eq:bfor}
	\tzrb=\frac{\bc(y-1)}{(\mv-\lv)\pp}.
\eeq
From a direct computation, 
\beqq
	\tpr'(\tzra) = \tpr'(\pp) 
	=\frac{\lv-\ac-\bc-\sqrt \discr}{2}\bigl(\slope(y-1)-(x-1)\bigr).
\eeqq
Thus, $\pp$ is the double root of $\tpr$ if and only if $y-1=(x-1)/\slope$, i.e. $(x,y)$ is on the boundary between $U_1$ and $U_2^<$. Otherwise, $\tpr'(\tzra)<0$. In this case, since 
\beqq
	\tpr(z)=(\mv-\lv)(z-\tzra)(z-\tzrb),
\eeqq
we have $\tzrb>\tzra$ when $\mv>\lv$ and $\tzrb<\tzra$ when $\mv<\lv$.

If $\mv>\lv$, from the formula \eqref{eq:bfor} and the fact that  $\pp<0$, we find $\tzrb<0$, $\tzrb=0$, or $\tzrb>0$ according to $y>1$, $y=1$, or $y<1$. 
When $y>1$, the above discussion implies that $\tzrb\in [\pp, 0)$. 

If $\mv<\lv$, the formula for $\tpr$ implies that $\tpr(-1)=\ac(x-1)$. Since one root of $\tpr$ is $\pp\in (-1,0)$, the other root satisfies 
$\tzrb<-1$, $\tzrb=-1$, or $\tzrb>-1$ according to $x>1$, $x=1$, or $x<1$.
Furthermore, in this case, 
since 
\beqq
	\tpr(\tzo^-)=\tpt(\tzo^-)>0
\eeqq
from \eqref{eq:tprtoze}, and 
$\tpr$ has negative leading coefficient and roots $\tzrb<\tzra=\pp$, we obtain $\tzrb<\tzo^-$. 
\end{proof}

When $\mv=\lv$, the polynomial $\tpr$ becomes linear. The corresponding equality cases are treated separately in Subsection~\ref{sec:U2equal} and at the end of Section~\ref{sec:U2U3}, for the $U_2$ region and the $U_2/U_3$ boundary, respectively.

%%%%%%%%%%%%%%%%%%%%%%%%%%%%%%%
\subsection{Phase gaps}\label{sec:phasegaps}

\begin{lem} \label{result:phasegaps}
For $(x,y)\neq(1,1)$ in $U_1$ or $U_2^<$, or on either the $U_1/U_2^<$ boundary or the $U_2^</U_3^<$ boundary, the following phase gaps are strictly positive whenever they are defined:
\begin{itemize}
\item $\tDel:=\Thetao(\tzo^+)-\Thetao(\tzo^-)>0$. 
\item $\tDelt:=\Thetat(\tzt^+)-\Thetat(\tzt^-)>0$ if $\tzt^-<\tzt^+$. 
\item $\tDelr:= \Thetar(\tzrb)-\Thetar(\tzra)>0$ if $\tzra, \tzrb\in (-1,0)$. 
\end{itemize}
\end{lem}

\begin{proof}
For $i=1,2$, the function $\Theta_i$ is strictly increasing on
$[\tz_i^-,\tz_i^+]$ whenever $\tz_i^-<\tz_i^+$. Hence,
$\tDel>0$ and $\tDelt>0$. 
Consider $\tDelr$. The points $\tzra,\tzrb$ are defined for
$(x,y)\in\overline{U}_2^<$ with $y\neq x$ and $\mv\neq\lv$, and
$\Thetar'(z)=\frac{(\mv-\lv)(z-\tzra)(z-\tzrb)}{z(z+1)}$. 
If $\mv>\lv$, then $\tzra<\tzrb$, and $\Thetar$ is strictly increasing on
$[\tzra,\tzrb]$ when $\tzra,\tzrb\in(-1,0)$. If $\mv<\lv$, then
$\tzrb<\tzra$, and $\Thetar$ is strictly decreasing on
$[\tzrb,\tzra]$. Therefore, in either case, $\Thetar(\tzrb)-\Thetar(\tzra)>0$. 
\end{proof}

%%%%%%%%%%%%%%%%%%%%%%%%%%%%%%%
%%%%%%%%%%%%%%%%%%%%%%%%%%%%%%%
\section{Steepest-descent estimates}
\label{sec:asymgeneral}

The critical-point analysis in Section~\ref{sec:cp} determines the contour geometry needed for steepest descent. In this section, we establish local asymptotics and contour bounds for the functions $\ff_*(z)$, $*\in\mc A_2$, appearing in $\J^{\bsigma}_{\btau}$.

%%%%%%%%%%%%%%%%%%%%%%%%%%%%%%%
\subsection{Setup and contours}

The asymptotic analysis in the following sections involves functions of the form\footnote{In the statements of the main results, $N$ denotes the large parameter. Throughout the asymptotic analysis, we use $L$ instead.}
\beq\label{eq:fLgnal}
	\ff_L(z)=\frac{z^{N_L} e^{T_Lz}}{(z+1)^{M_L}}, 
\eeq
for certain families of parameters $(M_L,N_L,T_L)\in \mathbb{Z}\times \mathbb{Z}\times [0,\infty)$ as $L\to\infty$. We consider two basic families of functions. A modified Gaussian family, used in the boundary regimes, will be introduced at the end of Section~\ref{sec:asymgcase1}. 

\begin{itemize}
\item (Gaussian family): Let $\alpha_1,\alpha_2\in\R$, $\alpha_3\ge0$, and $\beta_1,\beta_2,\beta_3\in\R$. For $L>0$, set
\beqq
	M_L=\alpha_1L+\beta_1L^{1/2}+\delta_L^1,\qquad
	N_L=\alpha_2L+\beta_2L^{1/2}+\delta_L^2,\qquad
	T_L=\alpha_3L+\beta_3L^{1/2},
\eeqq
where $\delta_L^1,\delta_L^2\in(-1,1)$, and $M_L$ and $N_L$ are integers. We consider three cases:
\begin{enumerate}[(a)]
\item $\alpha_1,\alpha_2>0$ and $\sqrt{\alpha_3}>\sqrt{\alpha_1}+\sqrt{\alpha_2}$.

\item $\alpha_1<0$ and $\alpha_2>0$, or $\alpha_1=0$ and $0<\alpha_2<\alpha_3$.

\item $\alpha_1>0$ and $\alpha_2<0$, or $\alpha_2=0$ and $0<\alpha_1<\alpha_3$.
\end{enumerate}

\item (Airy family): Let $\alpha_1,\alpha_2,\alpha_3>0$, $\beta_1,\beta_2,\beta_3\in\R$, and $\ga_1,\ga_2,\ga_3\in\R$. Suppose that
\beq\label{eq:typeiicond}
	\sqrt{\alpha_3}=\sqrt{\alpha_1}+\sqrt{\alpha_2},
	\qquad
	\be_3\sqrt{\alpha_1\alpha_2}= \be_1\sqrt{\alpha_2\alpha_3} +\be_2\sqrt{\alpha_1\alpha_3}. 
\eeq
For $L>0$, set
\beqq
\begin{aligned}
	&M_L=\alpha_1L+\beta_1L^{2/3}+\ga_1L^{1/3}+\delta_L^1,\\
	&N_L=\alpha_2L+\beta_2L^{2/3}+\ga_2L^{1/3}+\delta_L^2,\\
	&T_L=\alpha_3L+\beta_3L^{2/3}+\ga_3L^{1/3},
\end{aligned}
\eeqq
where $\delta_L^1,\delta_L^2\in(-1,1)$ and $M_L$ and $N_L$ are integers.
\end{itemize}

We use the following contours.

\begin{definition}\label{def:contours}
For $z_0\in(-1,0)$ and $L>0$, define
\beqq
\begin{split}
    \cont_-^L(z_0)
    &:=\{z_0+re^{2\pi\ii/3}:r\in[0,L^{-1/12}]\}
    \cup\{z_0+re^{-2\pi\ii/3}:r\in[0,L^{-1/12}]\}
    \cup\gamma_-,
\end{split}
\eeqq
where $\gamma_-$ is the circular arc centered at $-1$ joining the points $z_0+L^{-1/12}e^{\pm2\pi\ii/3}$ and passing to the left of $-1$. Similarly, define
\beqq
\begin{split}
    \cont_+^L(z_0)
    &:=\{z_0+re^{\pi\ii/3}:r\in[0,L^{-1/12}]\}
    \cup\{z_0+re^{-\pi\ii/3}:r\in[0,L^{-1/12}]\}
    \cup\gamma_+,
\end{split}
\eeqq
where $\gamma_+$ is the circular arc centered at $0$ joining the points $z_0+L^{-1/12}e^{\pm\pi\ii/3}$ and passing to the right of $0$.
\end{definition}

As throughout the paper, simple closed contours are oriented counterclockwise. The particular cutoff $L^{-1/12}$ is not essential; this choice ensures that the leading-phase decay on the circular arcs dominates the $L^{2/3}$-order term in the Airy family.

%%%%%%%%%%%%%%%%%%%%%%%%%%%%%%%%%%%%%%%%%%%%%%%%
\subsection{Gaussian-family asymptotics}
\label{sec:asymgcase1}

For the Gaussian family, \eqref{eq:fLgnal} takes the form
\beqq
	\ff_L(z)=e^{L\GG(z)+L^{1/2}\HH(z)+\EE_L(z)},
\eeqq
where\footnote{Throughout the paper, $\EE_L$ and $\EE_{*,L}$ are the correction terms for which $\ff_L$ and $\ff_{*,L}$ extend analytically to $\C\setminus\{-1,0\}$; these correction terms are uniformly bounded on every compact subset of $\C\setminus\{-1,0\}$. We do not repeat this statement below.}
\beq\label{eq:Gaph}
\begin{split}
	\GG(z)&:=-\alpha_1\log(z+1)+\alpha_2\log z+\alpha_3z,\\
	\HH(z)&:=-\beta_1\log(z+1)+\beta_2\log z+\beta_3z,\\
	\EE_L(z)&:=-\delta_L^1\log(z+1)+\delta_L^2\log z.
\end{split}
\eeq
Only the critical points of $\GG$ in the interval $(-1,0)$ will be used. It is straightforward to check the following properties:
\begin{itemize}
\item For Case (a), there are two critical points $-1<\cp^-<\cp^+<0$, and $\GG''(\cp^-)>0$ and $\GG''(\cp^+)<0$.

\item For Case (b), there is one critical point $\cp^+$ in $(-1,0)$, and $\GG''(\cp^+)<0$.

\item For Case (c), there is one critical point $\cp^-$ in $(-1,0)$, and $\GG''(\cp^-)>0$.
\end{itemize}
For the strict mixed-sign cases with $\alpha_3>0$, these properties are given in~\cite[Lemma~5.2]{Baik-Cordaro-Tripathi25}. The cases $\alpha_3=0$ in~\textnormal{(b)} and~\textnormal{(c)}, as well as the endpoint cases $\alpha_1=0$ and $\alpha_2=0$, follow directly from the same calculation, after canceling the corresponding factor in $\GG'$ when necessary. 

Taylor's theorem gives the following local asymptotic.

\begin{lem}[Lemma~5.1 of~\cite{Baik-Cordaro-Tripathi25}]\label{lem:asymptotics_f1}
Let $(\ff_L)_{L>0}$ be a Gaussian family, and let $\cp$ be a critical point of $\GG$. Then, for every $\epsilon\in(0,1/2)$, there exists $L_0>0$ such that, for all $L>L_0$ and $|w|\le L^{\epsilon/3}$,
\beqq
	\ff_L(\cp+wL^{-1/2})
	=\ff_L(\cp)e^{\frac12\GG''(\cp)w^2+\HH'(\cp)w}
	\bigl(1+O(L^{-1/2+\epsilon})\bigr).
	%\qquad  \text{for $|w| \leq L^{\epsilon/3}$.}
\eeqq
\end{lem}

We also need bounds on suitable contours. In \cite[Lemma~5.3]{Baik-Cordaro-Tripathi25}, uniform bounds were obtained on appropriate circular contours. The same steepest-descent argument, with the local change of variables $z=\cp^\pm+uL^{-1/2}$, yields the corresponding $L^1$-bounds. The argument also applies to the contours $\cont^L_{\pm}(z_0)$: the line-segment portions lie in the Gaussian decay sectors, while the circular portions are controlled by the monotonicity of $\re(\GG(z))$ established in \cite[Lemma~5.4]{Baik-Cordaro-Tripathi25}. The cases $\alpha_3=0$, as well as the endpoint cases in~\textnormal{(b)} and~\textnormal{(c)}, follow from the same calculation. We therefore obtain the following estimates.

\begin{lem}[Lemma~5.3 of~\cite{Baik-Cordaro-Tripathi25}]
\label{lem:asymptotics_f2}
Let $(\ff_L)_{L>0}$ be a Gaussian family.
\begin{enumerate}[(a)]
\item\label{eq:propertya}
In Cases~\textnormal{(a)} and~\textnormal{(c)}, for every $b\in\R$, set 
\beqq
	\Gamma_-:=\cont^L_-(\cp^-+bL^{-1/2}).
\eeqq
There exist $c,C,L_0>0$ such that, for all $L\ge L_0$ and every $\epsilon\in(0,1/2)$,
\beqq
	|\ff_L(z)|
	\le Ce^{-cL^{2\epsilon/3}}|\ff_L(\cp^-)|
	\qquad
	\text{if $z\in\Gamma_-$ and $|z-\cp^-|\ge L^{-1/2+\epsilon/3}$,}
\eeqq
and
\beqq
	\|\ff_L\|_{L^1(\Gamma_-)}
	\le \frac{C}{L^{1/2}}|\ff_L(\cp^-)|. 
\eeqq

\item\label{eq:propertyb}
In Cases~\textnormal{(a)} and~\textnormal{(b)}, for every $b\in\R$, set
\beqq
	\Gamma_+:=\cont^L_+(\cp^++bL^{-1/2}).
\eeqq
There exist $c,C,L_0>0$ such that, for all $L\ge L_0$ and every $\epsilon\in(0,1/2)$,
\beqq
	\frac{1}{|\ff_L(z)|}
	\le Ce^{-cL^{2\epsilon/3}}\frac{1}{|\ff_L(\cp^+)|}
	\qquad
	\text{if $z\in\Gamma_+$ and $|z-\cp^+|\ge L^{-1/2+\epsilon/3}$,}
\eeqq
and
\beqq
	\left\|\frac{1}{\ff_L}\right\|_{L^1(\Gamma_+)}
	\le \frac{C}{L^{1/2}|\ff_L(\cp^+)|}.
\eeqq
\end{enumerate}
The constants may be chosen uniformly when $b$ ranges over a fixed bounded set.
\end{lem}

We also need a variant of the Gaussian family in the boundary regimes.

\begin{lem}\label{lem:modifiedGaussian}
We call $(\ff_L)_{L>0}$ a \emph{modified Gaussian family} if
\beqq
        \ff_L(z)=e^{L\GG(z)+L^{2/3}\HH(z)+L^{1/3}\KK(z)+\EE_L(z)},
\eeqq
where $\GG$ is the leading phase of a Gaussian family in one of Cases~\textnormal{(a)}-\textnormal{(c)}, $\HH$ and $\KK$ are fixed functions of the form $c_1\log(z+1)+c_2\log z+c_3z$, and $\EE_L$ is as in \eqref{eq:Gaph}. Let $\cp$ be a regular critical point of $\GG$. Then, for every $\epsilon\in(0,1/2)$, there exists $L_0>0$ such that, for all $L>L_0$ and $|w|\le L^{\epsilon/3}$,
\beq\label{eq:modifiedGaussianlocal}
        \ff_L(\cp+wL^{-1/2})=\ff_L(\cp)e^{\frac12\GG''(\cp)w^2+\HH'(\cp)wL^{1/6}}\bigl(1+O(L^{-1/6+\epsilon/3})\bigr).
\eeq
If $\HH'(\cp)=0$ for $\cp=\cp^-$ or $\cp=\cp^+$, the corresponding decay and $L^1$ estimates in Lemma~\ref{lem:asymptotics_f2} remain valid without change. If $\HH'(\cp)\ne0$, then, for every $b\in\R$, the corresponding $L^1$ estimate is replaced by the appropriate bound below:
\beqq
        \|\ff_L\|_{L^1(\cont_-^L(\cp^-+bL^{-1/2}))}\le\frac{Ce^{CL^{1/3}}}{L^{1/2}}|\ff_L(\cp^-)|,\qquad \left\|\frac1{\ff_L}\right\|_{L^1(\cont_+^L(\cp^++bL^{-1/2}))}\le\frac{Ce^{CL^{1/3}}}{L^{1/2}|\ff_L(\cp^+)|}.
\eeqq
The constants may be chosen uniformly when $b$ ranges over a fixed bounded set.
\end{lem}

\begin{proof}
Taylor's theorem gives \eqref{eq:modifiedGaussianlocal}. For the $L^1$ estimates, consider first the two line-segment portions of $\cont_-^L(\cp^-+bL^{-1/2})$. Writing
\beqq
	z=\cp^-+bL^{-1/2}+re^{\pm2\pi\ii/3},
    \qquad 0\le r\le L^{-1/12},
\eeqq
we obtain, if $\HH'(\cp^-)\ne0$,
\beq \label{eq:agi1}
 	\log\left|\frac{\ff_L(z)}{\ff_L(\cp^-)}\right|
 	\le -c_1Lr^2
      +c_2L^{2/3}\bigl(r+L^{-1/2}\bigr)
      +c_3L^{1/3}\bigl(r+L^{-1/2}\bigr)+c_4
\eeq
for some positive constants $c_1,\ldots,c_4$. By the weighted arithmetic-geometric inequality, for every $\delta>0$, 
\beqq
	L^{2/3}r\le \delta Lr^2 + \frac1{4\delta} L^{1/3},
	\qquad
	L^{1/3}r\le \delta Lr^2 + \frac1{4\delta} L^{-1/3}.
\eeqq
Choosing $\delta>0$ sufficiently small, \eqref{eq:agi1} yields
\beq \label{eq:agi2}
 	\log\left|\frac{\ff_L(z)}{\ff_L(\cp^-)}\right|
      \le -c_5Lr^2+c_6L^{1/3}+c_7
\eeq
for some positive constants $c_5,c_6,c_7$.

If $\HH'(\cp^-)=0$, the $c_2L^{2/3}(r+L^{-1/2})$ term in \eqref{eq:agi1} is replaced by $c_2L^{2/3}(r^2+L^{-1})$, and \eqref{eq:agi2} changes to
\beqq
 	\log\left|\frac{\ff_L(z)}{\ff_L(\cp^-)}\right|
      \le -c_5Lr^2+c_6.
\eeqq
On the circular part, by \cite[Lemma~5.4]{Baik-Cordaro-Tripathi25}, $\re(\mcG(z))$ decreases as $z$ travels away from $z_c^-+bL^{-1/2}$, and thus it follows that $\bigl| \frac{\ff_L(z)}{\ff_L(\cp^-)} \bigr| \le e^{-cL^{5/6}}$ 
on the circular part for some $c>0$.
Integrating over the contour, we obtain the first $L^1$ estimate. The second $L^1$-estimate is similar.
\end{proof}

%%%%%%%%%%%%%%%%%%%%%%%%%%%%%%%%%%%%%%%%%%%%%%%%
\subsection{Airy-family asymptotics}
\label{sec:asymgcase2}

For an Airy family, 
\beqq
	\ff_L(z)=e^{L\GG(z)+L^{2/3}\HH(z)+L^{1/3}\KK(z)+\EE_L(z)},
\eeqq
where
\beqq 
\begin{split}
	\GG(z)&:=-\alpha_1\log(z+1)+\alpha_2\log z+\alpha_3 z,\\
	\HH(z)&:=-\be_1\log(z+1)+\be_2\log z+\be_3z,\\
	\KK(z)&:=-\ga_1\log(z+1)+\ga_2\log z+\ga_3z,\\
	\EE_L(z)&:=-\delta_L^1\log(z+1)+\delta_L^2\log z.
\end{split}
\eeqq
The assumptions \eqref{eq:typeiicond} imply that $\GG$ has a double critical point $z_c\in(-1,0)$, and
\beqq %\label{eq:dbcritpt2}
	z_c=-\frac{\sqrt{\alpha_2}}{\sqrt{\alpha_3}}, 
	\qquad
	\HH'(z_c)=0, 
	\qquad
	\GG'''(z_c)
	=-\frac{2(\sqrt{\alpha_1}+\sqrt{\alpha_2})^4}
	{\sqrt{\alpha_1\alpha_2}}<0.
\eeqq

\begin{lem}\label{result:dbcplocal}
Let $(\ff_L)_{L>0}$ be an Airy family. Then, for every $\epsilon\in(0,1/3)$, there exists $L_0>0$ such that, for all $L>L_0$ and $|w|\le L^{\epsilon/4}$,
\beqq
	\ff_L(z_c+wL^{-1/3})
	=\ff_L(z_c)e^{\frac16\GG'''(z_c)w^3+\frac12\HH''(z_c)w^2+\KK'(z_c)w}
	\bigl(1+O(L^{-1/3+\epsilon})\bigr).
\eeqq
\end{lem}

The above follows from Taylor's theorem.
We also have the following bounds.

\begin{lem}\label{lem:asymptotics_f2_case2}
Let $(\ff_L)_{L>0}$ be an Airy family. For every $b\in\R$, set
\beqq
	\cont_\pm:=\cont_\pm^L(z_c+bL^{-1/3}).
\eeqq
There exist $c,C,L_0>0$ such that, for all $L\ge L_0$ and every $\epsilon\in(0,1/3)$,
\beq\label{eq:f2_case2_left_decay}
	|\ff_L(z)|
	\le Ce^{-cL^{3\epsilon/4}}|\ff_L(z_c)|
	\qquad
	\text{if $z\in\cont_-$ and $|z-z_c|\ge L^{-1/3+\epsilon/4}$,}
\eeq
and
\beq\label{eq:f2_case2_right_decay}
	\frac{1}{|\ff_L(z)|}
	\le Ce^{-cL^{3\epsilon/4}}\frac{1}{|\ff_L(z_c)|}
	\qquad
	\text{if $z\in\cont_+$ and $|z-z_c|\ge L^{-1/3+\epsilon/4}$.}
\eeq
Moreover,
\beq\label{eq:f2_case2_L1bound} %\label{eq:f2_case2_bounds}
	\|\ff_L\|_{L^1(\cont_-)}
	\le\frac{C}{L^{1/3}}|\ff_L(z_c)|, 
	\qquad 
	\left\|\frac{1}{\ff_L}\right\|_{L^1(\cont_+)}
	\le\frac{C}{L^{1/3}|\ff_L(z_c)|}.
\eeq
The constants may be chosen uniformly when $b$ ranges over a fixed bounded set.
\end{lem}

\begin{proof}
Fix $\epsilon\in(0,1/3)$ and $b\in\R$, and set $z_0:=z_c+bL^{-1/3}$. Set 
\beqq
\begin{split}
	\Gamma^*_+
	&:=\{z_0+re^{\pi\ii/3}:r\in[0,L^{-1/12}]\}
	\cup\{z_0+re^{-\pi\ii/3}:r\in[0,L^{-1/12}]\},\\
	\Gamma^*_-
	&:=\{z_0+re^{2\pi\ii/3}:r\in[0,L^{-1/12}]\}
	\cup\{z_0+re^{-2\pi\ii/3}:r\in[0,L^{-1/12}]\},
\end{split}
\eeqq
so that $\cont_\pm=\Gamma^*_\pm\cup\ga_\pm$, where $\ga_\pm$ are defined as in Definition~\ref{def:contours}. Let $z^\pm_1$ denote the endpoint of $\ga_\pm$ in the upper half-plane.

Although \cite[Lemma~5.4(a)]{Baik-Cordaro-Tripathi25} assumes that $\alpha_3>(\sqrt{\alpha_1}+\sqrt{\alpha_2})^2$, its proof remains valid when $\alpha_3=(\sqrt{\alpha_1}+\sqrt{\alpha_2})^2$, with the strict inequalities in \cite[(5.11)]{Baik-Cordaro-Tripathi25} replaced by $|z_c+1|=s_-$ and $|z_c|=s_+$. Consequently, for every $s\in(0,|z_c+1|)$, the function $\re\GG(-1+se^{\ii\theta})$ decreases as $|\theta|$ increases from $0$. Similarly, for every $s\in(0,|z_c|)$, the function $\re\GG(se^{\ii\theta})$ is minimized at $\theta=\pi$ and increases as $\theta$ moves away from $\pi$. Since $\re\GG(z)=\re\GG(\overline z)$, it follows that
\beq
\label{eq:reGGincdecga+-}
	\mp\re\GG(z)\le\mp\re\GG(z^\pm_1),
	\qquad \text{for every $z\in\ga_\pm$,}
\eeq
for all sufficiently large $L$.

Since $z_c$ is a double critical point of $\GG$ and $\GG'''(z_c)<0$, write $\GG(z)-\GG(z_c)=\frac16\GG'''(z_c)(z-z_c)^3+O(|z-z_c|^4)$. On $\Gamma^*_- $ we have $z-z_c=bL^{-1/3}+re^{\pm2\pi\ii/3}$, while on $\Gamma^*_+$ we have $z-z_c=bL^{-1/3}+re^{\pm\pi\ii/3}$. 
Hence, there exist constants $k>0$ such that
\beq
\label{eq:reGtayonGammapm*}
	\mp\re\bigl(\GG(z)-\GG(z_c)\bigr)
	\le-k|z-z_c|^3
	\qquad \text{for $z\in\Gamma^*_\pm$,}
\eeq
for all sufficiently large $L$. 
Since $\HH'(z_c)=0$, there exists $C_0>0$ such that 
\beq
\label{eq:HKEbound}
	|\HH(z)-\HH(z_c)|\le C_0|z-z_c|^2,\qquad
	|\KK(z)-\KK(z_c)|\le C_0|z-z_c|,\qquad
	|\EE_L(z)-\EE_L(z_c)|\le C_0,
\eeq
for $z\in\Gamma^*_\pm$ and all sufficiently large $L$. Therefore, there exists $L_0>1$ such that
\beq
\label{eq:flfzccommonl1}
	\mp\log\left|\frac{\ff_L(z)}{\ff_L(z_c)}\right|
	\le-k|z-z_c|^3L
	+C_0|z-z_c|^2L^{2/3}
	+C_0|z-z_c|L^{1/3}
	+C_0,
	\qquad \text{for $z\in\Gamma^*_\pm$,}
\eeq
for every $L\ge L_0$. Hence, there exists $L_1>1$ such that
\beq
\label{eq:flzflzcexpboundinGamma*pm}
	\mp \log \left| \frac{\ff_L(z)}{\ff_L(z_c)} \right| \le -\frac{k}{2}|z-z_c|^3L \le -cL^{3\epsilon/4}, \quad \text{for $ z\in \Gamma^*_\pm \cap \{z \in \C : |z-z_c| \ge L^{-1/3+\epsilon/4}\}$ } 
\eeq
for every $L \ge L_1$ and some $c>0$. From \eqref{eq:reGGincdecga+-} and \eqref{eq:reGtayonGammapm*}, applied at $z=z^\pm_1$, and using $|z^\pm_1-z_c|=L^{-1/12}(1+o(1))$, we obtain, after adjusting the constants,
\beqq
	\mp\re\bigl(\GG(z)-\GG(z_c)\bigr)
	\le-kL^{-1/4},
	\qquad \text{for $z\in\ga_\pm$.}
\eeqq
Moreover, since $\ga_\pm$ remain in a compact subset of $\C\setminus\{-1,0\}$ for all sufficiently large $L$, there exist $L_2>1$ and $K>0$ such that
\beq
\label{eq:flzflzcexpboundingammapm}
	\mp\log\left|\frac{\ff_L(z)}{\ff_L(z_c)}\right|
	\le-kL^{3/4}+KL^{2/3},
	\qquad \text{for $z\in\ga_\pm$,}
\eeq
for every $L\ge L_2$. Thus, \eqref{eq:f2_case2_left_decay} and \eqref{eq:f2_case2_right_decay} follow from \eqref{eq:flzflzcexpboundinGamma*pm} and \eqref{eq:flzflzcexpboundingammapm}.

We now prove the $L^1$-bounds. Using \eqref{eq:flfzccommonl1} and the change of variables $z=z_c+wL^{-1/3}$, we find that there exists a constant $C_1>0$ such that
\beqq
\begin{split}
	\left|\frac{\ff_L(z_c)}{\ff_L(z)}\right|
	&\le C_1e^{-k|w|^3+C_0|w|^2+C_0|w|},
	\qquad
	\text{for $z\in\Gamma^*_+\cap\{z\in\C:|z-z_c|\le L^{-1/3+\epsilon/4}\}$},\\
	\left|\frac{\ff_L(z)}{\ff_L(z_c)}\right|
	&\le C_1e^{-k|w|^3+C_0|w|^2+C_0|w|},
	\qquad
	\text{for $z\in\Gamma^*_-\cap\{z\in\C:|z-z_c|\le L^{-1/3+\epsilon/4}\}$}.
\end{split}
\eeqq
The right-hand side is integrable along the scaled contours. The contributions from the portions of the contours satisfying $|z-z_c|\ge L^{-1/3+\epsilon/4}$ are exponentially small by the estimates above. This proves \eqref{eq:f2_case2_L1bound}. All constants above may be chosen uniformly when $b$ ranges over a fixed bounded set.
\end{proof}

%%%%%%%%%%%%%%%%%%%%%%

\subsection{Uniform bounds for the $\J$-integrals}\label{sec:uniform}

We use the following standard Cauchy-determinant estimate, which follows from Hadamard's inequality.

\begin{lem}\label{result:Cdetest}
If $|a_i-b_j|\ge d$ for all $1\le i,j\le n$, then
\beqq
    \bigl|\K_n \bigl( (a_1,\ldots,a_n ) \mid (b_1,\ldots,b_n)\bigr)\bigr|\le\frac{n^{n/2}}{d^n}.
\eeqq
\end{lem}

For every $L>0$, let $\ff_{L,1}$ and $\ff_{L,2}$ be functions of the form \eqref{eq:fLgnal}, and $\ff_{L,12}:=\ff_{L,1}\ff_{L,2}$. 
For every $\bn=(n_1,n_2)\in\N^2$ and $\bsigma,\btau\in\listn$, define, as in Definitions~\ref{def:Pisigmaxi} and~\ref{def:Jintegral},
\beqq
	\J^{\bsigma}_{\btau}(L)
	:=
	\frac{1}{(2\pi\ii)^{|\bsigma|+|\btau|}}
	\int\dd\bsxi^{\bsigma}\int\dd\bseta^{\btau}\,
	\nPi^{\bsigma}_{\btau}(\bsxi,\bseta)
	\FF^{\bsigma}_{\btau,L},
\eeqq
where
\beqq
	\FF^{\bsigma}_{\btau,L}
	:=
	\frac{\prod_{*\in\mc A_2}\prod_{i=1}^{a_*}\ff_{L,*}(\xi_i^*)}
	{\prod_{*\in\mc A_2}\prod_{i=1}^{b_*}\ff_{L,*}(\eta_i^*)},
\eeqq
and
\beq\label{eq:Caucdt}
	\nPi^{\bsigma}_{\btau}(\bsxi,\bseta)
	:=
	\K_{n_1}(\beot,\beo\mid\bxot,\bxo)
	\K_{a_1+b_2}(\bxo,\bet\mid\beo,\bxt)
	\K_{n_2}(\bxot,\bxt\mid\beot,\bet),
\eeq
with $\type(\bsigma)=(a_{12},a_1,a_2)$ and
$\type(\btau)=(b_{12},b_1,b_2)$.

\begin{definition}[Admissible $d_L$-separated contour system]
\label{def:admissiblecontours}

For a Gaussian or modified Gaussian family, we call
\beqq
    \cont_-^L(\cp^-+bL^{-1/2})
    \qquad\text{and}\qquad
    \cont_+^L(\cp^++bL^{-1/2}),
    \qquad b\in\R,
\eeqq
the left and right asymptotic contours, respectively, whenever the corresponding critical point exists.
For an Airy family with double critical point $\cp$, we call
\beqq
    \cont_-^L(\cp+bL^{-1/3})
    \qquad\text{and}\qquad
    \cont_+^L(\cp+bL^{-1/3}),
    \qquad b\in\R,
\eeqq
the left and right asymptotic contours, respectively. 
A contour system for $\J^{\bsigma}_{\btau}(L)$ is called \emph{admissible} if every $\xi_i^*$ is integrated over a left asymptotic contour for $\ff_{L,*}$, every $\eta_i^*$ is integrated over a right asymptotic contour for $\ff_{L,*}$, and the nesting prescribed by $\bsigma$ and $\btau$ is preserved.
For $d_L>0$, an admissible contour system is called \emph{$d_L$-separated} if any two contours carrying variables that appear on opposite sides of the same Cauchy determinant in \eqref{eq:Caucdt} are separated by a distance of at least $d_L$.
\end{definition}

The next proposition gives a uniform bound for $\J^{\bsigma}_{\btau}(L)$. For $\bn=(n_1,n_2)\in\N^2$ and $\bsigma,\btau\in\listn$, write
\beqq
    \type(\bsigma)=(a_{12},a_1,a_2),
    \qquad
    \type(\btau)=(b_{12},b_1,b_2).
\eeqq
For each $*\in\mc A_2$, suppose that $(\ff_{L,*})_{L>0}$ is an Airy family, a Gaussian family of Case~\textnormal{(a)}, or a modified Gaussian family whose leading phase is of Case~\textnormal{(a)}. For a Gaussian or modified Gaussian family, let $z_*^-<z_*^+$ denote its critical points, and for an Airy family set $z_*^-=z_*^+:=z_*^{\mathrm c}$. Define
\beqq
\begin{split}
    \mc B_{\mathrm G}
    &:=\{*\in\mc A_2:\ff_{L,*}\text{ is Gaussian}\},\\
    \mc B_{\mathrm{MG}}
    &:=\{*\in\mc A_2:\ff_{L,*}\text{ is modified Gaussian}\},\\
    \mc B_{\mathrm A}
    &:=\{*\in\mc A_2:\ff_{L,*}\text{ is Airy}\}.
\end{split}
\eeqq
Set
\beq \label{eq:abmu}
    \alpha:=\sum_{*\in\mc B_{\mathrm G}\cup\mc B_{\mathrm{MG}}}(a_*+b_*),
    \qquad
    \beta:=\sum_{*\in\mc B_{\mathrm A}}(a_*+b_*),
    \qquad
    \mu:=\sum_{*\in\mc B_{\mathrm{MG}}}
    \left(
        a_*\mathbf 1_{\HH_*'(z_*^-)\ne0}
        +b_*\mathbf 1_{\HH_*'(z_*^+)\ne0}
    \right),
\eeq 
where $\HH_*$ denotes the $L^{2/3}$-order phase of $\ff_{L,*}$.

\begin{prop}\label{result:generalasy}
There exist $C,L_0>0$ such that, 
for every $\bn=(n_1,n_2)\in\N^2$ and $\bsigma,\btau\in\listn$, 
if the contours in $\J^{\bsigma}_{\btau}(L)$ can be deformed, without crossing poles, to a $d_L$-separated admissible contour system, then 
\beqq
    |\J^{\bsigma}_{\btau}(L)|
    \le
    \frac{
        C^{n_1+n_2}e^{C\mu L^{1/3}}
        \sqrt{n_1!n_2!}\,(a_1+b_2)!
    }{
        d_L^{\alpha+\beta}
        L^{\alpha/2+\beta/3}
    }
    \prod_{*\in\mc A_2}
    \frac{|\ff_{L,*}(z_*^-)|^{a_*}}
    {|\ff_{L,*}(z_*^+)|^{b_*}} 
\eeqq
for all $L\ge L_0$.
The same bound holds if some Gaussian or modified Gaussian families have leading phase of Case~\textnormal{(b)} or~\textnormal{(c)}, provided that $a_*=0$ for every Case~\textnormal{(b)} family and $b_*=0$ for every Case~\textnormal{(c)} family. For such families, $z_*^+$ denotes the critical point in Case~\textnormal{(b)} and $z_*^-$ the critical point in Case~\textnormal{(c)}, and the factors involving nonexistent critical points are omitted.
\end{prop}

\begin{proof}
We deform the contours to an admissible contour system. Then
\beqq
    |\J^{\bsigma}_{\btau}(L)|
    \le
    \|\nPi^{\bsigma}_{\btau}\|_{\infty}
    \|\FF^{\bsigma}_{\btau,L}\|_{1}.
\eeqq
By Lemma~\ref{result:Cdetest},
\beqq
    \|\nPi^{\bsigma}_{\btau}\|_{\infty}
    \le
    d_L^{-(n_1+a_1+b_2+n_2)}
    n_1^{n_1/2}(a_1+b_2)^{(a_1+b_2)/2}n_2^{n_2/2}.
\eeqq
Note that $n_1+a_1+b_2+n_2 = \sum_{*\in \mc A_2} (a_*+b_*)
= \alpha + \beta$.
Since $(a_1+b_2)^{(a_1+b_2)/2}\le(a_1+b_2)^{a_1+b_2}$ and $k^k\le e^k k!$, there exists $C_0>0$ such that
\beqq
    \|\nPi^{\bsigma}_{\btau}\|_{\infty}
    \le
    C_0^{n_1+n_2}
    d_L^{-(\alpha+\beta)}
    \sqrt{n_1!n_2!}\,(a_1+b_2)!.
\eeqq
On the other hand, Lemmas~\ref{lem:asymptotics_f2}, \ref{lem:modifiedGaussian}, and~\ref{lem:asymptotics_f2_case2} imply that there exists $C_1>0$, independent of $L$, $\bsigma$, and $\btau$, such that
\beqq
    \|\FF^{\bsigma}_{\btau,L}\|_{1}
    \le
    \frac{
        C_1^{\sum_{*\in\mc A_2}(a_*+b_*)}
        e^{C_1\mu L^{1/3}}
    }{
        L^{\alpha/2+\beta/3}
    }
    \prod_{*\in\mc A_2}
    \frac{|\ff_{L,*}(z_*^-)|^{a_*}}
    {|\ff_{L,*}(z_*^+)|^{b_*}}.
\eeqq
Since $ \sum_{*\in\mc A_2}(a_*+b_*)\le2(n_1+n_2)$, 
the result follows by setting $C=\max\{ C_0C_1^2, C_1\}$.
\end{proof}

%%%%%%%%%%%%%%%%%%%%%%%%%%%%%%%
%%%%%%%%%%%%%%%%%%%%%%%%%%%%%%%
\section{The conditioning event}
\label{sec:conditioning}

Before turning to the conditional two-point asymptotics, we first state a sharp asymptotic formula for the one-point upper-tail probability that appears as the conditioning denominator. A closely related sharp asymptotic formula for the density appears in \cite[Lemma~6.3]{Baik-Cordaro-Tripathi25}.

\begin{lem}\label{lem:conditioning_event}
For $\ac,\bc>0$ and $\lv>(\sqrt \ac+\sqrt \bc)^2$, let 
\beq
	\tP:=\prob\left(\LPP(\lceil \ac L\rceil,\lceil \bc L\rceil)>\lv L\right). 
\eeq 
Then, as $L\to\infty$,
\beq\label{eq:conditioning_rank_one_asymptotic}
	\tP
	=
	\frac{\ffz(\tzo^-)}{\ffz(\tzo^+)}
	\frac{1}{2\pi L(\tzo^+-\tzo^-)^2
	\sqrt{-\Thetao''(\tzo^-)\Thetao''(\tzo^+)}}
	\bigl(1+o(1)\bigr),
\eeq
where $\Thetao$ and $\tzo^\pm$ are defined in Section~\ref{sec:phasefunctions}, and 
\beqq
	\ffz(z):=e^{-\lceil \ac L\rceil\log(1+z)+\lceil \bc L\rceil\log z+\lv Lz}. 
\eeqq
\end{lem}

\begin{proof}
Setting $M_L:=\lceil \ac L\rceil$ and $N_L:=\lceil \bc L\rceil$, by Proposition~\ref{prop:tail} with $m=1$,
\beq\label{eq:conditioning_onepoint_series}
	\tP=\sum_{n=1}^\infty\frac{\QQ_{1,L}^{(n)}}{(n!)^2},
	\quad
	\QQ_{1,L}^{(n)}
	=\frac{(-1)^{n+1}}{(2\pi\ii)^{2n}}
	\int_{\gamma^n}\dd\xi_1\cdots\dd\xi_n
	\int_{\Gamma^n}\dd\eta_1\cdots\dd\eta_n\,
	\K_n(\bseta\mid\bsxi)\K_n(\bsxi\mid\bseta)
	\prod_{j=1}^n\frac{\ffz(\xi_j)}{\ffz(\eta_j)}.
\eeq
Here $\gamma$ and $\Gamma$ are disjoint counterclockwise circles around $-1$ and $0$, respectively.

The leading contribution to \eqref{eq:conditioning_onepoint_series} comes from the term indexed by $n=1$. 
We deform $\gamma$ and $\Gamma$ to the curves $\cont^L_-(\tzo^-)$ and $\cont^L_+(\tzo^+)$, respectively.
By Lemma~\ref{lem:asymptotics_f2}, only neighborhoods of size $L^{-1/2+\epsilon/3}$ around $\tzo^\pm$ contribute. In these neighborhoods, using the local variables
\beqq
	\xi=\tzo^-+\frac{u}{L^{1/2}},
	\qquad
	\eta=\tzo^++\frac{v}{L^{1/2}},
\eeqq
applying Lemmas~\ref{lem:asymptotics_f1} and~\ref{lem:asymptotics_f2}, and then evaluating the resulting Gaussian integrals, we obtain
\beq\label{eq:conditioning_rank_one}
	\QQ_{1,L}^{(1)}
	=
	\frac{\ffz(\tzo^-)}{\ffz(\tzo^+)}
	\frac{1}{2\pi L(\tzo^+-\tzo^-)^2
	\sqrt{-\Thetao''(\tzo^-)\Thetao''(\tzo^+)}}
	\bigl(1+o(1)\bigr).
\eeq

Lemma~\ref{result:Cdetest} and the $L^1$-bounds in Lemma~\ref{lem:asymptotics_f2} imply that there exist $C_1, C_2,L_0>0$ such that, for all  $n\ge2$ and $L\ge L_0$,
\beqq
	\left|\QQ_{1,L}^{(n)}\right|
	\le \frac{C_1^n n^n}{L^n}
	\left|\frac{\ffz(\tzo^-)}{\ffz(\tzo^+)}\right|^n
	\le \frac{C_2^n n^n}{L^n} e^{-n \tDel L},
\eeqq
where $\tDel=\Thetao(\tzo^+)-\Thetao(\tzo^-)>0$ by Lemma~\ref{result:phasegaps}.
Combining this with \eqref{eq:conditioning_onepoint_series} and the bound
$n^n\le e^n n!$, we find that the sum over $n\ge 2$ is of order
$O\left(L^{-2}e^{-2\tDel L}\right)$. Hence, the result follows.
\end{proof}

For positive sequences $(a_L)_{L>0}$ and $(b_L)_{L>0}$, we write
$a_L \asymp b_L$ if there exist constants $c,C>0$ such that
$c \leq \frac{a_L}{b_L} \leq C$ 
for all sufficiently large $L$. The above lemma implies the following. 

\begin{cor}\label{result:tPldp}
As $L\to\infty$,  
\beq\label{eq:tPldp}
	\tP\asymp L^{-1}e^{-\tDel L},
	\qquad
	\tDel =\Thetao(\tzo^+)-\Thetao(\tzo^-)>0.
\eeq
\end{cor}

%%%%%%%%%%%%%%%%%%%%%%%%%%%%%%%
%%%%%%%%%%%%%%%%%%%%%%%%%%%%%%%
\section{Overview of the asymptotic analysis}
\label{sec:strategy}

We describe the common strategy used in the next four sections. 
Fix a point $(x,y)$. 
Let 
\beq \label{eq:xygen}
	(x_L, y_L)
	=
	\begin{cases}
		(x, y)
		&\text{in $U_1$ and $U_2$,}\\
		(x, y) + O(L^{-1/3}) 
		&\text{on the $U_1/U_2$ and $U_2/U_3$ boundaries,}
	\end{cases}
\eeq
and 
\beq\label{eq:Tgen}
	\tTo
	=
	\begin{cases}
		\mv(x,y) L+\sdev(x,y) \rr L^{1/3}
		&\text{in $U_1$ and on the $U_1/U_2$ and $U_2/U_3$ boundaries,}\\
		\mv(x,y) L+\sdev(x,y) \rr L^{1/2}
		&\text{in $U_2$,}
	\end{cases}
\eeq
where the precise perturbations $(x_L,y_L)$ will be specified in the corresponding sections.

The two-point formula in Proposition~\ref{prop:tail} requires the time parameters to be ordered. Thus, when $\mv(x,y)>\lv$, we take
\beq\label{eq:strategy_order_greater}
\begin{split}
    (M_{1,L},N_{1,L},T_{1,L})
    &=(\lceil \ac L\rceil,\lceil \bc L\rceil,\lv L),\\
    (M_{2,L},N_{2,L},T_{2,L})
    &=(\lceil \ac x_L L\rceil,\lceil \bc y_L L\rceil,\tTo),
\end{split}
\eeq
whereas, when $\mv(x,y)<\lv$, we reverse the order:
\beq\label{eq:strategy_order_less}
\begin{split}
    (M_{1,L},N_{1,L},T_{1,L})
    &=(\lceil \ac x_L L\rceil,\lceil \bc y_L L\rceil,\tTo),\\
    (M_{2,L},N_{2,L},T_{2,L})
    &=(\lceil \ac L\rceil,\lceil \bc L\rceil,\lv L).
\end{split}
\eeq
In either case, $T_{1,L}<T_{2,L}$ for all sufficiently large $L$. 
Proposition~\ref{prop:tail} implies that 
\beq\label{eq:strategy_conditional_probability}
	\prob\left[ \LPP(\lceil \ac x_L L\rceil, \lceil \bc y_L L\rceil)> \tTo \,\big|\,\LPP(\lceil \ac L\rceil,\lceil \bc L\rceil)>\lv L \right]
	=\frac{\QQ_{2,L}}{\tP}, 
\eeq
where $\tP$ is the conditioning probability defined in Lemma~\ref{lem:conditioning_event}, and
\beq \label{eq:Q2Lses}
\QQ_{2,L}
	:=
	\sum_{n_1,n_2\ge1}
	\frac{1}{(n_1!n_2!)^2} \QQ_{2,L}^{(n_1,n_2)}.
\eeq
The subscript $L$ records the dependence on the large parameter. 
After evaluating this upper-tail probability, we take the complement to obtain the distribution functions stated in Section~\ref{sec:results}. 

We make the dependence on $L$ explicit and write the functions in \eqref{eq:ff1212def} as
\beq \label{eq:fGen}
	\ff_{1,L}(z)=\frac{z^{N_{1,L}}e^{T_{1,L} z}}{(z+1)^{M_{1,L}}},\quad
	\ff_{2,L}(z)=\frac{z^{N_{2,L}-N_{1,L}}e^{(T_{2,L}-T_{1,L})z}}{(z+1)^{M_{2,L}-M_{1,L}}} ,\quad
	\ff_{12,L}(z)=\frac{z^{N_{2,L}}e^{T_{2,L}z}}{(z+1)^{M_{2,L}}}.
\eeq
Note that $\ff_{12,L}(z):=\ff_{1,L}(z)\ff_{2,L}(z)$ by definition. 
The leading phases of these functions are
\beq\label{eq:strategy_phase_functions}
    (\mcG_1,\mcG_2,\mcG_{12})
    :=
    \begin{cases}
        (\Thetao,\Thetar,\Thetat),&\text{if $\mv(x,y)>\lv$,}\\
        (\Thetat,-\Thetar,\Thetao),&\text{if $\mv(x,y)<\lv$,}
    \end{cases}
\eeq
where the functions $\Thetao$, $\Thetat$, and $\Thetar$ are defined in \eqref{eq:allphasefunctions}.
The lower-order terms in the exponents occur at scale $L^{1/3}$ in $U_1$, at scale $L^{1/2}$ in $U_2$, and at scales $L^{2/3}$ and $L^{1/3}$ on the $U_1/U_2$ and $U_2/U_3$ boundaries.
For each $*\in \{1, 2, 12\}$, the critical points of $\mcG_*$ are denoted by $\rz_*^\pm$. 
By \eqref{eq:strategy_phase_functions}, these critical points are related to those of $\Theta_i$ from Section~\ref{sec:cp} by
\beq\label{eq:strategy_cp}
\begin{cases}
    \rzo^\pm=\tzo^\pm,\quad \rzot^\pm=\tzt^\pm,
    &\text{if $\mv(x,y)>\lv$,}\\
    \rzo^\pm=\tzt^\pm,\quad \rzot^\pm=\tzo^\pm,
    &\text{if $\mv(x,y)<\lv$.}
\end{cases}
\eeq
The remaining critical points $\rzt^\pm$ are the critical points of $\Thetar$.

For the regimes involving $U_2$, by symmetry it suffices to consider $y<x$. The case $y>x$ follows by simultaneously exchanging $\ac\leftrightarrow\bc$ and $x\leftrightarrow y$. Accordingly, in the discussion and tables below, we restrict attention to $U_2^<$ and its boundaries. The case $\mv=\lv$ is treated separately in Subsection~\ref{sec:U2equal} and at the end of Section~\ref{sec:U2U3}, for the $U_2$ region and the $U_2/U_3$ boundary, respectively.

By Lemma~\ref{lem:simpleQ}, each $\QQ_{2,L}^{(n_1,n_2)}$ is a linear combination of the integrals $\J^{\bsigma}_{\btau}$. The original contour nesting is not always compatible with the contour ordering required for steepest-descent analysis, so we apply the contour-rearrangement identities in Lemma~\ref{result:originalintexprss}. The appropriate rearrangement depends on the location of $(x,y)$ through the ordering of the critical points of $\Thetao$, $\Thetat$, and $\Thetar$.
Table~\ref{tab:rearrangement} records the contour-rearrangement identities used in each regime, together with the corresponding version of Corollary~\ref{result:Jintrearboundgen} used to estimate the subleading terms. In some regimes, alternative rearrangements are possible; throughout the paper, we use the choices listed in the table.

\begin{table}[h]
\centering
\begin{tabular}{|l|l|l|l|}
\hline
\textbf{Regime} & \textbf{$\bsigma$-rearrangement} & \textbf{$\btau$-rearrangement} & \textbf{Estimate}\\
\hline
$U_1$
& \eqref{eq:Jup21}
& \eqref{eq:Jdown21}
& $S^{12}_{12}$\\
\hline
Boundary $U_1/U_2^<$
& \eqref{eq:Jup21}
& \eqref{eq:Jdown21}
& $S^{12}_{12}$\\
\hline
$U_2^<$ with $\mv>\lv$
& \eqref{eq:Jup21}
& \eqref{eq:Jdown12}
& $S^{12}_{21}$\\
\hline
$U_2^<$ with $\mv<\lv$
& \eqref{eq:Jup12}
& \eqref{eq:Jdown12}
& $S^{21}_{21}$\\
\hline
Boundary $U_2^</U_3^<$ with $\mv>\lv$
& \eqref{eq:Jup21}
& \eqref{eq:Jdown12}
& $S^{12}_{21}$\\
\hline
Boundary $U_2^</U_3^<$ with $\mv<\lv$
& \eqref{eq:Jup12}
& \eqref{eq:Jdown21}
& $S^{21}_{12}$\\
\hline
\end{tabular}
\caption{Contour-rearrangement identities and estimates used in each regime.}
\label{tab:rearrangement}
\end{table}

After these rearrangements, $\QQ_{2,L}$ is expressed as a sum of integrals. Table~\ref{tab:leadingterms} lists the indices $(n_1,n_2)$ that contribute at leading order in each regime.
The leading contributions are evaluated by steepest descent. The remaining terms are estimated using Proposition~\ref{result:generalasy} by comparing their exponential weights with the conditioning cost $\tDel$ in \eqref{eq:tPldp}. Each such term carries an additional strictly positive phase gap. In several cases, we must be careful with the $L^{2/3}$- or $L^{1/2}$-order lower-order terms in the exponents of $\ff_{*,L}$ and the separation between the contours.

Combining these phase-gap estimates with Corollary~\ref{result:Qn1n2bd}, \eqref{eq:dcest}, and \eqref{eq:triplfactorial}, we show in each regime that the total contribution of all nonleading terms is exponentially smaller than the conditioning probability. Thus, the asymptotic limit is determined entirely by the leading terms listed in Table~\ref{tab:leadingterms}.

\begin{table}[h]
\centering
\begin{tabular}{|l|l|l|}
\hline
\textbf{Regime} & \textbf{Leading terms} & \textbf{Limiting distribution}\\
\hline
$U_1$
& $(1,n)$, $n\in\N$
& GUE Tracy-Widom\\
\hline
Boundary $U_1/U_2^<$
& $(1,n)$, $n\in\N$
& one-spike BBP\\
\hline
$U_2^<$
& $(1,1)$
& Gaussian\\
\hline
Boundary $U_2^</U_3^<$ with $\mv>\lv$
& $(n,n)$ and $(n+1,n)$, $n\in\N$
& one-spike BBP\\
\hline
Boundary $U_2^</U_3^<$ with $\mv<\lv$
& $(n,1)$, $n\in\N$
& one-spike BBP\\
\hline
\end{tabular}
\caption{Leading terms and limiting distributions in each regime.}
\label{tab:leadingterms}
\end{table}

%%%%%%%%%%%%%%%%%%%%%%%%%%%%%%%%%%%%%%%%%
%%%%%%%%%%%%%%%%%%%%%%%%%%%%%%%%%%%%%%%%%
\section{Region $U_1$}\label{sec:case1}

We prove Theorem~\ref{result:CCLT} for $(x,y)\in U_1$. Fix $(x,y)\in U_1$ and $\rr\in\R$, and set $\mv:=\mv(x,y)$ and $\sdev:=\sdev(x,y)$. Since $\mv>\lv$, we use the ordering in \eqref{eq:strategy_order_greater}, with $(x_L,y_L)=(x,y)$ and $\tTo=\mv L+\sdev\rr L^{1/3}$ in \eqref{eq:xygen} and \eqref{eq:Tgen}. With this choice, the functions in \eqref{eq:fGen} are
\beq\label{eq:case1_f_expansions}
\begin{split}
    \ff_{1,L}(z)=e^{L\Gone(z)+E_{1,L}(z)},\quad
    \ff_{2,L}(z)=e^{L\Gtwo(z)+\sdev\rr L^{1/3}z+E_{2,L}(z)},\quad
    \ff_{12,L}(z)=e^{L\Gonetwo(z)+\sdev\rr L^{1/3}z+E_{12,L}(z)}.
\end{split}
\eeq
Their leading phases are
\beqq
    (\mcG_1,\mcG_2,\mcG_{12})=(\Thetao,\Thetar,\Thetat).
\eeqq
Note that $\ff_{1,L}=\ffz$ in Lemma~\ref{lem:conditioning_event}.

By \eqref{eq:strategy_cp} and Lemma~\ref{result:cpU1},
\beq\label{eq:U1cprelation}
    \rzo^\pm=\tzo^\pm,\qquad \rzt^-=\rzt^+=\tzr^-=\tzr^+=\tzr,\qquad \rzot^\pm=\tzt^\pm
\eeq
and
\beq\label{eq:case1_critical_order}
    -1<\rzo^-<\rzot^-<\rzt<\rzot^+<\rzo^+<0 , \qquad \rzt:= \rzt^-=\rzt^+. 
\eeq
Thus, $\ff_{1,L}$ and $\ff_{12,L}$ are Gaussian families, while $\ff_{2,L}$ is an Airy family. A direct computation using \eqref{eq:U1cpformula} gives
\beqq
    \mcG_2'''(\rzt)=\Thetar'''(\tzr)=-2\sdev^3.
\eeqq

%%%%%%%%%%%%%%%%%%%%%%%%%%%%%%%%%%%%%%%%%%%%%%%%%
\subsection{The leading contribution in the $U_1$ region}

The leading contribution to the series \eqref{eq:Q2Lses} comes from 
the terms indexed by $(1,n)$ with $n\in \N$.

\begin{lem}\label{result:U1leading}
We have
\beqq
	\lim_{L\to\infty}
	\frac{1}{\tP}
	\sum_{n=1}^\infty
	\frac{1}{(n!)^2} \QQ_{2,L}^{(1,n)}
	=\prob(\TW_2>\rr).
\eeqq
\end{lem}

\begin{proof}
By Lemma~\ref{lem:simpleQ},
\beqq
	\QQ_{2,L}^{(1,n)}
	=(-1)^{n+1}\J^{12^n}_{12^n}
	=\frac{(-1)^{n+1}}{(2\pi\ii)^{2n+2}}
	\int\dd\xi^1 \dd\bsxi^2
	\int\dd\eta^1 \dd\bseta^2\,
	\nPi^{12^n}_{12^n}(\bsxi,\bseta)
	\frac{\ff_{1,L}(\xi^1)}{\ff_{1,L}(\eta^1)}
	\prod_{j=1}^n\frac{\ff_{2,L}(\xi_j^2)}{\ff_{2,L}(\eta_j^2)},
\eeqq
where $\bsxi^2=(\xi_1^2,\ldots,\xi_n^2)$ and $\bseta^2=(\eta_1^2,\ldots,\eta_n^2)$.
We deform the $\xi^1$-, $\xi_j^2$-, $\eta_j^2$-, and $\eta^1$-contours to
\beqq
	\cont^L_-(\rzo^-),
	\qquad
	\cont_-^L(\rzt-L^{-1/3}),
	\qquad
	\cont_+^L(\rzt +L^{-1/3}),
	\qquad
	\cont^L_+(\rzo^+),
\eeqq
respectively. By \eqref{eq:case1_critical_order}, these contours satisfy the required nesting for all sufficiently large $L$, and the deformation crosses no poles.

By Lemmas~\ref{lem:asymptotics_f2} and~\ref{lem:asymptotics_f2_case2}, only neighborhoods of size $L^{-1/2+\epsilon/3}$ around $\rzo^\pm$ and $L^{-1/3+\epsilon/4}$ around $\rzt$ contribute. In these neighborhoods, use the local variables
\beq\label{eq:case1_local_variables}
	\xi^1=\rzo^-+\frac{\widehat u}{L^{1/2}},
	\qquad
	\xi_j^2=\rzt +\frac{u_j}{\sdev L^{1/3}},
	\qquad
	\eta_j^2=\rzt +\frac{v_j}{\sdev L^{1/3}},
	\qquad
	\eta^1=\rzo^++\frac{\widehat v}{L^{1/2}}.
\eeq
Lemmas~\ref{lem:asymptotics_f1} and~\ref{result:dbcplocal} give 
\beq\label{eq:U1nFlimt}
	\frac{\ff_{1,L}(\xi^1)}{\ff_{1,L}(\eta^1)}
	\prod_{j=1}^n\frac{\ff_{2,L}(\xi_j^2)}{\ff_{2,L}(\eta_j^2)}
	=
	\frac{\ff_{1,L}(\rzo^-)}{\ff_{1,L}(\rzo^+)}
	\frac{e^{\mcG_1''(\rzo^-)\hat u^2/2}}{e^{\mcG_1''(\rzo^+)\hat v^2/2}}
	\prod_{j=1}^n
	\frac{e^{-u_j^3/3+\rr u_j}}{e^{-v_j^3/3+\rr v_j}}
	\bigl(1+o(1)\bigr),
\eeq
and
\beq\label{eq:U1nPilimi}
	\nPi^{12^n}_{12^n}(\bsxi,\bseta)
	=(-1)^{n+1}
	\frac{\K_n(\bsxi^2\mid\bseta^2)^2}{(\xi^1-\eta^1)^2}
	\prod_{j=1}^n
	\frac{(\xi^1-\eta_j^2)(\eta^1-\xi_j^2)}
	{(\xi^1-\xi_j^2)(\eta^1-\eta_j^2)}
	=(-1)^{n+1}(\sdev L^{1/3})^{2n}
	\frac{\K_n(\bu\mid\bv)^2}{(\rzo^--\rzo^+)^2}
	\bigl(1+o(1)\bigr).
\eeq

Thus, for each fixed $n\ge1$, the steepest-descent analysis and the bounds in Lemmas~\ref{lem:asymptotics_f2} and~\ref{lem:asymptotics_f2_case2} justify taking the limit. Evaluating the Gaussian integrals in the $\hat u$- and $\hat v$-variables and applying Lemma~\ref{lem:conditioning_event}, we obtain
\beq\label{eq:case1_leading_limit}
\begin{split}
    \lim_{L\to\infty}\frac{\J^{12^n}_{12^n}}{\tP}
    &=
    \frac{1}{(2\pi\ii)^{2n}}
    \int_{(\Sigma_{\rL})^n}\dd\bu
    \int_{(\Sigma_{\rR})^n}\dd\bv\,
    \K_n(\bu\mid\bv)^2
    \prod_{j=1}^n
    \frac{e^{-u_j^3/3+\rr u_j}}{e^{-v_j^3/3+\rr v_j}}.
\end{split}
\eeq

Lemma~\ref{result:Cdetest} and the $L^1$-bounds of Lemmas~\ref{lem:asymptotics_f2} and~\ref{lem:asymptotics_f2_case2} imply that there exists $C>0$ such that, uniformly in $n\ge1$ and all sufficiently large $L$,
\beq\label{eq:U1Hadam}
	\frac{|\J^{12^n}_{12^n}|}{\tP}\le C^n n^n.
\eeq
Thus, by dominated convergence,
\beqq
	\lim_{L\to\infty}
	\frac{1}{\tP}
	\sum_{n=1}^\infty
	\frac{\QQ_{2,L}^{(1,n)}}{(n!)^2}
	=
	\sum_{n=1}^\infty
	\frac{(-1)^{n+1}}{(n!)^2}
	\frac{1}{(2\pi\ii)^{2n}}
	\int_{(\Sigma_{\rL})^n}\dd\bu
	\int_{(\Sigma_{\rR})^n}\dd\bv\,
	\K_n(\bu\mid\bv)^2
	\prod_{j=1}^n
	\frac{e^{-u_j^3/3+\rr u_j}}{e^{-v_j^3/3+\rr v_j}}.
\eeqq
The series on the right is $\prob(\TW_2>\rr)$ by the Fredholm series for the Airy kernel.
\end{proof}

%%%%%%%%%%%%%%%%%%%%%%%%%%%%%%%%%%%%%%%%%%%%%%%%
\subsection{Remainder in the $U_1$ region}
\label{sec:U1remainder}

By Corollaries~\ref{result:Qn1n2bd} and~\ref{result:Jintrearboundgen}, using the $S^{12}_{12}$ version, we have
\beq\label{eq:Q1maxbound}
\begin{split}
	|\QQ_{2,L}^{(n_1,n_2)}|
	&\le
	2^{7(n_1+n_2)}
	\max_{\substack{0\le i,j\le n_2\\i+j\ge2n_2-n_1+1}}
	\sum_{a=0}^{n_1\wedge(n_2-i)}
	\sum_{b=0}^{n_1\wedge(n_2-j)}
	a!b!\,|\J^{12}_{12}(a,b,n_1,n_2)|,
\end{split}
\eeq
where 
\beq \label{eq:U1Jabnn}
	\J^{12}_{12}(a,b,n_1,n_2) =\J^{1^{n_1-a}(12)^a2^{n_2-a}}_{1^{n_1-b}(12)^b2^{n_2-b}}. 
\eeq

\begin{lem}\label{result:U1remaineach}
There exist $C,c,L_0>0$ such that, for all $n_1\ge2$, $n_2\ge1$, $0\le a,b\le n_1\wedge n_2$ satisfying $a+b\le n_1-1$, and $L\ge L_0$, 
\beq\label{eq:U1remaineach}
	|\J^{12}_{12}(a,b,n_1,n_2)|
	\le
	e^{-\tDel L}C^{n_1+n_2}\sqrt{n_1!n_2!}\,(n_1+n_2-a-b)!e^{-cn_1L}. 
\eeq
\end{lem}

\begin{proof}
We have
\beqq
	\type(1^{n_1-a}(12)^a2^{n_2-a})=(a,n_1-a,n_2-a),
	\qquad
	\type(1^{n_1-b}(12)^b2^{n_2-b})=(b,n_1-b,n_2-b).
\eeqq
Recall Definition~\ref{def:Jintegral} for the integrals $\J^{\bsigma}_{\btau}$. 
We deform the $\xi_i^1$-, $\xi_i^{12}$-, $\xi_i^2$-, $\eta_i^2$-, $\eta_i^{12}$-, and $\eta_i^1$-contours 
for $\J^{12}_{12}(a,b,n_1,n_2)$ to
\beqq
	\cont^L_-(\rzo^-),
	\qquad
	\cont^L_-(\rzot^-),
	\qquad
	\cont_-^L(\rzt -L^{-1/3}),
	\qquad
	\cont_+^L(\rzt +L^{-1/3}),
	\qquad
	\cont^L_+(\rzot^+),
	\qquad
	\cont^L_+(\rzo^+),
\eeqq
respectively. By \eqref{eq:case1_critical_order}, they form an admissible contour system for the integral. The relevant pairs of contours are separated by at least $c_*L^{-1/3}$ for some constant $c_*>0$. Proposition~\ref{result:generalasy}, with
\beqq
	d_L=c_*L^{-1/3},\qquad
	\alpha=a_1+b_1+a_{12}+b_{12}=2n_1,\qquad
	\beta=a_2+b_2=2n_2-a-b, 
\eeqq
implies, after absorbing powers of $c_*^{-1}$ into the constant, that there are $C_1,L_0>0$ such that, for all $L\ge L_0$,
\beq\label{eq:U1remaineach_general_bound}
	|\J^{12}_{12}(a,b, n_1,n_2)|
	\le
	C_1^{n_1+n_2}\sqrt{n_1!n_2!}\,(n_1+n_2-a-b)!L^{-n_1/3}E_L(a,b),
\eeq
where, using \eqref{eq:U1cprelation}, 
\beq \label{eq:U1EL}
	E_L(a,b)
	:=
	\frac{|\ff_{1,L}(\rzo^-)|^{n_1-a}
	|\ff_{12,L}(\rzot^-)|^a
	|\ff_{2,L}(\rzt)|^{n_2-a}}
	{|\ff_{1,L}(\rzo^+)|^{n_1-b}
	|\ff_{12,L}(\rzot^+)|^b
	|\ff_{2,L}(\rzt)|^{n_2-b}}
	=
	\frac{|\ff_{1,L}(\tzo^-)|^{n_1-a}
	|\ff_{12,L}(\tzt^-)|^a}
	{|\ff_{1,L}(\tzo^+)|^{n_1-b}
	|\ff_{12,L}(\tzt^+)|^b}
	|\ff_{2,L}(\tzr)|^{b-a}.
\eeq

By \eqref{eq:case1_f_expansions}, noting that $\ff_{1,L}$ has no $L^{1/3}$-order term in the exponent, there exist $C_2,c_1>0$ such that 
\beqq
	E_L(a,b)
    \le
   	C_2^{n_1+n_2}e^{c_1 (a+b) L^{1/3}}e^{L\Phi_{a,b}}
    \le
	C_2^{n_1+n_2}e^{2c_1n_1L^{1/3}}e^{L\Phi_{a,b}},
\eeqq
where 
\beqq
	\Phi_{a,b}
	:=(n_1-a)\Thetao(\tzo^-)+a\Thetat(\tzt^-)
	-(n_1-b)\Thetao(\tzo^+)-b\Thetat(\tzt^+)
	+(b-a)\Thetar(\tzr).
\eeqq
Using $\Thetar=\Thetat-\Thetao$ from \eqref{eq:tGide}, 
\beqq
	\Phi_{a,b}
	=-(n_1-a-b)\tDel
	-a(\delta_1^++\delta_{2}^-)
	-b(\delta_1^-+\delta_{2}^+),
\eeqq
where 
\beqq
\begin{aligned}
	\delta_1^-&:=\Thetao(\tzr)-\Thetao(\tzo^-),
	&\qquad
	\delta_1^+&:=\Thetao(\tzo^+)-\Thetao(\tzr),\\
	\delta_{2}^-&:=\Thetat(\tzr)-\Thetat(\tzt^-),
	&
	\delta_{2}^+&:=\Thetat(\tzt^+)-\Thetat(\tzr).
\end{aligned}
\eeqq
All four quantities are strictly positive because, for each $i=1,2$, $\Theta_i$ is strictly increasing on $[\tz_i^-,\tz_i^+]$, and \eqref{eq:case1_critical_order} implies that $\tzr\in(\tz_i^-,\tz_i^+)$.
Hence, since $a+b\le n_1-1$,
\beqq
	\Phi_{a,b}
	\le
	-\tDel-(n_1-1)c_0,
	\qquad
	c_0:=\min\{\tDel,\delta_1^++\delta_{2}^-,\delta_1^-+\delta_{2}^+\}>0.
\eeqq
Thus,
\beqq
	E_L(a,b)
	\le
	e^{-\tDel L}C_2^{n_1+n_2}
	e^{2c_1n_1L^{1/3}-(n_1-1)c_0L}
	\le
	e^{-\tDel L}C_2^{n_1+n_2}e^{-cn_1L}
\eeqq
for some $c>0$, since $n_1\ge2$.
Combining this estimate with \eqref{eq:U1remaineach_general_bound} proves the result.
\end{proof}

\begin{cor}\label{prop:case1_remainder}
There exists $c>0$ such that
\beqq
	\frac{1}{\tP}
	\sum_{n_1=2}^\infty\sum_{n_2=1}^\infty
	\frac{|\QQ_{2,L}^{(n_1,n_2)}|}{(n_1!n_2!)^2}
	=O(e^{-cL}).
\eeqq
\end{cor}

\begin{proof}
For the indices appearing in \eqref{eq:Q1maxbound}, 
since $a\le n_2-i$, $b\le n_2-j$, and $i+j\ge2n_2-n_1+1$, we have $a+b\le n_1-1$. Thus, by Lemma~\ref{result:U1remaineach} and \eqref{eq:triplfactorial}, 
there are $C, c>0$ such that 
\beqq
	\sum_{n_1=2}^\infty\sum_{n_2=1}^\infty
	\frac{|\QQ_{2,L}^{(n_1,n_2)}|}{(n_1!n_2!)^2}
	\le
	e^{-\tDel L}
	\sum_{n_1=2}^\infty\sum_{n_2=1}^\infty
	\frac{C^{n_1+n_2}e^{-cn_1L}}{\sqrt{n_1!n_2!}}
	=O(e^{-(\tDel+c)L}).
\eeqq
Since $\tP\asymp L^{-1}e^{-\tDel L}$ by Corollary~\ref{result:tPldp}, the ratio is $O(Le^{-cL})=O(e^{-c'L})$ for some $c'>0$, and the result follows.
\end{proof}

From \eqref{eq:strategy_conditional_probability}, Lemma~\ref{result:U1leading}, and Corollary~\ref{prop:case1_remainder}, we find
\beqq
	\lim_{L\to\infty}
	\prob\left[
		\frac{\LPP(\ac x  L, \bc y L)-\mv L}{\sdev L^{1/3}}>\rr
		\,\bigg|\,
		\LPP(\ac L,\bc L)>\lv L
	\right]
	=\prob(\TW_2>\rr).
\eeqq
Thus, Theorem~\ref{result:CCLT} for $(x,y)\in U_1$ is proved.

%%%%%%%%%%%%%%
%%%%%%%%%%%%%%
\section{Boundary between $U_1$ and $U_2$}
\label{sec:U1U2}

Let $(x,y)\in (1,\infty)^2$ lie on the boundary between $U_1$ and $U_2$. 
We consider the case $y<x$ and prove Theorem~\ref{result:crossdist}\textnormal{(a)}. 
The case $y>x$ follows by symmetry. 
Fix $\ww,\rr\in\R$, and set $\mv:=\mv(x,y)$ and $\sdev:=\sdev(x,y)$. 
Since $\mv>\lv$, we use the ordering in \eqref{eq:strategy_order_greater}, where, with $\mr c$ in \eqref{eq:littlec}, 
\beqq
	(x_L, y_L)
	= (x,y) +  \mr c\ww (y-1)^{2/3}L^{-1/3} 
	\left( \frac{\sqrt\slope}{\sqrt \ac},
	- \frac{1}{\sqrt \bc} \right)
\eeqq
in \eqref{eq:xygen} and $\tTo=\mv L+\sdev \rr L^{1/3}$ in \eqref{eq:Tgen}. 
The functions in \eqref{eq:fGen} are
\beq\label{eq:U1U2_f_expansions}
\begin{split}
	\ff_{1,L}(z)&=e^{L\Gone(z)+E_{1,L}(z)},\\
	\ff_{2,L}(z)&=e^{L\Gtwo(z)+\mr c\ww(y-1)^{2/3}L^{2/3}\mcH(z)+\sdev\rr L^{1/3}z+E_{2,L}(z)},\\
	\ff_{12,L}(z)&=e^{L\Gonetwo(z)+\mr c\ww(y-1)^{2/3}L^{2/3}\mcH(z)+\sdev\rr L^{1/3}z+E_{12,L}(z)},
\end{split}
\eeq
where
\beqq
	(\mcG_1,\mcG_2,\mcG_{12}) = (\Thetao,\Thetar,\Thetat), 
\eeqq
and 
\beq \label{eq:HdefinU1U2}
	\mcH(z):=-\sqrt{\slope \ac}\log(z+1)-\sqrt \bc\log z.
\eeq

By \eqref{eq:strategy_cp} and Lemma~\ref{result:cpU1},
\beqq 
	\rzo^\pm= \tzo^\pm, \qquad \rzt^-=\rzt^+=\tzr^-=\tzr^+= \tzr , \qquad \rzot^\pm=\tzt^\pm
\eeqq
and
\beq\label{eq:U1U2_critical_order}
	-1<\rzo^-<\rzot^-<\rzt = \rzot^+ = \rzo^+=\pp <0, \qquad \rzt:= \rzt^-=\rzt^+. 
\eeq
Direct computations show that 
\beqq
	\mcG_2'''(\pp) =\Thetar'''(\pp)=-2\sdev^3,
	\qquad
	\mcH'(\pp)=0,
	\qquad
	\mcH''(\pp)=\frac{2\sdev^2}{\mr c\,(y-1)^{2/3}}.
\eeqq
The functions $\ff_{1,L}$, $\ff_{2,L}$, and $\ff_{12,L}$ are a Gaussian family, an Airy family, and a modified Gaussian family, respectively.

%%%%%%%%%%%%%%%%%%%%%%%%%%%%%%%%%%%%%%%%%%%%%%%%%
\subsection{The leading contribution on the boundary between $U_1$ and $U_2^<$}

The proof of the next result is similar to that of Lemma~\ref{result:U1leading} for the region $U_1$. 
We indicate only the changes from that proof.

\begin{lem}\label{result:U12Flimit}
We have
\beqq
	\lim_{L\to\infty}
	\frac{1}{\tP}
	\sum_{n=1}^\infty\frac{\QQ_{2,L}^{(1,n)}}{(n!)^2}
	=1-F_{\tn{BBP},\ww}(\rr+\ww^2).
\eeqq
\end{lem}

\begin{proof} 
This time we deform the $\xi^1$-, $\xi_i^2$-, $\eta_i^2$-, and $\eta^1$-contours for the integral $\J^{12^n}_{12^n}$ to
\beqq
	\cont^L_-(\rzo^-),
	\qquad
	\cont_-^L(\pp-2L^{-1/3}),
	\qquad
	\cont_+^L(\pp-L^{-1/3}),
	\qquad
	\cont^L_+(\pp),
\eeqq
respectively. By \eqref{eq:U1U2_critical_order}, these contours form an admissible contour system, and the deformation crosses no poles.
Under the local variables \eqref{eq:case1_local_variables} with $\rzt=\pp$, equation \eqref{eq:U1nFlimt} changes to 
\beqq
	\frac{\ff_{1,L}(\xi^1)}{\ff_{1,L}(\eta^1)}
	\prod_{i=1}^n\frac{\ff_{2,L}(\xi_i^2)}{\ff_{2,L}(\eta_i^2)}
	=
	\frac{\ff_{1,L}(\rzo^-)}{\ff_{1,L}(\rzo^+)}
	\frac{e^{\mcG_1''(\rzo^-) \hat u^2/2}}{e^{\mcG_1''(\rzo^+) \hat v^2/2}}
	\prod_{i=1}^n
	\frac{e^{-u_i^3/3+\ww u_i^2+\rr u_i}}
	{e^{-v_i^3/3+\ww v_i^2+\rr v_i}}
	\bigl(1+o(1)\bigr).
\eeqq
Equation~\eqref{eq:U1nPilimi} similarly becomes 
\beqq
	\nPi^{12^n}_{12^n}(\bsxi,\bseta)
	=(-1)^{n+1}(\sdev L^{1/3})^{2n}
	\frac{\K_n(\bu\mid\bv)^2}{(\rzo^--\rzo^+)^2}
	\left(\prod_{i=1}^n\frac{u_i}{v_i}\right)
	\bigl(1+o(1)\bigr).
\eeqq
Thus, \eqref{eq:case1_leading_limit} changes to 
\beqq
\begin{split}
	\lim_{L\to\infty}\frac{\J^{12^n}_{12^n}}{\tP}
	&=\frac{1}{(2\pi\ii)^{2n}}
	\int_{(\Sigma_{\rL})^n}\dd\bu
	\int_{(\Sigma_{\rR})^n}\dd\bv\,
	\K_n(\bu\mid\bv)^2
	\prod_{i=1}^n
	\frac{e^{-u_i^3/3+\ww u_i^2+\rr u_i}u_i}
	{e^{-v_i^3/3+\ww v_i^2+\rr v_i}v_i},
\end{split}
\eeqq
where $0$ lies to the right of $\Sigma_{\rR}$. The same Cauchy-determinant and $L^1$ estimates used to prove \eqref{eq:U1Hadam} give a bound $C^n n^n$ uniformly in $n$ and all sufficiently large $L$, so dominated convergence applies. After the changes of variables $u_i\mapsto\ww-v_i$ and $v_i\mapsto\ww-u_i$ and reversal of the resulting contour orientations, the new left contour lies to the right of $\ww$. The result then follows from \eqref{eq:BBPFredholmexpnas}.
\end{proof}

%%%%%%%%%%%%%%%%%%%%%%%%%%%%%%%%%%%%%%%%%%%%%%%%%%%%%
\subsection{Remainder on the boundary between $U_1$ and $U_2^<$}

We use the same bound \eqref{eq:Q1maxbound}. Recall \eqref{eq:U1Jabnn}. 

\begin{lem}\label{result:U12remaineach}
There exist $C,c,L_0>0$ such that, for all $n_1\ge 2$, $n_2\ge1$, 
$0\le a,b\le n_1\wedge n_2$ satisfying $a\le n_1-1$, and $L\ge L_0$, 
\beqq
	|\J^{12}_{12}(a,b,n_1,n_2)|
	\le
	e^{-\tDel L} C^{n_1+n_2}\sqrt{n_1!n_2!}\,(n_1+n_2-a-b)!e^{-c n_1 L}. 
\eeqq
\end{lem}

\begin{proof}
The proof is similar to that of Lemma~\ref{result:U1remaineach}. We deform the $\xi_i^1$-, $\xi_i^{12}$-, $\xi_i^2$-, $\eta_i^2$-, $\eta_i^{12}$-, and $\eta_i^1$-contours to
\beqq
    \cont_-^L(\rzo^-),\qquad \cont_-^L(\rzot^-),\qquad \cont_-^L(\pp-2L^{-1/3}),\qquad \cont_+^L(\pp-L^{-1/3}),\qquad \cont_+^L(\pp+L^{-1/2}),\qquad \cont_+^L(\pp+2L^{-1/2}),
\eeqq
respectively. By \eqref{eq:U1U2_critical_order}, these contours form an admissible contour system. Although the $\eta^{12}$- and $\eta^1$-contours are only of order $L^{-1/2}$ apart, these variables never appear on opposite sides of a Cauchy determinant in \eqref{eq:Caucdt}. Every relevant pair of contours is separated by at least $c_*L^{-1/3}$ for some constant $c_*>0$. Hence, Proposition~\ref{result:generalasy} applies with
\beqq
    d_L=c_*L^{-1/3},\qquad \alpha=a_1+b_1+a_{12}+b_{12}=2n_1,\qquad \beta=a_2+b_2=2n_2-a-b.
\eeqq
In this case, $\ff_{12,L}$ is a modified Gaussian family. Since $\mcH'(\pp)=0$, while in general $\mcH'(\rzot^-)\ne0$, the definition in \eqref{eq:abmu} gives $\mu=a$. Thus, \eqref{eq:U1remaineach_general_bound} holds with an additional multiplicative factor $e^{CaL^{1/3}}$.

Setting $\rzt=\rzo^+=\rzot^+=\pp$ and using $\ff_{2,L}=\ff_{12,L}/\ff_{1,L}$, the quantity in \eqref{eq:U1EL} simplifies to
\beqq
    E_L(a,b):=\frac{|\ff_{1,L}(\rzo^-)|^{n_1-a}|\ff_{12,L}(\rzot^-)|^a|\ff_{2,L}(\pp)|^{n_2-a}}{|\ff_{1,L}(\pp)|^{n_1-b}|\ff_{12,L}(\pp)|^b|\ff_{2,L}(\pp)|^{n_2-b}}=\frac{|\ff_{1,L}(\tzo^-)|^{n_1-a}|\ff_{12,L}(\tzt^-)|^a}{|\ff_{1,L}(\tzo^+)|^{n_1-a}|\ff_{12,L}(\tzt^+)|^a}=:E_L(a).
\eeqq
By \eqref{eq:U1U2_f_expansions}, and noting that $\ff_{1,L}$ has no $L^{2/3}$- or $L^{1/3}$-order terms in the exponent, there exist $C_2,c_1>0$ such that
\beqq
    E_L(a)\le C_2^{n_1+n_2}e^{c_1aL^{2/3}}e^{L\Phi_a},\qquad \Phi_a:=-(n_1-a)\tDel-a\tDelt,
\eeqq
where $\tDel,\tDelt>0$ by Lemma~\ref{result:phasegaps}. Since $n_1-a-1\ge0$ by assumption, 
\beqq
    \Phi_a\le-\tDel-(n_1-1)c_0,\qquad c_0:=\min\{\tDel,\tDelt\}.
\eeqq
Therefore,
\beqq
    E_L(a)\le e^{-\tDel L}C_2^{n_1+n_2}e^{c_1 a L^{2/3}-(n_1-1)c_0L}.
\eeqq

Thus, the estimate 
\eqref{eq:U1remaineach_general_bound} with an additional multiplicative factor $e^{CaL^{1/3}}$ implies that 
\beqq
	|\J^{12}_{12}(a,b, n_1,n_2)|
	\le
	e^{-\tDel L} 
    (C_1C_2)^{n_1+n_2}\sqrt{n_1!n_2!}\,(n_1+n_2-a-b)!L^{-n_1/3} e^{CaL^{1/3}} e^{c_1 a L^{2/3}-(n_1-1)c_0L}. 
\eeqq
Since $a\le n_1$ and $n_1\ge2$, there exists $c>0$ such that  $e^{CaL^{1/3}}e^{c_1 a L^{2/3}-(n_1-1)c_0L}\le e^{-c n_1L}$ for all sufficiently large $L$. Hence, the result follows.
\end{proof}

\begin{cor}\label{result:U1U2_remainder}
There exists $c>0$ such that
\beqq
	\frac{1}{\tP}
	\sum_{n_1\ge2}\sum_{n_2\ge1}
	\frac{1}{(n_1!n_2!)^2} |\QQ_{2,L}^{(n_1,n_2)}(\rr,\ww)|
	=O(e^{-cL}).
\eeqq
\end{cor}

\begin{proof}
The proof is the same as that of Corollary~\ref{prop:case1_remainder}, using Lemma~\ref{result:U12remaineach}, which is applicable since $a+b\le n_1-1$, and hence $a\le n_1-1$.
\end{proof}

By \eqref{eq:strategy_conditional_probability}, Lemma~\ref{result:U12Flimit} and Corollary~\ref{result:U1U2_remainder},
\beqq
	\lim_{L\to\infty}
	\prob\left[
		\frac{\LPP\left((\ac x L,\bc y L)+\mr c\ww\mathbf v\,(y-1)^{2/3}L^{2/3}\right)-\mv L}
		{\sdev L^{1/3}}>\rr
		\,\bigg|\,
		\LPP(\ac L,\bc L)>\lv L
	\right]
	=1-F_{\tn{BBP},\ww}(\rr+\ww^2).
\eeqq
This completes the proof of Theorem~\ref{result:crossdist}\textnormal{(a)}.

%%%%%%%%%%%%%%%%%%%%%%%%%%%%
%%%%%%%%%%%%%%%%%%%%%%%%%%%%
\section{Region $U_2$}
\label{sec:case2}

We prove Theorem~\ref{result:CCLT} for region $U_2$. 
It is enough to consider $(x,y)\in U_2^<$ since the result in $U_2^>$ follows by exchanging $\ac,\bc,x,y$ with $\bc,\ac,y,x$. Fix $(x,y)\in U_2^<$ and $\rr\in\R$, and set $\mv:=\mv(x,y)$ and $\sdev:=\sdev(x,y)$. 
The two-point formula depends on the ordering of the time parameters. 
We consider separately the three cases $\mv>\lv$, $\mv<\lv$, and $\mv=\lv$. 

From the formula \eqref{eq:LLNconjv} of $\mv$, 
\beq\label{eq:lvmmveq}
	2(\mv-\lv)
	=(x-1)(\lv+\ac-\bc-\sqrt \discr)+(y-1)(\lv-\ac+\bc+\sqrt \discr).
\eeq
Since $y<x$, this implies $\lv(y-1)<\mv-\lv<\lv(x-1)$. Hence, 
\beqq 
	\text{$x>1$ when $\mv>\lv$;} \qquad \text{$y<1$ when $\mv<\lv$.}
\eeqq
Recall $\Thetao$, $\Thetat$, and $\Thetar$ from \eqref{eq:allphasefunctions}, and $\pp$ from \eqref{eq:ppthree}. A direct computation gives
\beq\label{eq:U2_second_derivatives}
	\Thetao''(\pp)=-\ab_+^2,
	\quad
	\frac{\Thetat''(\pp)}{\Thetao''(\pp)}= \frac{\slope y-x}{\slope-1},
	\quad
	\frac{\Thetar''(\pp)}{\Thetao''(\pp)} =- \left( 1 - \frac{\slope y-x}{\slope-1} \right) , 
	\quad
	\sdev= \sqrt{ \frac{\Thetat''(\pp)\Thetar''(\pp)} {\Thetao''(\pp)}}. 
\eeq
We also note the following standard identity.
     
\begin{lem}\label{lem:Brownian_bridge_integral}
Let $A>0$, $t\in(0,1)$, and $s\in\R$. Let $\Gamma_1$ and $\Gamma_2$ be upward-oriented vertical lines, with $\Gamma_1$ lying to the left of $\Gamma_2$. Then
\beqq
	\frac{\sqrt{2\pi A}}{(2\pi\ii)^2}
	\int_{\Gamma_1}\dd u\int_{\Gamma_2}\dd v\,
	\frac{e^{\frac12Atu^2+su+\frac12A(1-t)v^2-sv}}{v-u}
%	=\prob (\sqrt A\,\B(t)>s ) 
    = 1 - \Phi \bigl( (At(1-t))^{-1/2} s\bigr), 
\eeqq
where $\Phi$ denotes the distribution function of a standard Gaussian random variable.
\end{lem}

\begin{proof}
Since $\re(v-u)>0$, we write
$\frac{1}{v-u}=\int_0^\infty e^{-\lambda(v-u)}\,\dd\lambda $
and evaluate the two Gaussian integrals.
\end{proof}

%%%%%%%%%%%%%%%%%%%%%%%%%%%%%%%%%%%%%%%%%%%%%%%%%
\subsection{The case $\mv>\lv$}
\label{sec:U2greater}

Since $\mv>\lv$, we use the ordering in \eqref{eq:strategy_order_greater}, where $(x_L, y_L)=(x,y)$ and 
$\tTo=\mv L+\sdev \rr L^{1/2}$ in \eqref{eq:xygen} and \eqref{eq:Tgen}.
Using \eqref{eq:strategy_phase_functions}, the functions in \eqref{eq:fGen} are
\beq\label{eq:case2greater_f_expansions}
\begin{split}
	\ff_{1,L}(z)=e^{L\Gone(z)+E_{1,L}(z)},\quad
	\ff_{2,L}(z)=e^{L\Gtwo(z)+\sdev\rr L^{1/2}z+E_{2,L}(z)},\quad
	\ff_{12,L}(z)=e^{L\Gonetwo(z)+\sdev\rr L^{1/2}z+E_{12,L}(z)}, 
\end{split}
\eeq
where
\beqq
	(\mcG_1,\mcG_2,\mcG_{12}) = (\Thetao,\Thetar,\Thetat). 
\eeqq
From \eqref{eq:strategy_cp} and Lemma~\ref{result:cpU2less}, since $(x,y)\in U_2^<$, the critical points satisfy 
\beq \label{eq:U2greatercprelation}
	\rzo^\pm= \tzo^\pm, \qquad (\rzt^-, \rzt^+)=(\tzra, \tzrb) , \qquad \rzot^\pm=\tzt^\pm
\eeq
and
\beq\label{eq:case2greater_critical_order}
	-1<\rzo^-<\rzot^-<\rzot^+=\rzo^+=\rzt^-=\pp<\rzt^+,
\eeq
where $\rzt^+=\tzrb <0$ if $y>1$. %Furthermore, 
%\beqq
%	\text{$\tzr^+\in(\pp,0)$ if $y>1$.} 
	%\qquad \text{and} \qquad \text{$\tzr^+\ge 0$ if $y\le 1$.} 
%\eeqq
Note that, since $N_{2,L}-N_{1,L}=\lceil \bc y L\rceil-\lceil \bc L\rceil$, we find  
\beq \label{eq:U21f2ana0}
	\text{$\frac1{\ff_{2,L}(z)}= \frac{(z+1)^{M_{2,L}-M_{1,L}}}{z^{N_{2,L}-N_{1,L}}e^{(T_{2,L}-T_{1,L})z}}$ is analytic at $0$ if $y\le 1$.} 
\eeq 
All functions $\ff_{1,L}$, $\ff_{2,L}$, and $\ff_{12, L}$ are Gaussian families. 

%%%%%%%%%%%%%%%%%%%%%%%%%%%%%%%%%%%%%%%%%%%%%%%%%
\subsubsection{The leading contribution in $U_2$ when $\mv>\lv$}

The leading contribution to \eqref{eq:Q2Lses} comes from the single term indexed by $(1,1)$.

\begin{lem}\label{lem:leadingtermU2>}
Assume $\mv>\lv$. Then
\beqq
	\lim_{L\to\infty}\frac{\QQ_{2,L}^{(1,1)}}{\tP}
	=1-\Phi(\rr).
\eeqq
\end{lem}

\begin{proof}
By Lemma~\ref{lem:simpleQ} and \eqref{eq:Jdown12},
\beqq
	\QQ_{2,L}^{(1,1)}(\rr)
	=\J^{12}_{12}
	=\J^{12}_{21}-\J^{12}_{(12)}.
\eeqq
The second integral is given by
\beqq
	\J^{12}_{(12)}
	=
	\frac{1}{(2\pi\ii)^3}
	\int\dd\xi^1\int\dd\xi^2\int\dd\eta^{12}\,
	\frac{1}{(\xi^1-\eta^{12})(\eta^{12}-\xi^2)(\xi^1-\xi^2)}
	\frac{\ff_{1,L}(\xi^1)\ff_{2,L}(\xi^2)}{\ff_{12,L}(\eta^{12})}.
\eeqq
We deform the $\xi^1$-, $\xi^2$-, and $\eta^{12}$-contours to
\beqq
	\cont_-^L(\rzo^-),
	\qquad
	\cont_-^L(\pp),
	\qquad
	\cont_+^L(\pp+L^{-1/2}),
\eeqq
respectively. By \eqref{eq:case2greater_critical_order}, these contours have the required nesting and cross no poles.
By the method of steepest descent, using the local variables
\beqq
	\xi^1=\rzo^-+u_1L^{-1/2},
	\qquad
	\xi^2=\pp+u_2L^{-1/2},
	\qquad
	\eta^{12}=\pp+v_{12}L^{-1/2},
\eeqq
and Lemmas~\ref{lem:asymptotics_f1} and~\ref{lem:asymptotics_f2}, we find 
\beqq
	\J^{12}_{(12)}
	=-\frac{1+o(1)}{(2\pi\ii)^3L}
	\frac{\ff_{1,L}(\rzo^-)\ff_{2,L}(\pp)}
	{\ff_{12,L}(\pp)(\pp-\rzo^-)^2}
	\int\dd u_1\,e^{\frac12\mcG_1''(\rzo^-)u_1^2}
	\iint\frac{\dd u_2\dd v_{12}}{v_{12}-u_2}
	e^{\frac12\mcG_2''(\pp)u_2^2+\sdev\rr u_2
	-\frac12\mcG_{12}''(\pp)v_{12}^2-\sdev\rr v_{12}}.
\eeqq
We evaluate the Gaussian $u_1$-integral, use
$\ff_{12,L}(\pp)=\ff_{1,L}(\pp)\ff_{2,L}(\pp)$, and use \eqref{eq:U2_second_derivatives}.
Then, by Lemmas~\ref{lem:conditioning_event} and~\ref{lem:Brownian_bridge_integral}, we obtain
\beqq
	\lim_{L\to\infty}\frac{\J^{12}_{(12)}}{\tP}
	=-(1-\Phi(\rr)).
\eeqq

We show that
\beqq
	\J^{12}_{21}
	=
	\frac{1}{(2\pi\ii)^4}
	\int\dd\xi^1\int\dd\xi^2\int\dd\eta^2\int\dd\eta^1\,
	\frac{(\eta^2-\xi^1)(\eta^1-\xi^2)}
	{(\eta^1-\xi^1)^2(\eta^2-\xi^2)^2
	(\xi^1-\xi^2)(\eta^2-\eta^1)}
	\frac{\ff_{1,L}(\xi^1)\ff_{2,L}(\xi^2)}
	{\ff_{2,L}(\eta^2)\ff_{1,L}(\eta^1)}
\eeqq
is negligible. If $y\le1$, then by \eqref{eq:U21f2ana0}, the integrand is analytic at $\eta^2=0$, and hence $\J^{12}_{21}=0$. Suppose $y>1$. In this case $\rzt^+<0$, and we can deform the $\xi^1$-, $\xi^2$-, $\eta^1$-, and $\eta^2$-contours to
\beqq
	\cont^L_-(\rzo^-),
	\qquad
	\cont^L_-(\pp),
	\qquad
	\cont^L_+(\pp+L^{-1/2}),
	\qquad
	\cont^L_+(\rzt^+),
\eeqq
respectively. Proposition~\ref{result:generalasy}, with $d_L=c_*L^{-1/2}$ for some $c_*>0$, $\alpha=4$, and $\beta=0$, gives
\beqq
	|\J^{12}_{21}|
	\le
	C\frac{|\ff_{1,L}(\rzo^-)|\,|\ff_{2,L}(\pp)|}
	{|\ff_{1,L}(\pp)|\,|\ff_{2,L}(\rzt^+)|}
	= \frac{|\ff_{1,L}(\tzo^-)|\,|\ff_{2,L}(\tzra)|}
	{|\ff_{1,L}(\tzo^+)|\,|\ff_{2,L}(\tzrb)|}
	\le Ce^{cL^{1/2}-(\tDel+\tDelr)L}
\eeqq
for some $C,c>0$. 
Since $\tDelr>0$ by Lemma \ref{result:phasegaps},   $\J^{12}_{21}$ is exponentially smaller than $\tP$. Hence, we obtain the result.
\end{proof}

%%%%%%%%%%%%%%%%%%%%%%%%%%%%%%%%%%%%%%%%%%%%%%%%%
\subsubsection{Remainder in $U_2$ when $\mv>\lv$}

By Corollary~\ref{result:Qn1n2bd} and~\ref{result:Jintrearboundgen}, using the $S^{12}_{21}$ bound,
\beq\label{eq:U2greater_reduction}
	|\QQ_{2,L}^{(n_1,n_2)}|
	\le
	2^{7(n_1+n_2)}
	\max_{\substack{0\le i,j\le n_2\\ i+j\ge2n_2-n_1+1}}
	\sum_{a=0}^{n_1\wedge(n_2-i)}
	\sum_{b=0}^{n_1\wedge j}
	a!b!\,
	|\J^{12}_{21}(a,b,n_1,n_2)|,
\eeq
where
\beqq
	\J^{12}_{21}(a,b,n_1,n_2)
	= 
	\J^{1^{n_1-a}(12)^a2^{n_2-a}}_{2^{n_2-b}(12)^b1^{n_1-b}}.
\eeqq

\begin{lem}\label{result:JresultingU2<h>l}
There exist $C,c,L_0>0$ such that, for all $n_1,n_2\ge1$, $0\le a,b\le n_1\wedge n_2$ satisfying $n_1\vee n_2\ge2$ and $a\le n_1-1$, and $L\ge L_0$,
\beqq
	|\J^{12}_{21}(a,b,n_1,n_2)|
	\le
	e^{-\tDel L}
	C^{n_1+n_2}\sqrt{n_1!n_2!}\,
	(n_1+n_2-a-b)!e^{-c(n_1+n_2)L}.
\eeqq
\end{lem}

\begin{proof}
We have
\beqq
	\type(1^{n_1-a}(12)^a2^{n_2-a})=(a,n_1-a,n_2-a),
	\qquad
	\type(2^{n_2-b}(12)^b1^{n_1-b})=(b,n_1-b,n_2-b).
\eeqq

Suppose first that $y>1$. We deform the $\xi_i^1$-, $\xi_i^{12}$-, $\xi_i^2$-, $\eta_i^1$-, $\eta_i^{12}$-, and $\eta_i^2$-contours to
\beqq
	\cont^L_-(\rzo^-),
	\qquad
	\cont^L_-(\rzot^-),
	\qquad
	\cont_-^L(\pp-L^{-1/2}),
	\qquad
	\cont_+^L(\pp),
	\qquad
	\cont^L_+(\pp+L^{-1/2}),
	\qquad
	\cont^L_+(\rzt^+),
\eeqq
respectively. By \eqref{eq:case2greater_critical_order}, these contours 
form an admissible contour system. Thus, Proposition~\ref{result:generalasy}, with
\beqq
	d_L=c_*L^{-1/2}\quad\text{for some }c_*>0,
	\qquad
	\alpha=2n_1+2n_2-a-b,
	\qquad
	\beta=0,
\eeqq
implies that 
\beq\label{eq:U2greater_general_bound}
	|\J^{12}_{21}(a,b,n_1,n_2)|
	\le
	\frac{
	C_1^{n_1+n_2}\sqrt{n_1!n_2!}\,(n_1+n_2-a-b)!}
	{L^{-(2n_1+2n_2-a-b)/2}L^{(2n_1+2n_2-a-b)/2}}
	E_L(a,b),
\eeq
where, using $\ff_{1,L}\ff_{2,L}=\ff_{12,L}$ and \eqref{eq:U2greatercprelation}, 
\beqq
	E_L(a,b)
	:=
	\frac{|\ff_{1,L}(\rzo^-)|^{n_1-a}
	|\ff_{12,L}(\rzot^-)|^a
	|\ff_{2,L}(\pp)|^{n_2-a}}
	{|\ff_{1,L}(\pp)|^{n_1-b}
	|\ff_{12,L}(\pp)|^b
	|\ff_{2,L}(\rzt^+)|^{n_2-b}}
	=
	\frac{|\ff_{1,L}(\tzo^-)|^{n_1-a}
	|\ff_{12,L}(\tzt^-)|^a |\ff_{2,L}(\tzr^-)|^{n_2-b}}
	{|\ff_{1,L}(\tzo^+)|^{n_1-a}
	|\ff_{12,L}(\tzt^+)|^a |\ff_{2,L}(\tzr^+)|^{n_2-b}}. 
\eeqq
From \eqref{eq:case2greater_f_expansions}, 
\beqq
	E_L(a,b)
	\le
	C_2^{n_1+n_2}
	e^{c_1(a+n_2-b)L^{1/2}}e^{\Phi L}
    \le
	C_2^{n_1+n_2}
	e^{c_1(n_1+n_2)L^{1/2}}e^{\Phi_{a,b} L},
	\quad
	\Phi_{a,b}:=-(n_1-a)\tDel-a\tDelt-(n_2-b)\tDelr, 
\eeqq
where, by Lemma~\ref{result:phasegaps}, $\tDel, \tDelt, \tDelr>0$. Since $a\le n_1-1$ by assumption, 
\beqq
	\Phi_{a,b}
	=-\tDel-(n_1-a-1)\tDel-a\tDelt-(n_2-b)\tDelr
	\le-\tDel-c_0(n_1+n_2-b-1),
\eeqq
where $c_0:=\min\{\tDel,\tDelt,\tDelr\}>0$. Since $b\le n_1\wedge n_2$ and $n_1\vee n_2\ge2$, we have 
$n_1+n_2-b-1\ge\frac14(n_1+n_2)$.
Combining these estimates with \eqref{eq:U2greater_general_bound} proves the result for $y>1$.

Suppose now that $y\le1$. Since $1/\ff_{2,L}$ is analytic at $0$,
$\J^{12}_{21}(a,b,n_1,n_2)=0$ if $n_2-b>0$.
Assume $n_2=b$. There are no $\eta_i^2$-integrals, and the case~\textnormal{(c)} extension of Proposition~\ref{result:generalasy} applies. The preceding argument remains valid without the $\tDelr$-term. 
Since $a\le n_1-1$, $n_1\ge n_2\ge 1$ and $n_1\vee n_2\ge2$, we find that 
\beqq
	\Phi_{a,b}%=-(n_1-a)\tDel-a\tDelt
	=-\tDel -(n_1-a-1)\tDel-a\tDelt
	\le-\tDel-(n_1-1) c_0
	\le-\tDel- (n_1+n_2) c_0/4, 
\eeqq
where $c_0:= \min\{\tDel,\tDelt\}>0$. 
Thus, the result holds for $y\le1$ as well.
\end{proof}

\begin{cor}\label{result:U2greater_remainder}
Assume $\mv>\lv$. There exists $c>0$ such that
\beqq
	\frac{1}{\tP}
	\sum_{\substack{n_1,n_2\ge1\\(n_1,n_2)\ne(1,1)}}
	\frac{|\QQ_{2,L}^{(n_1,n_2)}|}{(n_1!n_2!)^2}
	=O(e^{-cL}).
\eeqq
\end{cor}

\begin{proof}
For the indices appearing in \eqref{eq:U2greater_reduction}, since $a\le n_2-i$ and $i+j \ge 2n_2-n_1+1$, we have $ a \le n_1-1$. 
Hence, Lemma~\ref{result:JresultingU2<h>l} applies, and we obtain the result, as in the proof of Corollary~\ref{prop:case1_remainder}. 
\end{proof}

%%%%%%%%%%%%%%%%%%%%%%%%%%%%%%%%%%%%%%%%%%%%%%%%%
\subsection{The case $\mv<\lv$}
\label{sec:U2less}

Since $\mv<\lv$, we use the ordering in \eqref{eq:strategy_order_less}, where $(x_L, y_L)=(x,y)$ and 
$\tTo=\mv L+\sdev \rr L^{1/2}$ in \eqref{eq:xygen} and \eqref{eq:Tgen}.
Using \eqref{eq:strategy_phase_functions}, the functions in \eqref{eq:fGen} are 
\beq\label{eq:case2less_f_expansions}
\begin{split}
	\ff_{1,L}(z)=e^{L\Gone(z)+\sdev\rr L^{1/2}z+E_{1,L}(z)},\quad
	\ff_{2,L}(z)=e^{L\Gtwo(z)-\sdev\rr L^{1/2}z+E_{2,L}(z)},\quad 
	\ff_{12,L}(z)=e^{L\Gonetwo(z)+E_{12,L}(z)}, 
\end{split}
\eeq
where
\beqq
	(\mcG_1,\mcG_2,\mcG_{12}) = (\Thetat,-\Thetar,\Thetao). 
\eeqq
By \eqref{eq:strategy_cp} and Lemma~\ref{result:cpU2less}, since $(x,y)\in U_2^<$, the critical points satisfy 
\beq \label{eq:U2lesscprelation}
	\rzo^\pm= \tzt^\pm, \qquad (\rzt^-, \rzt^+)=(\tzrb, \tzra), \qquad \rzot^\pm=\tzo^\pm
\eeq
and
\beq\label{eq:case2less_critical_order}
	\rzt^-<\rzot^-<\rzo^-<\rzo^+=\rzot^+=\rzt^+=\pp<0, 
\eeq
where $\rzt^- = \tzrb>-1$ if $x<1$.
Furthermore, since $M_{2,L}-M_{1,L}=\lceil \ac L\rceil - \lceil \ac x_L L\rceil$, we find that 
\beq \label{eq:U2f2anamo}
	\text{$\ff_{2,L}(z)=\frac{z^{N_{2,L}-N_{1,L}}e^{(T_{2,L}-T_{1,L})z}}{(z+1)^{M_{2,L}-M_{1,L}}}$ is analytic at $-1$ if $x\ge 1$.} 
\eeq 
All three functions $\ff_{1,L}$, $\ff_{2,L}$, and $\ff_{12,L}$ are Gaussian families.

Theorem~\ref{result:CCLT} for points $(x,y)\in U_2$ with $x,y<1$ was proved in \cite[Theorem~1.5]{Baik-Cordaro-Tripathi25}. Therefore, in the case $\mv<\lv$, it remains only to consider $x\ge1$. Throughout this subsection, we assume
\beqq
    x\ge1.
\eeqq

%%%%%%%%%%%%%%%%%%%%%%%%%%%%%%%%%%%%%%%%%%%%%%%%%
\subsubsection{The leading contribution in $U_2$ when $\mv<\lv$}

The leading contribution to \eqref{eq:Q2Lses} also comes from the single term indexed by $(1,1)$. 

\begin{lem}\label{lem:leadingtermU2<}
Assume $\mv<\lv$ and $x\ge 1$. Then
\beqq
	\lim_{L\to\infty}\frac{\QQ_{2,L}^{(1,1)}}{\tP}
	=1-\Phi(\rr).
\eeqq
\end{lem}

\begin{proof}
By Lemma~\ref{lem:simpleQ} and \eqref{eq:Jup12},
\beqq
	\QQ_{2,L}^{(1,1)}=\J^{12}_{12} =\J^{21}_{12}+\J^{(12)}_{12}.
\eeqq
From \eqref{eq:U2f2anamo}, the integrand of $\J^{21}_{12}$ is analytic at $\xi^2=-1$, and thus $\J^{21}_{12}=0$. 
For the second integral 
\beqq
	\J^{(12)}_{12}
	=
	\frac{1}{(2\pi\ii)^3}
	\int\dd\xi^{12}\int\dd\eta^1\int\dd\eta^2\,
	\frac{1}{(\xi^{12}-\eta^1)(\xi^{12}-\eta^2)(\eta^1-\eta^2)}
	\frac{\ff_{12,L}(\xi^{12})}{\ff_{1,L}(\eta^1)\ff_{2,L}(\eta^2)}, 
\eeqq
we deform the $\xi^{12}$-, $\eta^2$-, and $\eta^1$-contours to
\beqq
	\cont^L_-(\rzot^-),
	\qquad
	\cont^L_+(\pp),
	\qquad
	\cont^L_+(\pp + L^{-1/2}),
\eeqq
respectively, and evaluate the integral using the method of steepest-descent as in Lemma~\ref{lem:leadingtermU2>}, and obtain the result. 
\end{proof}

%%%%%%%%%%%%%%%%%%%%%%%%%%%%%%%%%%%%%%%%%%%%%%%%%
\subsubsection{Remainder in $U_2$ when $\mv<\lv$}

We use Corollary~\ref{result:Qn1n2bd} and~\ref{result:Jintrearboundgen}. 
This time, we use the $S^{21}_{21}$ bound and find 
\beqq
\begin{split}
	|\QQ_{2,L}^{(n_1,n_2)}|
	&\le
	2^{7(n_1+n_2)}
	\max_{\substack{0\le i,j\le n_2\\i+j\ge2n_2-n_1+1}}
	\sum_{a=0}^{n_1\wedge i}
	\sum_{b=0}^{n_1\wedge j}
	a!b!\,|\J^{21}_{21}(a,b,n_1,n_2)|.
\end{split}
\eeqq
where 
\beqq 
	\J^{21}_{21}(a,b,n_1,n_2) =\J^{2^{n_2-a}(12)^a1^{n_1-a}}_{2^{n_2-b}(12)^b1^{n_1-b}}. 
\eeqq

\begin{lem}\label{result:JresultingU2<h<l}
Suppose $\mv<\lv$ and $x\ge 1$. 
There exist $C,c,L_0>0$ such that, for all $n_1,n_2\ge1$ with $n_1\vee n_2\ge2$, $0\le a, b\le n_1\wedge n_2$, and $L\ge L_0$,
\beqq
	|\J^{21}_{21}(a,b,n_1,n_2)|
	\le
	e^{-\tDel L} C^{n_1+n_2}\sqrt{n_1!n_2!}\,(n_1+n_2-a-b)!e^{-c(n_1+n_2)L}.
\eeqq
\end{lem}

\begin{proof}
Since $x\ge 1$, $\ff_{2,L}$ is analytic at $-1$ from \eqref{eq:U2f2anamo}, and hence 
\beqq
	\J^{21}_{21}(a,b,n_1,n_2)=0
	\qquad\text{if } a<n_2.
\eeqq

Suppose $a=n_2$. Then, necessarily $n_1\ge n_2$, and 
$\J^{21}_{21}(a,b,n_1,n_2) =\J^{(12)^{n_2}1^{n_1-n_2}}_{2^{n_2-b}(12)^b1^{n_1-b}}$. 
We have
\beqq
	\type((12)^{n_2}1^{n_1-n_2})=(n_2,n_1-n_2,0),
	\qquad
	\type(2^{n_2-b}(12)^b1^{n_1-b})=(b, n_1-b, n_2-b).
\eeqq
There are no $\xi_i^2$-integrals. 
We deform the $\xi_i^{12}$-, $\xi_i^1$-, $\eta_i^1$-, $\eta_i^{12}$, and $\eta_i^2$-contours for the integral $\J^{21}_{21}(a,b,n_1,n_2)$ to
\beqq
	\cont^L_-(\rzot^-),
	\qquad
	\cont^L_-(\rzo^-),
	\qquad
	\cont_+^L(\pp),
	\qquad
	\cont_+^L(\pp+L^{-1/2}),
	\qquad
	\cont_+^L(\pp+2L^{-1/2}),
\eeqq
respectively, which, by \eqref{eq:case2less_critical_order}, form an admissible contour system. We can deform to this contour system without crossing any poles.
Proceeding as in the proof of Lemma~\ref{result:JresultingU2<h>l}, Proposition~\ref{result:generalasy}, with
\beqq
	d_L=c_*L^{-1/2}\quad\text{for some }c_*>0,
	\qquad
	\alpha=2n_1+n_2-b,
	\qquad
	\beta=0,
\eeqq
implies 
\beqq
	|\J^{21}_{21}(a,b,n_1,n_2)|
	\le
	\frac{
	C_1^{n_1+n_2}\sqrt{n_1!n_2!}\,(n_1-b)!}
	{L^{-(2n_1+n_2-b)/2} L^{(2n_1+n_2-b)/2}}
	E_L(b),
\eeqq
where, using $\ff_{2,L}= \ff_{12,L}/\ff_{1,L}$ and \eqref{eq:U2lesscprelation}, 
\beqq
	E_L(b)
	:=
	\frac{|\ff_{1,L}(\rzo^-)|^{n_1-n_2}
	|\ff_{12,L}(\rzot^-)|^{n_2}}
	{|\ff_{1,L}(\pp)|^{n_1-b} |\ff_{12,L}(\pp)|^{b} |\ff_{2,L}(\pp)|^{n_2-b}}
	=
	\frac{|\ff_{1,L}(\tzt^-)|^{n_1-n_2}
	|\ff_{12,L}(\tzo^-)|^{n_2}}
	{|\ff_{1,L}(\tzt^+)|^{n_1-n_2} |\ff_{12,L}(\tzo^+)|^{n_2}}.
\eeqq
Thus, from \eqref{eq:case2less_f_expansions}, 
\beqq
	E_L(b)
	\le
	C_2^{n_1+n_2}e^{c_1n_1L^{1/2}} e^{\Phi L}, 
	\qquad \Phi= - (n_1-n_2) \tDelt - n_2 \tDel
\eeqq
where $\tDel, \tDelt>0$ by Lemma~\ref{result:phasegaps}. 
Since $n_1\ge n_2\ge 1$ and $n_1\vee n_2\ge 2$, we find that
\beqq
	\Phi= -\tDel -(n_1-n_2) \tDelt - (n_2-1) \tDel
	\le - \tDel - (n_1-1) c_0
	\le -\tDel - (n_1+n_2)c_0/4, 
\eeqq
where $c_0:= \min\{\tDel,\tDelt\}>0$. 
Thus, the result follows. 
\end{proof}

\begin{cor}\label{result:U2less_remainder}
Assume $\mv<\lv$ and $x\ge 1$. There exists $c>0$ such that
\beqq
	\frac{1}{\tP}
	\sum_{\substack{n_1,n_2\ge1\\(n_1,n_2)\ne(1,1)}}
	\frac{|\QQ_{2,L}^{(n_1,n_2)}|}{(n_1!n_2!)^2}
	=O(e^{-cL}).
\eeqq
\end{cor}

\begin{proof} 
The proof is similar to that of Corollary~\ref{result:U2greater_remainder}. 
\end{proof}

By \eqref{eq:strategy_conditional_probability}, Lemmas~\ref{lem:leadingtermU2>} and \ref{lem:leadingtermU2<}, and Corollaries~\ref{result:U2greater_remainder} and~\ref{result:U2less_remainder}, for $\mv\ne\lv$,
\beqq 
	\lim_{L\to\infty}
	\prob\left[
		\frac{\LPP(\ac x L,\bc y L)-\mv L}{\sdev L^{1/2}}>\rr
		\,\bigg|\,
		\LPP(\ac L,\bc L)>\lv L
	\right]
	=1-\Phi(\rr).
\eeqq
Thus, Theorem~\ref{result:CCLT} in $U_2^<$ for $\mv\ne\lv$ is proved. 

%%%%%%%%%%%%%%%%%%%%%%%%%%%%%%%%%%%%%%%%%%%%%%%%%
\subsection{The case $\mv=\lv$}
\label{sec:U2equal}

In this case, \eqref{eq:lvmmveq} and $y<x$ imply $\lv(y-1)<0<\lv(x-1)$. 
Hence, $x>1$ and $y<1$.
Since $\mv=\lv$, we have $\Thetar(z)=-\ac(x-1)\log(1+z)+\bc(y-1)\log z$ from \eqref{eq:allphasefunctions}. 
Hence,
\beqq
	\Thetar'(z)=\frac{p(z)}{z(z+1)},
\eeqq
for a linear function $p$, and thus, $\Thetar$ has a unique regular critical point, which is $\tzr=\pp$ since $\Thetar=\Thetat-\Thetao$ and $\pp$ is a critical point of both $\Thetat$ and $\Thetao$. Thus, the critical points satisfy
\beqq
	-1<\tzo^-<\tzt^-<\tzt^+=\tzo^+=\tzr=\pp<0.
\eeqq

Since the two threshold times are $\lv L$ and $\lv L+\sdev\rr L^{1/2}$, their ordering is determined by the sign of $\rr$.
When $\rr\ge0$, we use the ordering in \eqref{eq:strategy_order_greater} and proceed as in Section~\ref{sec:U2greater}, using the analyticity property \eqref{eq:U21f2ana0}. When $\rr<0$, we use the ordering in \eqref{eq:strategy_order_less} and proceed as in Section~\ref{sec:U2less}, using the analyticity property \eqref{eq:U2f2anamo}. The proofs of Lemmas~\ref{lem:leadingtermU2>} and~\ref{lem:leadingtermU2<}, together with the corresponding remainder arguments, carry over with these simplifications. This proves Theorem~\ref{result:CCLT} in $U_2^<$ when $\mv=\lv$. We omit the details.

%%%%%%%%%%%%%%%%%%%%%%%%%%%%%%%%%%%%%%%%%%%%%%%%
%%%%%%%%%%%%%%%%%%%%%%%%%%%%%%%%%%%%%%%%%%%%%%%%
\section{Boundary between $U_2$ and $U_3$}
\label{sec:U2U3}

We prove Theorem~\ref{result:crossdist}\textnormal{(b)}. It is enough to consider the boundary between $U_2^<$ and $U_3^<$ since the result for the boundary between $U_2^>$ and $U_3^>$ follows by exchanging $\ac,\bc,x,y$ with $\bc,\ac,y,x$. 
Fix $(x,y)\in \R_+^2$ on this boundary and $\ww,\rr\in\R$, and set $\mv:=\mv(x,y)$ and $\sdev:=\sdev(x,y)$. 
As in Section~\ref{sec:case2}, we consider separately the three cases $\mv>\lv$, $\mv<\lv$, and $\mv=\lv$. 
By the same argument as there, $x>1$ when $\mv>\lv$, while $y<1$ when $\mv<\lv$. 

Let
\beq \label{eq:HDefU2U3}
	\mcH(z):=- \sqrt{\slope \ac}\log(z+1)-\sqrt \bc\log z 
\eeq
as in \eqref{eq:HdefinU1U2}. 
Direct computations using \eqref{eq:allphasefunctions},  \eqref{eq:CLTconstant} and \eqref{eq:littlec} imply that  
\beqq
	\Thetat'''(\pp)=-2\sdev^3,
	\qquad
	\mcH'(\pp)=0,
	\qquad
	\mcH''(\pp)=\frac{2\sdev^2}{\mr c\,y^{2/3}}. 
\eeqq

For $L>0$, set
\beq\label{eq:U2U3_shifted_parameters}
	(x_L, y_L)
	= (x,y) + \mr c\ww y^{2/3}L^{-1/3} \left( -\frac{\sqrt{\slope}}{\sqrt \ac}, 
	\frac{1}{\sqrt \bc} \right). 
\eeq

%%%%%%%%%%%%%%%%%%%%%%%%%%%%%%%%%%%%%%%%%%%%%%%%%
\subsection{The case $\mv>\lv$}
\label{sec:U2U3greater}

Since $\mv>\lv$, we use the ordering in \eqref{eq:strategy_order_greater}, where $(x_L, y_L)$ in \eqref{eq:xygen} is given by \eqref{eq:U2U3_shifted_parameters}, and 
$\tTo=\mv L+\sdev \rr L^{1/3}$ in \eqref{eq:Tgen}.
From \eqref{eq:strategy_phase_functions}, the functions in \eqref{eq:fGen} are
\beq\label{eq:U2U3greater_f_expansions}
\begin{split}
	\ff_{1,L}(z)&=e^{L\Gone(z)+E_{1,L}(z)},\\
	\ff_{2,L}(z)&=e^{L\Gtwo(z) -\mr c\ww y^{2/3}L^{2/3}\mcH(z)+\sdev\rr L^{1/3}z+E_{2,L}(z)},\\
	\ff_{12,L}(z)&=e^{L\Gonetwo(z)- \mr c\ww y^{2/3}L^{2/3}\mcH(z)+\sdev\rr L^{1/3}z+E_{12,L}(z)},
\end{split}
\eeq
where
\beqq
	(\mcG_1,\mcG_2,\mcG_{12}) = (\Thetao,\Thetar,\Thetat). 
\eeqq
%and $E_{1,L}, E_{2,L}, E_{12,L}$ are uniformly bounded on compact subsets of $\C\setminus\{-1,0\}$.
By \eqref{eq:strategy_cp} and Lemma~\ref{result:cpU2less}, since $y=x/\slope$,
\beq \label{eq:U2U3greatercprelation}
	\rzo^\pm= \tzo^\pm, \qquad (\rzt^-, \rzt^+)=(\tzra, \tzrb) , \qquad \rzot^\pm=\tzt^+=\tzt^-=:\tzt
\eeq
and
\beq\label{eq:U2U3greater_critical_order}
	-1<\rzo^-<\rzot=\rzo^+=\rzt^-=\pp<\rzt^+, \qquad \rzot:=\rzot^-=\rzot^+, 
\eeq
where $\rzt^+<0$ if $y>1$, and $\rzt^+>0$ if $y<1$. 
When $y=1$, the root $\tzrb=0$ of $\tpr$ cancels the factor $z$ in \eqref{eq:tGdertp}, and $\rzt^+$ does not exist, and 
$\Gtwo$ has only one regular critical point $\rzt^-=\tzra=\pp$.

If $y<1$, then $N_{2,L}-N_{1,L}<0$ for all sufficiently large $L$. Thus, $1/\ff_{2,L}$ is analytic at $0$ as in \eqref{eq:U21f2ana0}. 
If $y=1$, the sign of $N_{2,L}-N_{1,L}$ depends on the $O(L^{-1/3})$ perturbation of $y_L$. 

The functions $\ff_{1,L}$, $\ff_{2,L}$, and $\ff_{12,L}$ are a Gaussian family, a modified Gaussian family, and an Airy family, respectively.

%%%%%%%%%%%%%%%%%%%%%%%%%%%%%%%%%%%%%%%%%%%%%%%%%
\subsubsection{Decomposition when $\mv>\lv$}

We first isolate the terms that contribute at leading order to the series $\QQ_{2,L}$ in \eqref{eq:Q2Lses}. Since the decomposition below is purely algebraic, we suppress the dependence on $L$. The leading terms $\mathcal M_n$ arise from specific residue contributions in the contour-rearrangement formulas for $\QQ_2^{(n,n)}$ and $\QQ_2^{(n+1,n)}$. All remaining contributions are collected into $\mathcal R_{n_1,n_2}$ and will be shown to be subleading.

\begin{lem} \label{eq:U23decompgreater} 
For $n_1, n_2\ge 1$ and $0\le a, b\le n_1\wedge n_2$, let
\beqq 
	\J^{12}_{21}(a,b,n_1,n_2)=\J^{1^{n_1-a}(12)^a2^{n_2-a}}_{2^{n_2-b}(12)^b1^{n_1-b}}, 
\eeqq
as in Corollary~\ref{result:Jintrearboundgen}. 
Then,  
\beq
\label{eq:QQ2expinremainMnU2U3h>l}
	\left| \QQ_{2}  
	- \sum_{n=1}^\infty\frac{\mathcal M_n}{(n!)^2} \right|
	\le \sum_{n_1,n_2\ge1}\frac{\mathcal R_{n_1,n_2}}{(n_1!n_2!)^2}, 
	\qquad
	\mathcal M_n:= n(-1)^n\J^{1(12)^{n-1}2}_{(12)^n}+(-1)^{n+1}\J^{1(12)^n}_{(12)^n1},
\eeq
where
\beq
\label{eq:U2U3Rn1n2h>l}
	\mathcal R_{n_1,n_2}
	:= 
	2^{7(n_1+n_2)}
	\max_{\substack{0\le i,j\le n_2\\i+j\ge2n_2-n_1+1}}
	\sum_{a=0}^{n_1\wedge(n_2-i)}
	\sum_{b=0}^{n_1\wedge j}
	a!b!\, | \J^{12}_{21}(a,b,n_1,n_2) |  
	\mathbf{1}_{(a,b)\ne(n_1-1,n_2)}. 
\eeq
\end{lem}

\begin{proof}
From \eqref{eq:Jup21} and \eqref{eq:Jdown12},
\beq\label{eq:U2U3greater_full_exchange}
	\J^{2^{n_2-i}1^{n_1}2^i}_{2^{n_2-j}1^{n_1}2^j} 
	=
	 \sum_{a=0}^{n_1\wedge(n_2-i)}
	\sum_{b=0}^{n_1\wedge j}
	\varepsilon_{a,b}^{n_1, n_2} \dc_a^{n_1,n_2-i}\dc_b^{n_1,j}
	\J^{12}_{21}(a,b,n_1,n_2)  , 
	\qquad \varepsilon_{a,b}^{n_1, n_2}\in \{-1, 1\}. 
\eeq
We can verify the following two special values: 
\beqq
	\varepsilon_{n-1,n}^{n, n} 
    = (-1)^n , 
	\qquad 
	\varepsilon_{n,n}^{n+1, n}=(-1)^n.
\eeqq

For every $n\ge 1$, 
by \eqref{eq:case3_simpleQ_less} and \eqref{eq:U2U3greater_full_exchange}, 
\beqq
	\QQ_{2}^{(n,n)}
	=(-1)^{n-1}
	\sum_{\substack{0\le i,j\le n \\i+j\ge n+1}}
	(-1)^{i+j}\binom ni\binom nj
	\sum_{a=0}^{n\wedge(n-i)}
	\sum_{b=0}^{n\wedge j}
	\varepsilon_{a,b}^{n, n} \dc_a^{n,n-i}\dc_b^{n,j}
	\J^{12}_{21}(a,b,n,n).
\eeqq
We separate out the term with $(i,j,a,b)=(1,n, n-1,n)$. 
Using $\binom{n}{k}\le 2^n$, $(n+1)^2\le 2^{2n}$, and \eqref{eq:dcest}, 
\beqq
	\left| \QQ_{2}^{(n,n)}
	- n(-1)^n(n!)^2\J^{1(12)^{n-1}2}_{(12)^n} \right|
	\le 
	2^{8n} \max_{\substack{0\le i,j\le n \\i+j\ge n+1}}
	\sum_{a=0}^{n\wedge(n-i)}
	\sum_{b=0}^{n\wedge j}
	a! b! 
	| \J^{12}_{21}(a,b,n,n)  | \mathbf{1}_{(i,j,a,b)\neq (1,n, n-1,n)}.
\eeqq
The right-hand side is bounded by $\mathcal R_{n,n}$. 

For every $n\ge 1$, by \eqref{eq:case3_simpleQ_greater} and \eqref{eq:U2U3greater_full_exchange}, 
\beqq
	\QQ_2^{(n+1,n)}
	= 
	- \sum_{\substack{0\le i,j\le n\\ i+j=n}}
	\binom{n}{i}\binom{n}{j}
	\sum_{a=0}^{(n+1)\wedge(n-i)}
	\sum_{b=0}^{(n+1)\wedge j}
	\varepsilon_{a,b}^{n+1, n} \dc_a^{n+1,n-i}\dc_b^{n+1,j}
	\J^{12}_{21}(a,b,n+1,n).
\eeqq
Separating out the term with $(i,j, a,b)=(0, n, n, n)$, and using $\binom{n}{k}\le 2^n$, $(n+1)^2\le 2^{2n}$, and \eqref{eq:dcest}, 
\beqq 
\begin{split}
	&\left| \QQ_2^{(n+1,n)}
	- (-1)^{n+1}((n+1)!)^2\J^{1(12)^n}_{(12)^n1} \right| \\
	&\qquad \le 
	2^{6(2n+1)} 
	\max_{\substack{0\le i,j\le n\\ i+j=n}}
	\sum_{a=0}^{(n+1)\wedge(n-i)}
	\sum_{b=0}^{(n+1)\wedge j}
	a! b!
	|\J^{12}_{21}(a,b,n+1,n) | \mathbf{1}_{(i,j, a,b)\neq (0, n, n, n)}. 
\end{split} 
\eeqq
The right-hand side is bounded by $\mathcal R_{n+1,n}$. 

For $(n_1,n_2)\notin\{(n,n),(n+1,n):n\in\N\}$, Corollaries~\ref{result:Qn1n2bd} and~\ref{result:Jintrearboundgen}, using the $S^{12}_{21}$ version, give
\beq\label{eq:U23tempo1}
    \bigl|\QQ_2^{(n_1,n_2)}\bigr|\le 2^{7(n_1+n_2)}\max_{\substack{0\le i,j\le n_2\\ i+j\ge2n_2-n_1+1}}\sum_{a=0}^{n_1\wedge(n_2-i)}\sum_{b=0}^{n_1\wedge j}a!b!\,|\J^{12}_{21}(a,b,n_1,n_2)|.
\eeq
We claim that $(a,b)=(n_1-1,n_2)$ does not occur in the sum. Indeed, if it did, then $b=n_2\le n_1\wedge j$ and $j\le n_2$, so $j=n_2$ and $n_2\le n_1$. The condition $i+j\ge2n_2-n_1+1$ would then imply $i\ge n_2-n_1+1$. On the other hand, $a=n_1-1\le n_2-i$, and hence
$i\le n_2-n_1+1$. 
Therefore, $i=n_2-n_1+1$. Since $i\ge0$, we have $n_1\le n_2+1$. Together with $n_2\le n_1$, this gives $n_1=n_2$ or $n_1=n_2+1$, contradicting the assumption on $(n_1,n_2)$. Thus, $(a,b)=(n_1-1,n_2)$ does not occur in the sum, and the right-hand side of \eqref{eq:U23tempo1} is bounded by $\mathcal R_{n_1,n_2}$.
\end{proof} 

%%%%%%%%%%%%%%%%%%%%%%%%%%%%%%%%%%%%%%%%%%%%%%%%%
\subsubsection{The leading contribution on the $U_2^</U_3^<$ boundary when $\mv>\lv$} 
\label{sec:U2U3leadinggreater}

The term $\mathcal M_n$ consists of two integrals. We evaluate them in the next two lemmas. 
Let
\beq \label{eq:I1}
	I^{n-1}_n:= \frac{1}{(2\pi\ii)^{2n-1}}
	\int_{<0}\dd\bu\int \dd\bv
	\frac{\prod_{i=1}^{n-1}e^{-u_i^3/3-\ww u_i^2+\rr u_i}}{\prod_{j=1}^ne^{-v_j^3/3-\ww v_j^2+\rr v_j}}
	\K_n((\bu,0)\mid\bv)^2 \frac{\prod_{j=1}^nv_j}{\prod_{i=1}^{n-1}u_i},	
\eeq
where $\bu=(u_1,\ldots,u_{n-1})$ and $\bv=(v_1,\ldots,v_n)$. Here, the $u_i$-contours are the usual left Airy contours lying to the left of $0$, and the $v_j$-contours are the usual right Airy contours. 
Let 
\beq \label{eq:I2}
	I^n_n:=\frac{1}{(2\pi\ii)^{2n}}
	\int_{<0}\dd\bu\int \dd\bv
	\frac{\prod_{i=1}^ne^{-u_i^3/3-\ww u_i^2+\rr u_i}}{\prod_{j=1}^ne^{-v_j^3/3-\ww v_j^2+\rr v_j}}
	\K_n(\bu\mid\bv)^2 \frac{\prod_{j=1}^nv_j}{\prod_{i=1}^{n}u_i}
\eeq 
where $\bu=(u_1,\ldots,u_{n})$ and $\bv=(v_1,\ldots,v_n)$, and the $u_i$-contours again lie to the left of $0$. 

\begin{lem}\label{result:U23Jledgre1}
For every $n\ge1$,
\beqq
	\lim_{L\to\infty}\frac{1}{\tP}\J^{1(12)^{n-1}2}_{(12)^n} = - I^{n-1}_n, 
	\qquad
	\lim_{L\to\infty}\frac{1}{\tP}\J^{1(12)^n}_{(12)^n1} = I^n_n. 
\eeqq
\end{lem}

\begin{proof} 
First, consider 
\beqq
	\J^{1(12)^{n-1}2}_{(12)^n}
	=
	\frac{1}{(2\pi\ii)^{2n+1}}
	\int\dd\xi^1 \dd \bsxi^{12} \dd\xi^2
	\int\dd\bseta^{12}\, 
	\nPi^{1(12)^{n-1}2}_{(12)^n}(\bsxi,\bseta)
	\FF^{1(12)^{n-1}2}_{(12)^n}(\bsxi,\bseta),
\eeqq
where, writing $\bsxi^{12}=(\xi_1^{12},\ldots,\xi_{n-1}^{12})$ and $\bseta^{12}=(\eta_1^{12},\ldots,\eta_n^{12})$, 
\beq\label{eq:U2U3greater_PiA}
	\nPi^{1(12)^{n-1}2}_{(12)^n}(\bsxi,\bseta)
	=\K_n(\bseta^{12}\mid\bsxi^{12},\xi^1)\K_1(\xi^1\mid\xi^2)\K_n(\bsxi^{12},\xi^2\mid\bseta^{12}),
\eeq
and
\beqq
	\FF^{1(12)^{n-1}2}_{(12)^n}(\bsxi,\bseta)
	=\frac{\ff_{1,L}(\xi^1)\ff_{2,L}(\xi^2)\prod_{i=1}^{n-1}\ff_{12,L}(\xi_i^{12})}{\prod_{j=1}^n\ff_{12,L}(\eta_j^{12})}.
\eeqq

By \eqref{eq:U2U3greater_critical_order}, $-1<\rzo^-<\rzot= \rzt^-= \pp<0$.
We deform the $\xi^1$-, $\xi_i^{12}$-, $\xi^2$-, and $\eta_j^{12}$-contours to
\beqq
	\cont^L_-(\rzo^-),\qquad
	\cont_-^L(\pp-L^{-1/3}),\qquad
	\cont_-^L(\pp),\qquad
	\cont_+^L(\pp+L^{-1/3}),
\eeqq
respectively. For all sufficiently large $L$, these contours form an admissible contour system, and the deformation crosses no poles. We apply the method of steepest descent, using Lemmas~\ref{lem:asymptotics_f2}, \ref{lem:modifiedGaussian}, and~\ref{lem:asymptotics_f2_case2}, and using the variables near the critical points given by 
\beq \label{eq:U2U3greater_local_variables_A}
	\xi^1=\rzo^-+\frac{\hat u_1}{L^{1/2}},
	\qquad
	\xi_i^{12}=\pp+\frac{u_i}{\sdev L^{1/3}},
	\qquad \xi^2=\pp+\frac{\hat u_2}{L^{1/2}},
	\qquad
	\eta_j^{12}=\pp+\frac{v_j}{\sdev L^{1/3}}.
\eeq

By Lemmas~\ref{lem:asymptotics_f1}, \ref{lem:modifiedGaussian}, and~\ref{result:dbcplocal},
\beqq
	\FF^{1(12)^{n-1}2}_{(12)^n}(\bsxi,\bseta)
	=\frac{\ff_{1,L}(\rzo^-)\ff_{2,L}(\pp)}{\ff_{12,L}(\pp)}
	e^{\Gone''(\rzo^-)\hat u_1^2/2 + \Gtwo''(\pp) \hat u_2^2/2}
	\frac{\prod_{i=1}^{n-1}e^{-u_i^3/3-\ww u_i^2+\rr u_i}}{\prod_{j=1}^ne^{-v_j^3/3-\ww v_j^2+\rr v_j}}(1+o(1)).
\eeqq
Note that 
\beqq
	\frac{\ff_{1,L}(\rzo^-)\ff_{2,L}(\pp)}{\ff_{12,L}(\pp)}=\frac{\ff_{1,L}(\rzo^-)}{\ff_{1,L}(\pp)}, 
	\qquad 
	\Gtwo''(\pp)= \Gonetwo''(\pp)-\Gone''(\pp)= -\Gone''(\pp)
\eeqq 
since $\ff_{12,L}=\ff_{1,L}\ff_{2,L}$, and $\pp$ is the double critical point of $\Gonetwo$. 

For complex vectors $\alpha=(\alpha_1,\ldots, \alpha_k)$ and $\beta=(\beta_1,\ldots, \beta_{k'})$, let
\beqq
	\mr C_{k,k'}(\alpha; \beta):=\begin{bmatrix} \frac{1}{\alpha_i - \beta_j} \end{bmatrix}_{1\le i\le k, 1\le j\le k'}
\eeqq
denote the rectangular Cauchy matrix. Under \eqref{eq:U2U3greater_local_variables_A}, 
\beqq
	\nPi^{1(12)^{n-1}2}_{(12)^n}(\bsxi,\bseta)
	=(1+o(1))
	\frac{(\sdev L^{1/3})^{2n-1}}{\rzo^--\pp}
	\det\begin{bmatrix}
		&\frac{1}{\pp-\rzo^-}\\
		\mr C_{n,n-1}(\bv;\bu)&\vdots\\
		&\frac{1}{\pp-\rzo^-}
	\end{bmatrix}
	\det\begin{bmatrix}
		\mr C_{n-1,n}(\bu;\bv)\\
		-\frac1{v_1}\ \cdots\ -\frac1{v_n}
	\end{bmatrix}.
\eeqq
The last determinant is $\K_n((\bu,0)\mid\bv)$. Transposing the first determinant and using $\frac{v_j}{u_i(u_i-v_j)}=\frac{1}{u_i-v_j}-\frac{1}{u_i}$, it is equal to 
\beqq
	\frac{(-1)^{n-1}}{\pp-\rzo^-}
	\det\begin{bmatrix}
		\mr C_{n-1,n}(\bu;\bv)\\
		1\ \cdots\ 1
	\end{bmatrix}
	=\frac{(-1)^n}{\pp-\rzo^-}\frac{\prod_{j=1}^nv_j}{\prod_{i=1}^{n-1}u_i}
	\det\begin{bmatrix}
		\mr C_{n-1,n}(\bu;\bv)\\
		-\frac1{v_1}\ \cdots\ -\frac1{v_n}
	\end{bmatrix}.
\eeqq
Therefore,
\beqq
	\nPi^{1(12)^{n-1}2}_{(12)^n}(\bsxi,\bseta)
	=(1+o(1))\frac{(-1)^{n-1}(\sdev L^{1/3})^{2n-1}}{(\pp-\rzo^-)^2}
	\K_n((\bu,0)\mid\bv)^2\frac{\prod_{j=1}^nv_j}{\prod_{i=1}^{n-1}u_i}.
\eeqq

Apply the method of steepest-descent, evaluating the Gaussian integrals in $\hat u_1$ and $\hat u_2$, reversing the orientations of the $v_j$-contours, and using Lemma~\ref{lem:conditioning_event}, 
we obtain the limit. 
Note that from the choice of the $\xi_i^{12}$- contours, the $u_i$-contours lie to the left of $0$. 

Now, we consider 
\beqq
	\J^{1(12)^n}_{(12)^n1}
	=\frac{1}{(2\pi\ii)^{2n+2}}
	\int\dd\xi^1 \dd \bsxi^{12} 
	\int 
    \dd \bseta^{12}\dd\eta^1
	\nPi^{1(12)^n}_{(12)^n1}(\bsxi,\bseta)
	\frac{\ff_{1,L}(\xi^1)\prod_{i=1}^n\ff_{12,L}(\xi_i^{12})}{\ff_{1,L}(\eta^1)\prod_{j=1}^n\ff_{12,L}(\eta_j^{12})},
\eeqq
where, writing $\bsxi^{12}=(\xi_1^{12},\ldots,\xi_n^{12})$ and $\bseta^{12}=(\eta_1^{12},\ldots,\eta_n^{12})$, 
\beq\label{eq:U2U3greater_PiB}
\begin{split}
	\nPi^{1(12)^n}_{(12)^n1}(\bsxi,\bseta)
	&=\K_{n+1}(\bseta^{12},\eta^1\mid\bsxi^{12},\xi^1)\K_1(\xi^1\mid\eta^1)\K_n(\bsxi^{12}\mid\bseta^{12})\\
	&=(-1)^{n+1}\frac{\K_n(\bsxi^{12}\mid\bseta^{12})^2}{(\xi^1-\eta^1)^2}
	\prod_{i=1}^n\frac{(\eta^1-\eta_i^{12})(\xi^1-\xi_i^{12})}{(\eta^1-\xi_i^{12})(\xi^1-\eta_i^{12})}.
\end{split}
\eeq

By \eqref{eq:U2U3greater_critical_order}, 
$-1<\rzo^-<\rzot =\rzo^+=\pp<0$, and we deform the $\xi^1$-, $\xi_i^{12}$-, $\eta^1$-, and $\eta_i^{12}$-contours to
\beqq
	\cont_-^L(\rzo^-),\qquad
	\cont_-^L(\pp-L^{-1/3}),\qquad
	\cont_+^L(\pp),\qquad
	\cont_+^L(\pp+L^{-1/3}),
\eeqq
respectively. For all sufficiently large $L$, these contours form an admissible contour system, and we can deform to them without crossing any poles. 
Near the critical points, we use the change of variables 
\beqq
	\xi^1=\rzo^-+\frac{\hat u}{L^{1/2}},
	\qquad
	\xi_i^{12}=\pp+\frac{u_i}{\sdev L^{1/3}},
	\qquad
	\eta^1=\pp+\frac{\hat v}{L^{1/2}},
	\qquad 
	\eta_i^{12}=\pp+\frac{v_i}{\sdev L^{1/3}}. 
\eeqq
By Lemmas~\ref{lem:asymptotics_f1} and~\ref{result:dbcplocal},
\beqq
	\frac{\ff_{1,L}(\xi^1)\prod_{i=1}^n\ff_{12,L}(\xi_i^{12})}{\ff_{1,L}(\eta^1)\prod_{j=1}^n\ff_{12,L}(\eta_j^{12})}
	=\frac{\ff_{1,L}(\rzo^-)}{\ff_{1,L}(\pp)} \frac{e^{\Gone''(\rzo^-)\hat u^2/2}}{e^{ \Gone''(\pp)\hat v^2/2}}
	\frac{\prod_{i=1}^ne^{-u_i^3/3-\ww u_i^2+\rr u_i}}{\prod_{j=1}^ne^{-v_j^3/3-\ww v_j^2+\rr v_j}}(1+o(1)).
\eeqq
On the same contours, \eqref{eq:U2U3greater_PiB} gives 
\beqq
	\nPi^{1(12)^n}_{(12)^n1}(\bsxi,\bseta)
	=(1+o(1))\frac{(-1)^{n+1}(\sdev L^{1/3})^{2n}}{(\pp-\rzo^-)^2}
	\K_n(\bu\mid\bv)^2 \frac{\prod_{j=1}^nv_j}{\prod_{i=1}^{n}u_i}.
\eeqq
By the steepest-descent analysis, we obtain, after reversing the orientations of the $v_i$-contours, the result. 
Note that again from the choice of the $\xi_i^{12}$- contours, the $u_i$-contours lie to the left of $0$. 
\end{proof}

\begin{lem}\label{result:U23greatersum}
For every $n\ge1$,
\beqq
\begin{split}
	I^{n}_n
	= - n I^{n-1}_n 
	+ \frac{1}{(2\pi\ii)^{2n}}
	\int_{>0}\dd\bu\int \dd\bv \, 
	\frac{\prod_{i=1}^ne^{-u_i^3/3-\ww u_i^2+\rr u_i}}{\prod_{i=1}^ne^{-v_i^3/3-\ww v_i^2+\rr v_i}}
	\K_n(\bu\mid\bv)^2 \frac{\prod_{j=1}^nv_j}{\prod_{i=1}^{n}u_i} , 
\end{split}
\eeqq
where $\bu=(u_1,\ldots,u_n)$ and $\bv=(v_1,\ldots,v_n)$, and the $u_i$-contours lie to the right of $0$. 
\end{lem}

\begin{proof}
The integrand of the integral $I^n_n$ has a simple pole at $u_i=0$ for each $i$. Moving each $u_i$-contour to the right across $0$ and using symmetry in the $u_i$-variables, the residue theorem implies the result. 
\end{proof}

\begin{cor}\label{result:U23leadingseriesh>l}
Let $\mathcal M_n:= n(-1)^n\J^{1(12)^{n-1}2}_{(12)^n}+(-1)^{n+1}\J^{1(12)^n}_{(12)^n1}$ as in \eqref{eq:QQ2expinremainMnU2U3h>l}. We have
\beqq
	\lim_{L\to\infty}\frac{1}{\tP}
	\sum_{n=1}^\infty\frac{\mathcal M_n}{(n!)^2} =\sum_{n=1}^\infty\frac{(-1)^{n+1}}{(n!)^2}\frac{1}{(2\pi\ii)^{2n}}
	\int_{>\ww}\dd\bu\int \dd\bv
	\prod_{i=1}^n\frac{e^{-u_i^3/3+(\rr+\ww^2)u_i}(v_i-\ww)}{e^{-v_i^3/3+(\rr+\ww^2)v_i}(u_i-\ww)}
	\K_n(\bu\mid\bv)^2.
\eeqq
The contours are the usual left and right Airy contours, oriented from bottom to top, 
and the $u_i$-contours lie to the right of $\ww$.
\end{cor}

\begin{proof}
The result follows from the last two lemmas after changing the variables $u_i\mapsto u_i-\ww$ and $v_i\mapsto v_i-\ww$. The interchange of the limit and the sum is justified by the same type of uniform bounds used in the preceding sections. 
\end{proof}

%%%%%%%%%%%%%%%%%%%%%%%%
\subsubsection{Remainder on the $U_2^</U_3^<$ boundary when $\mv>\lv$} 
\label{sec:U2U3remaingreater}

\begin{lem}\label{result:U23greaterremaineach}
There exist $C,c,L_0>0$ such that, for all $n_1,n_2\ge1$, $0\le a\le(n_1-1)\wedge n_2$, $0\le b\le n_1\wedge n_2$ satisfying $(a,b)\ne(n_1-1,n_2)$, and $L\ge L_0$,
\beq\label{eq:U23greater_each_remainder_bound}
	|\J^{12}_{21}(a,b, n_1,n_2)|
	\le
	e^{- \tDel L} C^{n_1+n_2}\sqrt{n_1!n_2!}\,(n_1+n_2-a-b)!e^{-cL}. 
\eeq
\end{lem}

\begin{proof}
We have
\beqq
	\type(1^{n_1-a}(12)^a2^{n_2-a})=(a,n_1-a,n_2-a),
	\qquad
	\type(2^{n_2-b}(12)^b1^{n_1-b})=(b,n_1-b,n_2-b).
\eeqq

Suppose first that $y>1$. By \eqref{eq:U2U3greater_critical_order}, 
$-1<\rzo^-<\rzot=\rzo^+=\rzt^-=\pp<\rzt^+<0$. 
We deform the $\xi_i^1$-, $\xi_i^{12}$-, $\xi_i^2$-, $\eta_i^1$-, $\eta_i^{12}$-, and $\eta_i^2$-contours for the integral $\J^{12}_{21}(a,b,n_1,n_2)$ to to 
\beqq
	\cont^L_-(\rzo^-),
	\qquad
	\cont_-^L(\pp-2L^{-1/3}),
	\qquad
	\cont_-^L(\pp-L^{-1/2}),
	\qquad
	\cont_+^L(\pp+L^{-1/2}),
	\qquad
	\cont_+^L(\pp+2L^{-1/3}),
	\qquad
	\cont^L_+(\rzt^+),
\eeqq
respectively, which form an admissible contour system. We can deform to these contours without crossing poles. 
Although some contours are only of order $L^{-1/2}$ apart, those pairs do not appear on opposite sides of a Cauchy determinant in \eqref{eq:Caucdt}. Every pair relevant to Definition~\ref{def:admissiblecontours} is separated by at least $c_*L^{-1/3}$ for some $c_*>0$. 
Recall that the functions $\ff_{1,L}$, $\ff_{2,L}$, and $\ff_{12,L}$ are a Gaussian family, a modified Gaussian family, and an Airy family, respectively. Note that $\mcH'(\rzt^-)=\mcH'(\pp)=0$. 
Thus, Proposition~\ref{result:generalasy}, with 
\beqq
	d_L=c_*L^{-1/3},
	\qquad
	\alpha=2(n_1+n_2-a-b),
	\qquad
	\beta=a+b,
	\qquad
	\mu=  n_2-b
\eeqq
implies that 
\beqq
	|\J^{12}_{21}(a,b,n_1,n_2)|
	\le
	\frac{C_1^{n_1+n_2}e^{C_1(n_2-b)L^{1/3}}\sqrt{n_1!n_2!}\,(n_1+n_2-a-b)!}
	{L^{-(2n_1+2n_2-a-b)/3} L^{(n_1+n_2-a-b) + (a+b)/3}} 
	\,E_L(a,b),
\eeqq
where, using $\ff_{1,L}\ff_{2,L}= \ff_{12, L}$ and \eqref{eq:U2U3greatercprelation}, 
\beqq
	E_L(a,b)
	:=
	\frac{|\ff_{1,L}(\rzo^-)|^{n_1-a}|\ff_{12,L}(\pp)|^a|\ff_{2,L}(\pp)|^{n_2-a}}
	{|\ff_{1,L}(\pp)|^{n_1-b}|\ff_{12,L}(\pp)|^b|\ff_{2,L}(\rzt^+)|^{n_2-b}}
	=
	\frac{|\ff_{1,L}(\tzo^-)|^{n_1-a} |\ff_{2,L}(\tzra)|^{n_2-b}}{|\ff_{1,L}(\tzo^+)|^{n_1-a} |\ff_{2,L}(\tzrb)|^{n_2-b}}. 
\eeqq
Since the exponent of $\ff_{1,L}$ contains no terms of order $L^{2/3}$ or $L^{1/3}$, \eqref{eq:U2U3greater_f_expansions} gives
\beqq
	E_L(a,b)
	\le
	C_2^{n_1+n_2}e^{c_1(n_2-b)L^{2/3}}
	e^{-[(n_1-a)\tDel+(n_2-b)\tDelr]L}.
\eeqq
After increasing $c_1$ if necessary, to absorb the factor $e^{C_1(n_2-b)L^{1/3}}$ from the modified Gaussian estimate into the $L^{2/3}$-order factor above, 
\beq \label{eq:U23remd1}
	|\J^{12}_{21}(a,b,n_1,n_2)|
	\le
	e^{-\tDel L} C^{n_1+n_2}\sqrt{n_1!n_2!}\,(n_1+n_2-a-b)!
	e^{c_1(n_2-b)L^{2/3}- (n_2-b)\tDelr L} e^{-(n_1-a-1)\tDel L}.
\eeq
Since $(a,b)\ne(n_1-1,n_2)$, at least one of $n_1-a-1$ and $n_2-b$ is positive. This proves the result when $y>1$.

Suppose next that $y<1$. Then $1/\ff_{2,L}$ is analytic at $0$ for all sufficiently large $L$. 
Thus, if $n_2-b>0$, the integral vanishes by shrinking an innermost $\eta_i^2$-contour to $0$. If $b=n_2$, there are no $\eta_i^2$-integrals, and the case (c) extension of Proposition~\ref{result:generalasy} applies, and \eqref{eq:U23remd1} still holds with $b=n_2$. 
Since $(a,n_2)\ne(n_1-1,n_2)$, we have $n_1-a-1\ge1$, and the result follows.

It remains to consider $y=1$. If $N_{2,L}\le N_{1,L}$, the preceding argument applies. 
Suppose that $N_{2,L}>N_{1,L}$. 
Fix $q\in(\pp,0)$ sufficiently close to $0$, and deform the $\eta_i^2$-contours to 
$\con_+(q) = \{z: |z|=|q|\}$. 
Since $q\in(\pp,0)$, the monotonicity argument used for the right contours in Lemma~\ref{lem:asymptotics_f2} implies that there exists $C>0$ such that 
\beqq
	\left\|\frac{1}{\ff_{2,L}}\right\|_{L^1(\con_+(q))}
	\le
	\frac{C}{L^{1/2}|\ff_{2,L}(q)|}
\eeqq
for all sufficiently large $L$. Thus, the same argument implies the estimate \eqref{eq:U23remd1} where $\tDelr$ is changed to 
\beqq
	\tDelr(q):=\Thetar(q)-\Thetar(\pp). 
\eeqq
Since $\Thetar$ is an increasing function on $(\pp, 0)$, we find $\tDelr(q)>0$, and  \eqref{eq:U23greater_each_remainder_bound} holds when $y=1$.
\end{proof}

\begin{cor}\label{result:U23greater_remainder}
There exists $c>0$ such that
\beqq
	\frac{1}{\tP}
	\sum_{n_1,n_2\ge1}
	\frac{\mathcal R_{n_1,n_2}}{(n_1!n_2!)^2}
	=O(e^{-cL}).
\eeqq
\end{cor}

\begin{proof}
For the indices in the definition \eqref{eq:U2U3Rn1n2h>l} of $\mathcal R_{n_1,n_2}$, the inequalities $a\le n_2-i$, $i+j\ge2n_2-n_1+1$, and $j\le n_2$ imply $a\le n_1-1$. Hence, Lemma~\ref{result:U23greaterremaineach} applies. The result then follows from \eqref{eq:triplfactorial} and Corollary~\ref{result:tPldp}. 
\end{proof}

%%%%%%%%%%%%%%%%%%%%%%%%%%%%%%%%%%%%%%%%%%%%%%%%%
\subsection{The case $\mv<\lv$}
\label{sec:U2U3less}

Since $\mv<\lv$, we use the ordering in \eqref{eq:strategy_order_less}, where $(x_L, y_L)$ in \eqref{eq:xygen} is given by \eqref{eq:U2U3_shifted_parameters}, and 
$\tTo=\mv L+\sdev \rr L^{1/3}$ in \eqref{eq:Tgen}.
From \eqref{eq:strategy_phase_functions}, the functions in \eqref{eq:fGen} are
\beq\label{eq:U2U3less_f_expansions}
\begin{split}
	\ff_{1,L}(z)
	&=e^{L\Gone(z)-\mr c\ww y^{2/3}L^{2/3}\mcH(z)+\sdev\rr L^{1/3}z+E_{1,L}(z)},\\
	\ff_{2,L}(z)
	&=e^{L\Gtwo(z)+\mr c\ww y^{2/3}L^{2/3}\mcH(z)-\sdev\rr L^{1/3}z+E_{2,L}(z)},\\
	\ff_{12,L}(z)
	&=e^{L\Gonetwo(z)+E_{12,L}(z)},
\end{split}
\eeq
where 
\beqq
	(\mcG_1,\mcG_2,\mcG_{12}) = (\Thetat,-\Thetar,\Thetao), 
\eeqq
and $\mcH$ is given by \eqref{eq:HDefU2U3}. 
By \eqref{eq:strategy_cp} and Lemma~\ref{result:cpU2less}, since $y=x/\slope$, 
\beq \label{eq:U2U3lesscprelation}
	\rzo^\pm= \tzt^-=\tzt^+=:\tzt,  \qquad (\rzt^-, \rzt^+)=(\tzrb,\tzra) , \qquad \rzot^\pm=\tzo^\pm
\eeq
and
\beq\label{eq:U2U3less_critical_order}
	\rzt^-<\rzot^-<\rzo=\rzot^+=\rzt^+=\pp<0,\qquad  \rzo:=\rzo^-=\rzo^+, 
\eeq
where $\rzt^->-1$ if $x<1$, and $\rzt^-<-1$ if $x>1$. 
If $x=1$, the zero $\tzrb=-1$ of $\tpr$ cancels the factor $z+1$ in \eqref{eq:tGdertp}, and $\Gtwo$ has only the regular critical point $\rzt^+=\tzra=\pp$.

If $x>1$, then $M_{2,L}-M_{1,L}<0$ for all sufficiently large $L$, and hence $\ff_{2,L}$ is analytic at $-1$.
The functions $\ff_{1,L}$, $\ff_{2,L}$, and $\ff_{12,L}$ are an Airy family, a modified Gaussian family, and a Gaussian family, respectively.

%%%%%%%%%%%%%%
\subsubsection{Decomposition when $\mv<\lv$}

We isolate the terms that contribute at leading order to the series \eqref{eq:Q2Lses}. 
Those terms arise from specific residue contributions in $\QQ_{2}^{(n,1)}$. 
We suppress the dependence on $L$ in the following lemma. 

\begin{lem} \label{lem:U23decompless} 
For $n_1, n_2\ge 1$ and $0\le a, b\le n_1\wedge n_2$, let
\beqq 
	\J^{21}_{12}(a,b,n_1,n_2) =\J^{2^{n_2-a}(12)^a1^{n_1-a}}_{1^{n_1-b}(12)^b2^{n_2-b}}, 
\eeqq
as in Corollary~\ref{result:Jintrearboundgen}. 
Then,  
\beq\label{eq:U23lessRestint}
	\left| \QQ_{2}  
	- \sum_{n=1}^\infty\frac{\mathcal M_n}{(n!)^2} \right|
	\le \sum_{n_1,n_2\ge1}\frac{\mathcal R_{n_1,n_2}}{(n_1!n_2!)^2}, 
	\qquad
	\mathcal M_n:= n\J^{(12)1^{n-1}}_{1^n2}+(-1)^{n+1}\J^{(12)1^n}_{1^n(12)},
\eeq
where
\beqq
	\mathcal R_{n_1,n_2}
	:=
	2^{7(n_1+n_2)}
	\max_{\substack{0\le i,j\le n_2\\i+j\ge2n_2-n_1+1}}
	\sum_{a=0}^{n_1\wedge i}
	\sum_{b=0}^{n_1\wedge(n_2-j)}
	a!b!\,|\J^{21}_{12}(a,b,n_1,n_2)|
	\mathbf{1}_{(n_2,a)\ne(1,1)}.
\eeqq
\end{lem}

\begin{proof}
For $(n_1,n_2)=(1,1)$, by Lemma~\ref{lem:simpleQ} and \eqref{eq:Jup12},
\beqq
	\QQ_2^{(1,1)}=\J^{12}_{12}=\J^{(12)}_{12}+\J^{21}_{12}.
\eeqq
For $(n_1, n_2)=(n,1)$ with $n\ge 2$, by Lemma~\ref{lem:simpleQ}, and \eqref{eq:Jup12} and \eqref{eq:Jdown21}, 
\beqq
\begin{split}
	\QQ_2^{(n,1)}
	&=(-1)^{n+1}\left((n-2)\J^{21^n}_{21^n}+\J^{21^n}_{1^n2}+\J^{1^n2}_{21^n}\right) \\
	&=n\J^{(12)1^{n-1}}_{1^n2}
	+(-1)^nn^2\J^{(12)1^{n-1}}_{1^{n-1}(12)}
	+(-1)^{n+1}n\J^{21^n}_{1^n2}
	+n(n-1)\J^{21^n}_{1^{n-1}(12)}.
\end{split}
\eeqq
Thus, shifting $n$ to $n+1$ in $(-1)^nn^2\J^{(12)1^{n-1}}_{1^{n-1}(12)}$, we obtain, 
\beqq
	\sum_{n=1}^\infty\frac{\QQ_2^{(n,1)}}{(n!)^2}
	-\sum_{n=1}^\infty\frac{1}{(n!)^2}
	\left(n\J^{(12)1^{n-1}}_{1^n2}+(-1)^{n+1}\J^{(12)1^n}_{1^n(12)}\right)
	=\sum_{n=1}^\infty\frac{1}{(n!)^2}
	\left((-1)^{n+1}n\J^{21^n}_{1^n2}+n(n-1)\J^{21^n}_{1^{n-1}(12)}\right).
\eeqq
The last sum is bounded by $\sum_{n=1}^\infty \frac{\mathcal R_{n,1}}{(n!)^2}$. 

For $(n_1,n_2)\notin \{(n,1) : n\in \N\}$, Corollaries~\ref{result:Qn1n2bd} and~\ref{result:Jintrearboundgen}, using the $S^{21}_{12}$ version, give 
\beqq 
	\bigl|  \QQ_2^{(n_1,n_2)} \bigr|
	\le 
	2^{7(n_1+n_2)}
	\max_{\substack{0\le i,j\le n_2\\i+j\ge2n_2-n_1+1}}
	\sum_{a=0}^{n_1\wedge i}
	\sum_{b=0}^{n_1\wedge(n_2-j)}
	a!b!\,|\J^{21}_{12}(a,b,n_1,n_2)|. 
\eeqq
Since $n_2\neq 1$, the right-hand side is $\mathcal R_{n_1, n_2}$. Thus the result follows. 
\end{proof}

%%%%%%%%%%%%%%
\subsubsection{The leading contribution on the $U_2^</U_3^<$ boundary when $\mv<\lv$} 
\label{sec:U2U3leadingless}

\begin{lem}\label{result:U23lessJlim1}
For every $n\ge1$,
\beqq
	\lim_{L\to\infty}\frac{1}{\tP}\J^{(12)1^{n-1}}_{1^n2}=(-1)^{n+1} I^{n-1}_n, 
	\qquad 
	\lim_{L\to\infty}\frac{1}{\tP}\J^{(12)1^n}_{1^n(12)} = I^n_n,
\eeqq
where $I^{n-1}_n$ and $I^n_n$ are defined in \eqref{eq:I1} and \eqref{eq:I2}, respectively. 
\end{lem}

\begin{proof}
The proof is similar to that of Lemma~\ref{result:U23Jledgre1}. 
We have 
\beqq
	\J^{(12)1^{n-1}}_{1^n2}
	=\frac{1}{(2\pi\ii)^{2n+1}}
	\int\dd\xi^{12} \, \dd \bsxi^{1}
	\int  \dd \bseta^{1}\,\dd\eta^2 \, 
	\nPi^{(12)1^{n-1}}_{1^n2}(\bsxi,\bseta)
%	\FF^{(12)1^{n-1}}_{1^n2}(\bsxi,\bseta)
	\frac{\ff_{12,L}(\xi^{12})\prod_{i=1}^{n-1}\ff_{1,L}(\xi_i^1)}{\ff_{2,L}(\eta^2)\prod_{j=1}^n\ff_{1,L}(\eta_j^1)},
\eeqq
and
\beqq
	\J^{(12)1^n}_{1^n(12)}
	=\frac{1}{(2\pi\ii)^{2n+2}}
	\int\dd\xi^{12} \,  \dd \bsxi^{1} 
	\int \dd \bseta^{1} \,\dd\eta^{12} \, 
	\nPi^{(12)1^n}_{1^n(12)}(\bsxi,\bseta)
%	\FF^{(12)1^n}_{1^n(12)}(\bsxi,\bseta). 
	\frac{\ff_{12,L}(\xi^{12})\prod_{i=1}^n\ff_{1,L}(\xi_i^1)}{\ff_{12,L}(\eta^{12})\prod_{j=1}^n\ff_{1,L}(\eta_j^1)} 
\eeqq
where $\bsxi^{1}=(\xi_1^{1},\ldots,\xi_{n-1}^{1})$ in the first integral and $\bsxi^{1}=(\xi_1^{1},\ldots,\xi_n^{1})$  in the second integral. For both integrals, $\bseta^{1}=(\eta_1^{1},\ldots,\eta_n^{1})$. 

For the first integral, we deform the $\xi^{12}$-, $\xi_i^1$-, $\eta^2$-, and $\eta_j^1$-contours to
\beqq
	\cont^L_-(\rzot^-),\qquad
	\cont_-^L(\pp-L^{-1/3}),\qquad
	\cont^L_+(\pp),\qquad
	\cont_+^L(\pp+L^{-1/3}),
\eeqq
respectively. For the second integral, we deform the $\xi^{12}$-, $\xi_i^1$-, $\eta^{12}$-, and $\eta_j^1$-contours to the same contours. They are an admissible contour system for each integral. Near the critical points, for the first integral we set 
\beqq
	\xi^{12}=\rzot^-+\frac{\hat u}{L^{1/2}},
	\qquad
	\xi_i^1=\pp+\frac{u_i}{\sdev L^{1/3}},
	\qquad 
	\eta^2=\pp+\frac{\hat v}{L^{1/2}},
	\qquad
	\eta_j^1=\pp+\frac{v_j}{\sdev L^{1/3}}.
\eeqq
For the second integral, we set $\eta^{12}=\pp+\frac{\hat v}{L^{1/2}}$ instead. Using Lemmas~\ref{lem:asymptotics_f1}, \ref{lem:modifiedGaussian}, and~\ref{result:dbcplocal}, and noting that $\mcH'(\pp)=0$, we obtain
\beqq
	\frac{\ff_{12,L}(\xi^{12})\prod_{i=1}^{n-1}\ff_{1,L}(\xi_i^1)}{\ff_{2,L}(\eta^2)\prod_{j=1}^n\ff_{1,L}(\eta_j^1)}
	=\frac{\ff_{12,L}(\rzot^-)}{\ff_{1,L}(\pp)\ff_{2,L}(\pp)}
	\frac{e^{\Gonetwo''(\rzot^-)\hat u^2/2}}{e^{ \Gtwo''(\pp)\hat v^2/2}}
	\frac{\prod_{i=1}^{n-1}e^{-u_i^3/3-\ww u_i^2+\rr u_i}}{\prod_{j=1}^ne^{-v_j^3/3-\ww v_j^2+\rr v_j}}(1+o(1)).
\eeqq
and
\beqq
	\frac{\ff_{12,L}(\xi^{12})\prod_{i=1}^n\ff_{1,L}(\xi_i^1)}{\ff_{12,L}(\eta^{12})\prod_{j=1}^n\ff_{1,L}(\eta_j^1)}
	=\frac{\ff_{12,L}(\rzot^-)}{\ff_{12,L}(\pp)}
	\frac{e^{\Gonetwo''(\rzot^-)\hat u^2/2}}{e^{ \Gonetwo''(\pp)\hat v^2/2}}
	\frac{\prod_{i=1}^ne^{-u_i^3/3-\ww u_i^2+\rr u_i}}{\prod_{j=1}^ne^{-v_j^3/3-\ww v_j^2+\rr v_j}}(1+o(1)).
\eeqq
Note that $\ff_{1,L}(\pp)\ff_{2,L}(\pp)=\ff_{12,L}(\pp)$, $\ff_{12,L}=\ffz$ in Lemma~\ref{lem:conditioning_event}, and 
$\Gtwo''(\pp) = \Gonetwo''(\pp)-\Gone''(\pp)= \Gonetwo''(\pp)$ since $\pp$ is a double critical point of $\Gone$. 
Also note that 
$\Gonetwo(\rzot^-) - \Gonetwo(\pp)= \Thetao(\tzo^-)- \Thetao(\tzo^+)$. 

We also have
\beqq
\begin{split}
	\nPi^{(12)1^{n-1}}_{1^n2}(\bsxi,\bseta)
	&=\K_n(\bseta^1\mid\xi^{12},\bsxi^1)\K_n(\bsxi^1,\eta^2\mid\bseta^1)\K_1(\xi^{12}\mid\eta^2) \\
	&=(1+o(1))\frac{(\sdev L^{1/3})^{2n-1}}{(\pp-\rzot^-)^2}
	\K_n((\bu,0)\mid\bv)^2\frac{\prod_{j=1}^nv_j}{\prod_{i=1}^{n-1}u_i}.
\end{split}
\eeqq
and
\beqq
\begin{split}
	\nPi^{(12)1^n}_{1^n(12)}(\bsxi,\bseta)
	&=\K_{n+1}(\eta^{12},\bseta^1\mid\xi^{12},\bsxi^1)\K_n(\bsxi^1\mid\bseta^1)\K_1(\xi^{12}\mid\eta^{12})\\
	&=(1+o(1))\frac{(-1)^{n+1}(\sdev L^{1/3})^{2n}}{(\pp-\rzot^-)^2}
	\K_n(\bu\mid\bv)^2 \frac{\prod_{j=1}^nv_j}{\prod_{i=1}^{n}u_i}. 
\end{split}
\eeqq
Hence, the result follows from the method of steepest-descent. 
\end{proof}

\begin{cor}\label{result:U23leadingserieshlessl}
Let $\mathcal M_n= n\J^{(12)1^{n-1}}_{1^n2}+(-1)^{n+1}\J^{(12)1^n}_{1^n(12)}$ as in \eqref{eq:U23lessRestint}.  
We have
\beqq
\begin{split}
	\lim_{L\to\infty}\frac{1}{\tP}
	\sum_{n=1}^\infty\frac{\mathcal M_n}{(n!)^2} 
	=\sum_{n=1}^\infty\frac{(-1)^{n+1}}{(n!)^2}\frac{1}{(2\pi\ii)^{2n}}
	\int_{>\ww}\dd\bu\int \dd\bv
	\prod_{i=1}^n
	\frac{e^{-u_i^3/3+(\rr+\ww^2)u_i}(v_i-\ww)}{e^{-v_i^3/3+(\rr+\ww^2)v_i}(u_i-\ww)}
	\K_n(\bu\mid\bv)^2.
\end{split}
\eeqq
The contours are the usual left and right Airy contours, oriented from bottom to top, and the $u_i$-contours lie to the right of $\ww$.
\end{cor}

\begin{proof}
It follows from Lemmas~\ref{result:U23lessJlim1} and~\ref{result:U23greatersum}, as in Corollary~\ref{result:U23leadingseriesh>l}. The interchange of the limit and the sum is justified by uniform bounds as in other regions. 
\end{proof}

%%%%%%%%%%%%%%%%%%%%%%%%%%%%%%%%%%%%%%%%%%%%%%%%%
%%%%%%%%%%%%%%%%%%%%%%%%%%%%%%%%%%%%%%%%%%%%%%%%%%%%%%%%%%%%%%%%
\subsubsection{Remainder on the $U_2^</U_3^<$ boundary when $\mv<\lv$}

\begin{lem}\label{result:U23lessremaineach}
There exist $C,c,L_0>0$ such that, for all $n_1,n_2\ge1$, $0\le a,b\le n_1\wedge n_2$ satisfying $(n_2,a)\ne(1,1)$, and $L\ge L_0$,
\beqq
	|\J^{21}_{12}(a,b,n_1,n_2)| 
	\le e^{-\tDel L}  C^{n_1+n_2}\sqrt{n_1!n_2!}\,(n_1+n_2-a-b)!e^{-cL}.
\eeqq
\end{lem}

\begin{proof}
We have
\beqq
	\type(2^{n_2-a}(12)^a1^{n_1-a})=(a,n_1-a,n_2-a),
	\qquad
	\type(1^{n_1-b}(12)^b2^{n_2-b})=(b,n_1-b,n_2-b).
\eeqq

Suppose first that $x<1$. By \eqref{eq:U2U3less_critical_order},
\beqq
	-1<\rzt^-<\rzot^-<\rzo=\rzot^+=\rzt^+=\pp<0. 
\eeqq
We deform the $\xi_i^2$-, $\xi_i^{12}$-, $\xi_i^1$-, $\eta_i^2$-, $\eta_i^{12}$-, and $\eta_i^1$-contours for the integral $\J^{21}_{12}(a,b,n_1,n_2)$ to to
\beqq
	\cont^L_-(\rzt^-),\qquad
	\cont^L_-(\rzot^-),\qquad
	\cont_-^L(\pp-L^{-1/3}),\qquad
	\cont^L_+(\pp),\qquad
	\cont^L_+(\pp+L^{-1/2}),\qquad
	\cont_+^L(\pp+L^{-1/3}),
\eeqq
respectively. 
They form an admissible contour system for the integral, and we can deform these contours. 
Although some contours are only of order $L^{-1/2}$ apart, those pairs do not appear on opposite sides of a Cauchy determinant in \eqref{eq:Caucdt}. Every pair relevant to Definition~\ref{def:admissiblecontours} is separated by at least $c_*L^{-1/3}$ for some $c_*>0$. 
Recall that $\ff_{1,L}$, $\ff_{2,L}$, and $\ff_{12,L}$ are an Airy family, a modified Gaussian family, and a Gaussian family, respectively. 
Note that $\HH'(\rzot^+)=\HH'(\pp)=0$. 
Thus, Proposition~\ref{result:generalasy}, with
\beqq
	d_L=c_*L^{-1/3},
	\qquad
	\alpha=2n_2,
	\qquad
	\beta=2n_1-a-b,
	\qquad 
	\mu = n_2-a, 
\eeqq
implies that 
\beqq
	|\J^{21}_{12}(a,b,n_1,n_2)|
	\le
	L^{-n_2/3}C_1^{n_1+n_2}e^{C_1(n_2-a)L^{1/3}}\sqrt{n_1!n_2!}\,(n_1+n_2-a-b)!\,E_L(a),
\eeqq
where, using $\ff_{1,L}=\ff_{12,L}/\ff_{2,L}$ and \eqref{eq:U2U3lesscprelation}, 
\beqq
	E_L(a)
	:=\frac{|\ff_{1,L}(\pp)|^{n_1-a}|\ff_{12,L}(\rzot^-)|^a|\ff_{2,L}(\rzt^-)|^{n_2-a}}
	{|\ff_{1,L}(\pp)|^{n_1-b}|\ff_{12,L}(\pp)|^b|\ff_{2,L}(\pp)|^{n_2-b}}
	=
	\frac{|\ff_{12,L}(\tzo^-)|^a |\ff_{2,L}(\tzrb)|^{n_2-a}}{|\ff_{12,L}(\tzo^+)|^a |\ff_{2,L}(\tzra)|^{n_2-a}} . 
\eeqq
Since the exponent of $\ff_{12,L}$ contains no terms of order $L^{2/3}$ or $L^{1/3}$, \eqref{eq:U2U3less_f_expansions} gives
\beqq
	E_L(a)\le C_2^{n_1+n_2}e^{c_1(n_2-a)L^{2/3}}e^{-a\tDel L-(n_2-a)\tDelr L} 
\eeqq
After increasing $c_1$ if necessary, and writing $a\tDel+ (n_2-a)\tDelr = \tDel + (n_2-1) \tDel + (n_2-a) (\tDelr-\tDel)$, 
\beq\label{eq:U23less_normalized_phase_bound}
	|\J^{21}_{12}(a,b,n_1,n_2)|
	\le e^{- \tDel L} C^{n_1+n_2}\sqrt{n_1!n_2!}\,(n_1+n_2-a-b)!
	e^{c_1(n_2-a)L^{2/3}-(n_2-a)(\tDelr-\tDel)L} e^{-(n_2-1)\tDel L} .
\eeq 
By Lemma~\ref{lem:delG2>delG12U2U2h<l} below, $\tDelr-\tDel>0$. 
Since $n_2-a\ge 0$, $n_2\ge 1$, and $(n_2-a, n_2-1)\neq (0,0)$, which follows from $(n_2,a)\neq (1,1)$, 
the result follows when $x<1$.

Suppose next that $x>1$. Then $\ff_{2,L}$ is analytic at $-1$ for all sufficiently large $L$. If $n_2-a>0$, the integral vanishes by shrinking an innermost $\xi_i^2$-contour to $-1$. If $a=n_2$, there are no $\xi_i^2$-variables, and the case (b) extension of Proposition~\ref{result:generalasy} applies, and 
\eqref{eq:U23less_normalized_phase_bound} still holds with $a=n_2$. 
Since $(n_2,a)\ne(1,1)$, we have $n_2\ge2$, and thus, the result follows when $x>1$. 

Suppose finally that $x=1$. By Lemma~\ref{lem:delG2>delG12U2U2h<l}, there exists $\delta>0$ such that $\tDelr^{(0)}>\tDel+3\delta$. Since $\Thetar$ extends analytically across $-1$, there is a neighborhood of $-1$ on which
\beq\label{eq:U23less_x1_phase_bound}
	\re\bigl(\Thetar(z)-\Thetar(\pp)\bigr)\ge\tDel+2\delta.
\eeq
Take the $\xi_i^2$-contours to be nested circles $\Sigma$ around $-1$ contained in this neighborhood, and use the same contours as above for the remaining variables. From \eqref{eq:U2U3less_f_expansions} and \eqref{eq:U23less_x1_phase_bound},
\beq\label{eq:U23less_x1_f2_bound}
	\|\ff_{2,L}\|_{L^1(\Sigma)}
	\le C_2e^{-(\tDel+\delta)L+c_1L^{2/3}}|\ff_{2,L}(\pp)|. 
\eeq
Repeating the proof of Proposition~\ref{result:generalasy}, using \eqref{eq:U23less_x1_f2_bound} for the $n_2-a$ left $\ff_{2,L}$-factors, gives
\beqq
	|\J^{21}_{12}(a,b,n_1,n_2)|
	\le e^{- \tDel L} L^{n_2/6}C^{n_1+n_2}\sqrt{n_1!n_2!}\,(n_1+n_2-a-b)!
	e^{c_1(n_2-a)L^{2/3}-\delta(n_2-a)L} e^{-(n_2-1)\tDel L}
\eeqq
Thus, the result follows for $x=1$. 
\end{proof}

\begin{lem}\label{lem:delG2>delG12U2U2h<l}
On the $U_2^</U_3^<$ boundary, when $\mv<\lv$, we have 
\beqq
	\tDelr>\tDel\quad\text{if $x<1$},
\eeqq
and 
\beqq
	\tDelr^{(0)}:=\Thetar(-1)-\Thetar(\pp)>\tDel\quad\text{if $x=1$}.
\eeqq
\end{lem}

\begin{proof}
Let $x<1$. 
Since $\Thetao=\Thetat-\Thetar$, 
\beqq
	\tDelr-\tDel
	=\Thetat(\tzo^-)-\Thetat(\pp)+\Thetar(\tzrb)-\Thetar(\tzo^-).
\eeqq
By \eqref{eq:U2U3less_critical_order}, $-1<\tzrb<\tzo^-<\pp$. 
Since $\Thetat'(z)= \frac{\mv(z-\pp)^2}{z(z+1)}$ on the $U_2^</U_3^<$ boundary, $\Thetat$ is strictly decreasing on $(-1,\pp)$, and hence the first difference is positive. 
Since $\Thetar'(z)= \frac{(\mv-\lv)(z-\tzrb)(z-\pp)}{z(z+1)}$ and $\mv<\lv$, $\Thetar$ is strictly decreasing on $(\tzrb,\pp)$, 
and hence the second difference is also positive. 
The proof for $x=1$ is identical, using
\beqq
	\tDelr^{(0)}-\tDel
	=\Thetat(\tzo^-)-\Thetat(\pp)+\Thetar(-1)-\Thetar(\tzo^-).
\eeqq
\end{proof}

\begin{cor}\label{result:U2U3h<l_remainder}
There exists $c>0$ such that
\beqq
	\frac{1}{\tP}
	\sum_{n_1,n_2\ge1}\frac{\mathcal R_{n_1,n_2}}{(n_1!n_2!)^2} =O(e^{-cL}).
\eeqq
\end{cor}

\begin{proof}
It follows from Lemma~\ref{result:U23lessremaineach}, \eqref{eq:triplfactorial}, and Corollary~\ref{result:tPldp}. 
\end{proof}

Noting the formula \eqref{eq:BBPFredholmexpnas} of $F_{\tn{BBP},\ww}(\rr+\ww^2)$, Corollaries~\ref{result:U23leadingseriesh>l},~\ref{result:U23greater_remainder},~\ref{result:U23leadingserieshlessl}, and~\ref{result:U2U3h<l_remainder} imply that 
\beq \label{eq:U23concls}
	\lim_{L\to\infty}
	\prob\left[
		\frac{\LPP\left((\ac xL,\bc yL)-\mr c\ww\mathbf v\,y^{2/3}L^{2/3}\right)-\mv L}{\sdev L^{1/3}}>\rr
		\,\bigg|\,
		\LPP(\ac L,\bc L)>\lv L
	\right]
	=1-F_{\tn{BBP},\ww}(\rr+\ww^2).
\eeq
Thus, Theorem~\ref{result:crossdist}\textnormal{(b)} is proved when $\mv\neq \lv$. 

It remains to consider the case $\mv(x,y)=\lv$. 
In this case, we need to choose \eqref{eq:strategy_order_greater} if $\rr\ge0$, and \eqref{eq:strategy_order_less} if $\rr<0$. 
The analysis in both cases parallels that of the cases $\mv>\lv$ and $\mv<\lv$, but there are some simplifications due to the fact that $y<1$ and $x>1$, which implies the analyticity of $1/\ff_{2,L}$ and $\ff_{2,L}$. 
The limit \eqref{eq:U23concls} holds. 
We skip the details.

%%%%%%%%%%%%%%%%%%%%%%%%%%%%%%%%%%%%%%%%%%%%%%%%
%%%%%%%%%%%%%%%%%%%%%%%%%%%%%%%%%%%%%%%%%%%%%%%%
\section{Conditioning on $\LPP(\ac L,\bc L) = \lv L$}
\label{sec:equalcond}

We briefly explain how the results of this paper extend to conditioning on
$\LPP(\ac L,\bc L)=\lv L$ in place of $\LPP(\ac L,\bc L)>\lv L$.

For an event $E$, we interpret
\beqq
    \prob(E \mid \LPP(M,N)=T)
    =
    \lim_{\epsilon\downarrow 0}
    \frac{
        \prob\bigl(E\cap\{\LPP(M,N)\in(T-\epsilon,T+\epsilon)\}\bigr)
    }{
        \prob\bigl(\LPP(M,N)\in(T-\epsilon,T+\epsilon)\bigr)
    }
    =
    \frac{
        \frac{\partial}{\partial T}
        \prob\bigl(E\cap\{\LPP(M,N)>T\}\bigr)
    }{
        \frac{\partial}{\partial T}
        \prob\bigl(\LPP(M,N)>T\bigr)
    }.
\eeqq
Recall the conventions and notation in Lemma \ref{lem:conditioning_event}. 
It was shown in \cite[Lemma~6.3]{Baik-Cordaro-Tripathi25} that 
\beq
\left.\frac{\partial}{\partial T}
\prob\left(\LPP(\lceil \ac L\rceil,\lceil \bc L\rceil)>T\right)\right|_{T=\lv L}
=
-\frac{\ffz(\tzo^-)}{\ffz(\tzo^+)}
\frac{1}{
    2\pi L(\tzo^+-\tzo^-)
    \sqrt{-\Thetao''(\tzo^-)\Thetao''(\tzo^+)}
}
\bigl(1+o(1)\bigr).
\eeq
Therefore,
\beq
\label{eq:lalbl=relto>}
\left.\frac{\partial}{\partial T}
\prob\left(\LPP(\lceil \ac L\rceil,\lceil \bc L\rceil)>T\right)\right|_{T=\lv L}
=
-(\tzo^+-\tzo^-)
\prob\left(\LPP(\lceil \ac L\rceil,\lceil \bc L\rceil)>\lv L\right)
\bigl(1+o(1)\bigr).
\eeq

For $\bfr=(r_1,\ldots,r_n)$ and $\bfs=(s_1,\ldots,s_n)$ in $\C^n$, define
\beqq
\rS(\bfr\mid\bfs)=\sum_{i=1}^n(r_i-s_i).
\eeqq
Differentiating the formulas in Proposition~\ref{prop:tail} when $m=2$, we have
\beqq
\begin{split}
\frac{\partial}{\partial T_1}
\prob\left(
    \LPP(M_1,N_1)>T_1,\LPP(M_2,N_2)>T_2
\right)
&=
\widehat\QQ_2(\bM,\bN,\bT), \\
\frac{\partial}{\partial T_2}
\prob\left(
    \LPP(M_1,N_1)>T_1,\LPP(M_2,N_2)>T_2
\right)
&=
\widetilde\QQ_2(\bM,\bN,\bT),
\end{split}
\eeqq
where the formulas for $\widehat\QQ_2$ and $\widetilde\QQ_2$ are the same as the formula for $\QQ_2$ in \eqref{def:cQQ}, except that $\nPi_{\bn}(\bsxi,\bseta)$ in \eqref{eq:Pi_n} is replaced by
\beqq
\begin{split}
\widehat\nPi_{\bn}(\bsxi,\bseta)
=
\nPi_{\bn}(\bsxi,\bseta)
\rS(\bsxi^1,\bseta^2\mid\bseta^1,\bsxi^2), \qquad 
\widetilde\nPi_{\bn}(\bsxi,\bseta)
=
\nPi_{\bn}(\bsxi,\bseta)
\rS(\bsxi^2\mid\bseta^2),
\end{split}
\eeqq
respectively.
It is easy to check that each leading contribution acquires an extra factor $-(\tzo^+-\tzo^-)$, which cancels exactly with the corresponding extra factor in \eqref{eq:lalbl=relto>}. Hence, the ratio, as in \eqref{eq:strategy_conditional_probability}, remains unchanged, and therefore the results remain unchanged. The additional $\rS$-factor does not affect the remainder estimates, since on the contour systems used in Sections~\ref{sec:case1}-\ref{sec:U2U3} it is bounded by $C(n_1+n_2)$, which is absorbed into the existing uniform bounds.

Hence, Theorems~\ref{result:CCLT} and~\ref{result:crossdist} remain valid when the conditioning event $\LPP(\ac L,\bc L)>\lv L$ is replaced by $\LPP(\ac L,\bc L)=\lv L$.

%%%%%%%%%%%%%%%%%%%%%%%%%%%%%%
%%%%%%%%%%%%%%%%%%%%%%%%%%%%%%

%%%%%%%%%%%%%%g%%%%%%%%%%%%%%%%%%%%%%%%%%%%%%%%%%%%%%%%%%%%%%%%%%%%%%%%%%
% Bibliography included directly; no external .bib file is required.

\end{document}